\documentclass[11pt]{amsart}

\usepackage[marginratio=1:1,height=670pt,width=480pt,tmargin=70pt]{geometry}

\usepackage{amssymb,latexsym,amsmath,amsthm,amscd,tensor}
\usepackage{mathrsfs}   
  
\usepackage{fancyhdr}

\usepackage{etoolbox}

\DeclareFontFamily{U}{MnSymbolC}{}
\DeclareSymbolFont{MnSyC}{U}{MnSymbolC}{m}{n}
\DeclareFontShape{U}{MnSymbolC}{m}{n}{
    <-6>  MnSymbolC5
   <6-7>  MnSymbolC6
   <7-8>  MnSymbolC7
   <8-9>  MnSymbolC8
   <9-10> MnSymbolC9
  <10-12> MnSymbolC10
  <12->   MnSymbolC12}{}
\DeclareMathSymbol{\im}{\mathbin}{MnSyC}{'270}

\appto\appendix{\addtocontents{toc}{\protect\setcounter{tocdepth}{1}}}
\appto\listoffigures{\addtocontents{lof}{\protect\setcounter{tocdepth}{1}}}
\appto\listoftables{\addtocontents{lot}{\protect\setcounter{tocdepth}{1}}}
\theoremstyle{plain}
\newtheorem{theorem}{Theorem}[section]

\newtheorem{lemma}[theorem]{Lemma}
\newtheorem{proposition}[theorem]{Proposition}

\theoremstyle{remark}
\newtheorem{definition}[theorem]{Definition}

\newtheorem{remark}[theorem]{Remark}

\newtheorem*{example*}{Example}

\numberwithin{equation}{section}

\usepackage{enumitem,array,rotating}  
\usepackage{color}
\usepackage{xcolor}
\definecolor{carmine}{rgb}{0.59, 0.0, 0.09}
\definecolor{mediumpersianblue}{rgb}{0.0, 0.4, 0.65}
\definecolor{persianplum}{rgb}{0.44, 0.11, 0.11}
\usepackage[colorlinks=true,
            linkcolor=persianplum, 
            urlcolor=olive,
            citecolor=mediumpersianblue, 
            backref=page]{hyperref}
\usepackage{amsmath}
\usepackage{amsfonts}
\usepackage{amssymb}
 
 \def\mathbi#1{\textbf{\em #1}}

\newcommand{\cC}{\mathcal{C}}
\newcommand{\cG}{\mathcal{G}}

\newcommand{\cD}{\mathcal{D}}

\newcommand{\cP}{\mathcal{P}}
\newcommand{\cF}{\mathcal{F}}

\newcommand{\cT}{\mathcal{T}}

\newcommand{\cK}{\mathcal{K}}

\newcommand{\X}{\mathcal{X}}
\newcommand{\bS}{\mathbb{S}}
\newcommand{\D}{\mathcal{D}}

\newcommand{\R}{\mathbb{R}}
\newcommand{\RR}{\mathbb{R}}
\newcommand{\NN}{\mathbb{N}}
\newcommand{\CC}{\mathbb{C}}

\newcommand{\PP}{\mathbb{P}}

\newcommand{\const}{\mathrm{const}}

\newcommand{\w}{{\,{\wedge}\;}}
\newcommand{\exd}{\textrm{d}}

\newcommand{\spn}{\operatorname{span}}
\newcommand{\tr}{\operatorname{tr}}

\newcommand{\ad}{\operatorname{ad}}
\newcommand{\Id}{\operatorname{Id}}

\newcommand{\wsf}{\textsc{Wsf }}

\usepackage[small,nohug]{diagrams}        

\diagramstyle[labelstyle=\scriptstyle]

\begin{document}

 \author{Wojciech Kry\'nski \and Omid Makhmali}
 \address{\newline Wojciech Kry\'nski\\\newline
   Institute of Mathematics, Polish Academy of Sciences, ul.~\'Sniadeckich 8, 00-656 Warszawa, Poland\\\newline
   \textit{Email address: }{\href{mailto:krynski@impan.pl}{\texttt{krynski@impan.pl}}}\\\newline\newline
   Omid Makhmali\\\newline
   Institute of Mathematics and Statistics, Masaryk University, Kotl\'a\v rsk\'a  2, 61137 Brno, Czech Republic\\\newline
   \textit{Email address: }{\href{mailto:makhmali@math.muni.cz}{\texttt{makhmali@math.muni.cz}}}
 }

\title[]
          {The Cayley cubic and differential equations}
\date{\today}

\begin{abstract}
  We define Cayley structures  as a field of Cayley's ruled cubic surfaces  over a four dimensional manifold and motivate their study by showing their similarity to indefinite conformal structures and their link to differential equations and the theory of integrable systems. In particular, for Cayley  structures  an extension of certain notions defined for indefinite conformal structures in dimension four are introduced, e.g., half-flatness, existence of a null foliation, ultra-half-flatness, an associated pair of second order ODEs, and a dispersionless Lax pair. After solving the  equivalence problem we obtain the fundamental invariants, find the local generality of several classes of Cayley structures and give examples. 
\end{abstract}

\subjclass{53C10, (53A20, 53A30, 53B15, 53B25,  58A15,  37K10)}
\keywords{Causal geometry, conformal  geometry, path geometry, integrable systems, half-flatness, Lax pair, Cayley's ruled cubic}

\maketitle
  
\vspace{-.5 cm}

\setcounter{tocdepth}{2}
\tableofcontents

\section{Introduction}
\label{sec:introduction}

The main purpose of this article is to demonstrate a link between the theory of integrable systems and a  class of 4-dimensional causal structures which is established via an extension of the  twistor theory.  Roughly speaking, causal structures are defined as the field of light cones arising from the conformal class of a pseudo-Finsler metric. We assume an additional condition of \emph{half-flatness}  for causal structures which is a generalization  of the notion of half-flatness or self-duality for indefinite conformal structures in dimension four. We consider causal structures referred to as \emph{Cayley structures} for which the light cone in each tangent space is  the cone over the projective surface known as   Cayley's ruled cubic.  As will be explained, Cayley structures display the main features of general half-flat causal structures and moreover stand out among them  due to many interesting properties of  Cayley's ruled cubic compared to other  projective surfaces.     The content of the paper is divided into a study of  the geometry of Cayley structures, finding suitable twistorial constructions, and analyzing the corresponding dispersionless  Lax pair system. Our results, when compared with the well-known case of indefinite conformal structures,   establish a rich interplay among twistor theory, integrable systems and half-flat causal structures in dimension four.

To give some perspective, recall that Cayley's ruled cubic is a projective surface in $\PP^3$ expressed as
\begin{equation}
  \label{eq:cayley-cubic-local-coord}
  \textstyle{\frac{1}{3}(y^2)^3+y^0y^3y^3-y^1y^2y^3=0}, 
\end{equation}
where $[y^0:y^1:y^2:y^3]$ are homogeneous coordinates for $\PP^3.$  Among projective surfaces, Cayley's ruled cubic has many remarkable properties, e.g., it is a non-degenerate ruled Demoulin surface that is projectively minimal, whose generating curve can be chosen to be a twisted cubic \cite{Sasaki-Book,Sasaki-Notes}. It is also the unique non-degenerate surface whose Fubini cubic form is parallel with respect to a naturally defined connection \cite{NP-Cayley, NS-Book}.  Another significant  property  is the fact that it decomposes into an open \emph{homogeneous} projective surface, i.e., an open projective surface on which a subgroup of $\mathrm{PGL_4}(\RR)$ acts transitively, and an invariant flag described in \eqref{eq:invflag}.  It turns out that the  \emph{symmetry algebra} of the homogeneous surface is three dimensional which makes it the unique non-degenerate surface with second largest symmetry algebra. More precisely, in \cite{DDKR-HomSurf} it is shown that  the largest dimension for the symmetry algebra of  a projective surface, that is not the quadric, is three. Among these so-called \emph{submaximal} projective surfaces, Cayley's ruled cubic is the only non-degenerate surface.
It is fair to say that the properties of Cayley's ruled cubic makes it the \emph{second most desirable} projective surface after the quadric, which has motivated attempts towards finding its higher dimensional analogues \cite{BNS-affine,EE-Cayley,DV-Cayley}. Although most results on Cayley's cubic are stated in the context of  affine geometry \cite{NS-Book}, they are equally valid in the projective setting \cite{Sasaki-Book, Sasaki-Notes} which is more pertinent to this article.     

Despite all the body of work  on projective surfaces, when it comes to geometries arising from a \emph{field} of projective surfaces on four dimensional manifolds, almost all studies have been focused on  fields of quadrics, which constitute the broad subject of  conformal pseudo-Riemannian geometry.
Recall that a field of quadratic  cones  uniquely identifies  the conformal class of a pseudo-Riemannian metric for which those cones are the null cones (or the light cones). 
   The subject has a long history, going back to Cartan  \cite{Cartan-Conformal1,Cartan-Conformal2} and  Eisenhart \cite{Eisenhart}. It is  not our intention to get any further  into this subject. However, for the purposes of  this paper we highlight the important role that  conformal pseudo-Riemannian structures  play in the theory of integrable systems \cite{FK,DFK,DunKry,MW-Book}, equivalence of distributions \cite{Bryant-36,Nurowski-G2}, and twistor theory \cite{Penrose,FurtherAdvancesII}. 

  An attempt to consider a field of null cones that are not necessarily quadrics can be found in \cite{Segal-Book}, as a Finslerian generalization of conformal geometry which was referred to as a \emph{causal structure} (see Definition \ref{def:causal}). In holomorphic category, causal structures, as a subclasss of \emph{cone structures},   appear naturally on certain class of  \emph{uniruled} complex manifolds were studied in \cite{Hwang-Causal}  with some additional properties (see Remark \ref{rmk:Hwang-causal}). 
  In \cite{Omid-Sigma}, the local equivalence problem for causal structures was  solved using Cartan's method of equivalence. Moreover, in  \cite{Omid-Thesis}, it was shown that,  among four dimensional causal structures whose null cones are not the quadric, the unique causal geometry with \emph{submaximal} symmetry algebra is the \emph{flat  Cayley structure} (see Definition \ref{def:cayley-isotrivial-causal}). This suggests that geometries arising from a field of ruled Cayley cubics, which we will refer to as Cayley structures, might have many desirable properties  observed in indefinite conformal structures. In this paper we show that  half-flat Cayley structures, and more generally, half-flat causal structures,  provide an extension of half-flat conformal structures and its relation to the twistor theory and integrable systems.

 We point out  that in \cite{HS1} causal structures in dimension three were studied via their link to the contact equivalence classes of third order ODEs. Higher dimensional causal structures, also referred to as cone structures, were investigated in \cite{HS2,JS}. In \cite{Hwang-Cone} Hwang studied the local equivalence of geometries arising from a field of certain projective varieties on a uniruled complex manifold where  those varieties coincide with the \emph{variety of minimal rational tangents} (VMRT for short) of the underlying manifold (see \cite{Mok,Hwang-Survey} for a survey).  Finally, we mention that  unlike causal structures, there have been substantial number of works on  geometries arising from a field of   rational normal curves \cite{B-exotic,DT,GN-ODE,Smith,K3,FK-GL2,KM},  referred to as \emph{paraconformal structures} or $\mathrm{GL_2}(\RR)$-\emph{structures}, and Segr\'e varieties \cite{ AG-Book,Baston,BE-Paraconformal,CS-Segre,Mettler-Segre}, referred to by many names, including \emph{Segr\'e structures} and \emph{Grassmann structures}. 
 
 The content and main results of the article are as follows. 

In Section \ref{sec:V-iso-triv}, we define Cayley structures, write the first order structure equations as a $G$-structure and discuss the important notions that will be needed later, such as the ruling planes, the null cones bundle $\cC$, and the  characteristic curves foliating $\cC$ which are the natural generalization of  null geodesics in conformal pseudo-Riemannian geometry. The notion of half-flatness is defined in  Definition \ref{def:half-flatness-}  which will be the main theme of this article. This definition relies on Penrose's observation that self-duality for indefinite conformal structures can be viewed as the existence of a  3-parameter family of null surfaces in the manifold. After obtaining an $\{e\}$-structure for  half-flat Cayley structures and deriving the structure equations, we have the following theorem. 
\newtheorem*{thm0}{\bf Theorem \ref{thm:half-flatness-full-1}}
\begin{thm0}
  \textit{ Given a half-flat Cayley structure, there exists a principal $\mathrm G_\rho$-bundle $\mu:\cP\rightarrow M,$ and an   $\RR^4\rtimes \mathfrak{g}_\rho$-valued 1-form  $\varphi$ which satisfies the structure equations \eqref{eq:str-eqns-cayley} for functions $a_i$'s and $b_{ijk}$'s defined on $\cP$  satisfying \eqref{eq:relations-among-invs}. The vanishing of  $a_1$ and  $b_{123}$  implies that the Cayley structure is flat i.e. $a_i=b_{ijk}=0$ for all $i,j,k.$  Moreover,  $(\cP,M,\varphi)$ is a Cartan geometry  of type $(\RR^4\rtimes \mathrm{G}_\rho,\mathrm{G}_\rho)$  if and only if the Cayley structure is torsion-free  i.e. $a_1=\cdots=a_6=0.$ Torsion-free half-flat Cayley structures locally depend on 6 functions of 2 variables.     }
\end{thm0}

In Section \ref{sec:path-geom-manif} we will review some well-known facts about path geometries. We will make use of Definition \ref{def:2D-path-geom} to define path geometries in subsequent sections. The Cartan connection and fundamental invariants for path geometries in dimension three and two, as expressed in  \eqref{eq:path-geom-cartan-conn-3D}   and \eqref{eq:path-geom-cartan-conn-2D}, will be used in subsequent sections.

In Section \ref{sec:twist-constr} we study path geometries induced on the 3-dimensional  \emph{twistor space} of  half-flat Cayley structures and  characterize them invariantly. We use the 1-form $\varphi$ in Theorem \ref{thm:half-flatness-full-1} to define an $\mathfrak{sl}_4(\RR)$-valued Cartan connection for such path geometries and express their fundamental invariants. 
In particular, we will prove the following.
\theoremstyle{plain}
\newtheorem*{thm1b}{\bf Theorem \ref{thm1a}}
\begin{thm1b}
There is a one to one correspondence between half-flat Cayley structures  and point equivalence classes of  pairs of second order ODEs for which the torsion is of rank one and constant up to  M\"obius transformations.
\end{thm1b}
The  M\"obius transformations appearing above will be explained in Remark \ref{rmk:mobius-trans}. We shall restate the above theorem in terms of explicit differential equations  on the  jet space $J^1(\RR,\RR^2)$ in some local coordinate system.  To our knowledge Theorem \ref{thm1a} is the first instance of characterizing  systems of ODEs for which the torsion  is non-trivial. From this point of view, conditions \eqref{eq:inv-cond-ODE-1} in Theorem \ref{thm1a} shows that half-flat Cayley structures can be thought of, again,  as the \emph{second-simplest} class of causal structures,  after conformal geometry, arising from a pair of second order ODEs. This gives another incentive to study Cayley structures.

The discussion above can be summarized in the following diagram.
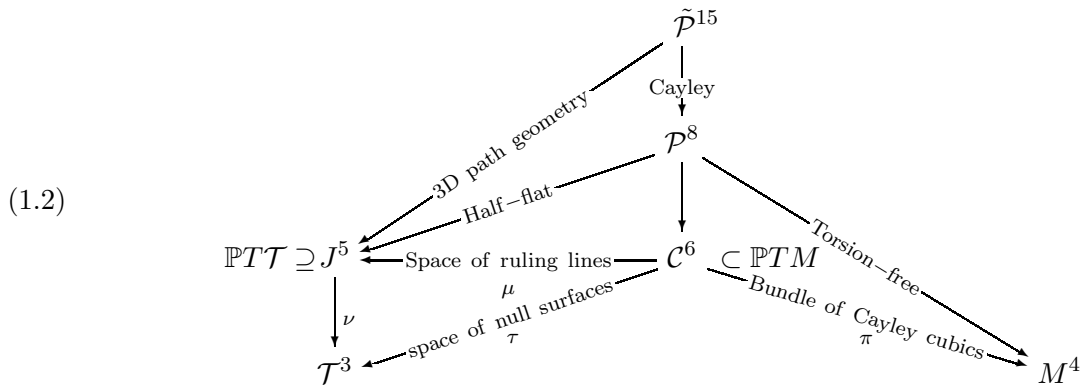
\begin{equation}
  \label{eq:AllFib-FINAL}
\begin{diagram}
             &&&   &&&                 &\ \ \ \tilde\cP^{15}\\
             &&&   &&&\ldTo(6,4)~{\mathrm{3D\ path\ geometry}}       &        \dTo~{\mathrm{Cayley}}         &\\
             &&&   &&&                                          &         \cP^8                 &    \\
             &&&   &&& \ldTo(6,2)~{\mathrm{Half-flat}}                              &                      \dTo&\rdTo(6,4)~{\mathrm{Torsion-free}}       \\
             \PP T\cT\supseteq&J^5&&& \lTo_\mu~{\mathrm{Space\  of\  ruling\  lines}}  &&                                  &\cC^6&\subset \PP T M                  \\
            &\dTo_\nu &&&&&                                          \ldTo(6,2)_\tau~{\mathrm{space\ of\ null\ surfaces}}&&       \rdTo(6,2)_\pi~{\mathrm{Bundle\ of\ Cayley\ cubics}}     \\
             &\cT^3&&&&&                                      &                  &&            &&&& M^4\\
\end{diagram}
\end{equation}
As will be explained in subsequent sections,  $\cC$ is the bundle of projectivized null cones whose fibers are 2-dimensional ruled Cayley surfaces. The 15-dimensional  bundle $\tilde\cP$ is the principal $P_{123}$-bundle over $\cC$ for any 4-dimensional causal structure $(M,\cC),$ where $P_{123}$ is the Borel subgroup of $\mathrm{SL}_4(\mathbb{R})$. Assuming that the cones are ruled Cayley cubics, one can reduce $\tilde\cP$ to the 8-dimensional principal bundle $\cP$ which is a $\mathbb{R}^2$-bundle over $\cC$. In the special case of torsion-free Cayley structures, $\cP$ is  a principal $\mathrm{G}_\rho$-bundle over $M$ equipped with a Cartan connection $(\RR^4\rtimes \mathrm{G}_\rho,\mathrm{G}_\rho)$. Assuming half-flatness, one can locally define the 3-dimensional twistor space $\cT.$ The pair $(J,\cT)$ is equipped with a 3-dimensional path geometry where $J$ is the 5-dimensional space of ruling lines of the bundle of ruled Cayley  cubics. The Cartan connection for this path geometry is defined on the principal $P_{12}$-bundle $(\tilde P,J)$ where $P_{12}\subset \mathrm{SL}_4(\mathbb{R})$ is the parabolic subgroup obtained by crossing the first two nodes in the Dynkin diagram of $A_{3}.$  

Lastly, we mention that the twistor incidence relation remains valid in this setting just as in the case of half-flat conformal structures. Following  \cite{Penrose}, we call two points in $M$ \emph{null separated} if they can be joined by a curve that is the projection of  characteristic curve in $\cC.$ Moreover, since the twistor space, $\cT,$ is the leaf space of the 3-parameter family of surfaces that are null with respect to the ruled Cayley cubic,  a point in $\cT$ corresponds to a null surface on $M$. Also a  point  $p\in M$ corresponds to a curve $L_p\subset \cT$ obtained from the 1-parameter family of null surfaces passing through $p$. The collection of these paths induce a 3D path geometry on $\cT$ (see Section \ref{sec:3d-path-geometry}). Furthermore, since any characteristic curve belongs to a unique null surface, $S_0\subset M$, if two points $p,q\in M$ are null separated, then their corresponding curves $L_p,L_q\subset\cT$ intersect at the point  that corresponds to the surface $S_0.$ As a result, one obtains a twistorial definition of  null separation for two points $p,q\in M$ as the intersection of  their twistor lines $L_p,L_q\subset \cT$ being nonempty. It is straightforward to see that our discussion on  null separation and twistor incidence relation remains valid for the larger class of  4-dimensional half-flat causal structures as developed in \cite{Omid-Thesis}.  
 
Section \ref{sec:addit-geom-struct} involves three different ways through which path geometry on surfaces appears in half-flat Cayley structures. In this section we will make use of  Definition \ref{def:2D-path-geom} to introduce a 2-dimensional  path geometry in terms of a triple $(S,\ell_1,\ell_2)$. First, the natural path geometry induced by the characteristic curves on each leaf of the 3-parameter family of null surfaces in a half-flat Cayley structure is considered. Proposition \ref{prop:flat-path-geom} shows that such  path geometries are flat, generalizing what is known for indefinite half-flat conformal structures \cite{LM-zollfrei,Belgun}.  The second appearance is in half-flat Cayley structures  equipped with an additional foliation by null surfaces. After  giving an invariant condition for the existence of an additional null foliation in a half-flat Cayley structure, we show that the path geometry induced by the characteristic curves on each leaf of this foliation, $(S_{\mathsf{null}},\ell_1,\ell_2),$ defines a projective structure which depends on 1 function of 2 variables. Finally, we find an invariant condition that implies the existence of a path geometry on the 2-dimensional leaf space of the aforementioned null foliation, denoted by $(S_{\mathsf{uhf}},\ell_1,\ell_2)$. We call such class of half-flat Cayley structures  \emph{ultra-half-flat} (see Remark \ref{rmk:ultra-half-flatness}) which are the analogue of  a class of indefinite half-flat conformal structures  studied in \cite{Calderbank, DW1, DW2}. The section ends with the following theorem.
\newtheorem*{thm1}{\bf Theorem \ref{thm:ultra-half-flatness}}
\begin{thm1}
  \textit{ 
 A half-flat Cayley structure satisfying $a_6=a_5=\textstyle{\frac{\partial}{\partial\omega^3}a_2+b_{203}}=0$ in \eqref{eq:str-eqns-cayley} is ultra-half-flat. Locally, such Cayley structures depend on 3 functions of 3 variables. Moreover, such ultra-half-flat Cayley structures for which the path geometries $(S_{\mathsf{null}},\ell_1,\ell_2)$ and $(S_{\mathsf{uhf}},\ell_1,\ell_2)$ are both projectively flat depend on 8 functions of 2 variables. }
\end{thm1}

In Section \ref{sec:examples} we consider a number of examples of  half-flat Cayley structures. We express certain Cayley structures  in terms of a coframe or their corresponding pair of ODEs on their twistor space. Moreover we find all  submaximal models  and introduce them explicitly. 
 
Section ~\ref{sec:disp-lax-pairs} is devoted to a further study of the half-flatness condition for Cayley  structures. Our approach  in this section is motivated by the one used  in \cite{DFK} for   conformal structures and in  \cite{KM} for  $\mathrm{GL}_2(\RR)$-structures. In  Theorems \ref{propNormalform} and \ref{thm2} we prove that there is a one to one correspondence between  half-flat Cayley structures and solutions to a system of seven PDEs for three unknown functions of four variables. Moreover, we prove that the system possesses a dispersionless Lax pair, i.e. it can be put in the Lax form $[L_0,L_1]=0,$
 where $L_0$ and $L_1$ are first order operators depending on the unknown functions and  an auxiliary spectral parameter. Furthermore, we prove the following
 \newtheorem*{thm2}{\bf Theorem \ref{thmCharVar}}
 \begin{thm2}
   \textit{  The characteristic variety for the linearization of the Lax system along a solution is the ruled Cayley cubic in the cotangent bundle of each point, whose dual variety gives a ruled Cayley cubic in each  tangent space using which one can define a Cayley structure  on the  four manifold that corresponds to the given solution of the aforementioned  system of PDEs.}
 \end{thm2}  
 This result can be compared with the result in \cite{FK} where the authors introduce conformal structures as characteristic varieties of certain second order PDEs. Let us point out that as a by-product of our results we provide a general method for studying the principal symbol of differential operators having a dispersionless Lax representations.  Finally, at the end of Section~\ref{sec:disp-lax-pairs}, we will use the Lax pair in order to give an alternative way of viewing  the 2D path geometries discussed in Section \ref{sec:addit-geom-struct}.

The EDS calculations mentioned in the text are carried out using the \texttt{Cartan} package in \texttt{Maple}. 

\addtocontents{toc}{\setcounter{tocdepth}{1}}

\subsection*{Conventions}
In this article we will be working in the real smooth category. However, the results are valid in the complex  setting with minor changes.  Since our results are of local nature, the manifolds  can be taken to be the maximal open sets over which the assumptions made in each statement is valid.

 Throughout this article, given a manifold $N$, its projectivized tangent bundle will be denoted by $\PP TN$ whose fibers are the projectivized tangent spaces, denoted by $\PP T_xN$ where $x\in N.$ 
Also $M$ stands for a four dimensional manifold equipped with a Cayley structure with local coordinate $(x^0,x^1,x^2,x^3).$ We use the coordinates $(y^0,y^1,y^2,y^3)$ as the fiber coordinates for $TM$. The 3-dimensional manifold $\cT$ is referred to as the \emph{twistor space} of the Cayley structure defined in Section \ref{sec:half-flatness-full} with coordinates $(t,z^1,z^2)$. The manifold  $J$ is  five dimensional  defined as the space of the ruling lines of the Cayley structure (see Section \ref{sec:half-flatness-full}).  In Section \ref{sec:3d-path-geometry} it is shown that  $J$ is locally equivalent to an open subset of $\PP T\cT$ that submerses to $\cT$ which can, therefore, be identified with $J^1(\RR,\RR^2),$ with  coordinates $(t,z^1,z^2,p^1,p^2).$ 

We will sometimes work with the cone over a projective hypersurface $V\subset \PP^n,$ which will be  denoted by $\widehat V\subset \RR^{n+1}.$ Moreover, the projective class of a cone over a set $S\subset \RR^{n+1}$ will be denoted by $[S]\subset \PP^n.$

Throughout the article we will make use of the summation convention over repeated indices.

When dealing with symmetric products of 1-forms, such as the bilinear form $g$ in \eqref{eq:conforma-class-g} defined in terms of   $\omega^0,\omega^1,\omega^2,\omega^3,$ by abuse of notation, we will write $g=\omega^0\omega^3-\omega^1\omega^2$ where it is understood that the tensor product is symmetric. Pfaffian systems are introduced using curly brackets, e.g., $I_{\mathsf{cong}}=\{\omega^2,\omega^3\}.$ 

Given a coframe $(\omega^0,\omega^1,\omega^2,\omega^3,\theta^1,\theta^2)$ over $\cC$ introduced in \ref{sec:cayl-isotr-flat}, the differential of a function $F$ defined on $\cC$ will be expressed as
\[\exd F=F_{;i}\omega^i+F_{;4}\theta^1+F_{;5}\theta^2,\] 
where $0\leq i\leq 3$ and the quantities $F_{;\mu},0\leq \mu\leq 5$ are called the coframe derivatives of $F$ with respect to the aforementioned coframe.

\subsection*{Acknowledgments}
We would like to thank  D. Calderbank, M. Dunajski, E. Ferapontov, B. Kruglikov,  P. Nurowski, A. Sergyeyev and  D. The for helpful conversations, suggestions and interest. The starting point of our collaboration occurred during  the Simons Semester \emph{Symmetry and Geometric Structures} at IMPAN in 2018. This work was partially supported by the grant 346300 for IMPAN from the Simons Foundation and the matching 2015-2019 Polish MNiSW fund. OM acknowledges the partial support of  the grant GA19-06357S from  the Czech Science Foundation, GA\v CR. WK acknowledges the partial support of the grant 2019/34/E/ST1/00188 from the National Science Centre, Poland.

\section{Half-flat Cayley structures}
\label{sec:V-iso-triv}

\addtocontents{toc}{\setcounter{tocdepth}{2}}

In this section, we recall the definition of a causal structure, some of its important features and define the notions of half-flatness and isotriviality, initially introduced in \cite{Omid-Thesis}. The section ends with introducing an $\{e\}$-structure for half-flat Cayley structures  which will be crucial for subsequent sections. In this section we assume $0\leq i,j,k\leq 3.$

\subsection{Cayley-isotriviality}
\label{sec:structure-equations}

Following \cite{Omid-Sigma}, a causal structure on an $(n+1)$-dimensional manifold $M$ is defined as follows.
\begin{definition}
\label{def:causal}
  A \textit{causal structure} of signature $(p+1,q+1), p+q=n-1$ on $M^{n+1}$ is given by an immersion  $\iota : \cC\rightarrow\PP TM$ where $\cC$ is a connected, smooth manifold of dimension $2n$ with the property that   the fibration $\mu=\pi\circ\iota : \cC\rightarrow M$ is a submersion whose fibers $\cC_x:=\mu^{-1}(x)$ are mapped, via the immersion $\iota_x : \cC_x\rightarrow \PP T_xM,$ to connected projective hypersurfaces  $\PP T_xM,$  whose projective second fundamental form has signature $(p,q)$.

  Two causal structures $\iota : \cC\rightarrow \PP TM$ and $\iota' : \cC'\rightarrow\PP TM'$ are locally equivalent at   $x\in M, x'\in M',$ if there exists a diffeomorphism $\phi : U\rightarrow U'$   where $x\in U\subset M$ and $x'\in U'\subset M',$ such that $x'=\phi(x)$ and  $\phi_*\left(\iota(\cC)\right)=\iota'(\cC'_{\phi(y)})$ for all $y\in U.$ 

 A causal structure $(\cC^{2n},M^{n+1})$ is  \textit{$V$-isotrivial}, where $V\subset \PP^{n},$ if  $\cC_x$ is  projectively equivalent to $V,$ as a projective hypersurface  in $\PP^n,$ for all $x\in M.$ A causal structure $(\cC,M)$ is called \textit{$V$-isotrivially flat} if it is locally equivalent to $(V\times U,U)$ for some open set $U\subset \RR^{n+1}.$
\end{definition}

\begin{remark}
  Since our consideration in this article is local, we can restrict ourselves to an open set $U\subset\cC$ which can be embedded in $\PP TM$ with the fibers $U\cap \cC_x$ being projective hypersurfaces in $\PP T_xM.$ We mention that the term $V$-isotrivial is borrowed from  the articles of Hwang and Mok on cone structures. See  \cite{Mok,Hwang-Survey} for a  survey. We also point out that classically causal structures are assumed to have closed strictly convex null cones and, therefore, by the above definition have  Lorentzian signature, $(p+1,1)$. 
\end{remark} 
From now on we will work on 4-dimensional manifolds. Using the above definition, we define Cayley structures  as follows.
\begin{definition}\label{def:cayley-isotrivial-causal}
  A \emph{Cayley structure} is a $V$-isotrivial causal structure on an 4-dimensional manifold where $V$ is projectively equivalent to the ruled Cayley cubic surface.  In this case  a  $V$-isotrivially flat causal structure is referred to as  a \emph{flat Cayley structure}. 
\end{definition}

Recall that the ruled Cayley cubic in $\PP^3$ can be expressed as
\begin{equation}
  \label{eq:flat-Cayley-L}
  \textstyle\frac{1}{3}(y^2)^3+y^0y^3y^3-y^1y^2y^3=0,
\end{equation}
where $[y^0:y^1:y^2:y^3]$ are homogeneous coordinates for $\PP^3.$ 
As a result, a Cayley structure on $M$ can be introduced as the conformal class, $[\rho],$ of the symmetric cubic form 
\begin{equation}
  \label{eq:cayely}
\rho={\textstyle\frac{1}{3}}(\omega^2)^3+\omega^0\omega^3\omega^3-\omega^1\omega^2\omega^3
\end{equation}
where $(\omega^0,\omega^1,\omega^2,\omega^3)$   are a coframe on $M.$ One can check that by Definition \ref{def:causal} Cayley structures have signature $(2,2)$ which is expected since Cayley's ruled cubic is an indefinite projective surface.   
 
Let $\mathrm G_\rho\subset \mathrm{GL_4}(\RR)$  be the group of transformations that preserves the conformal class of $\rho.$ It follows that  $\mathrm G_\rho$ can be expressed as the set of matrices 
\begin{equation}
  \label{eq:StrGroup}
    \def\arraystretch{1.15}
  \begin{pmatrix}
    \mathbf{f_0} & \mathbf{u} & \mathbf{v} & \frac{1}{\mathbf{f_0}}\mathbf{\mathbf{uv}}-\frac{1}{3\mathbf{f_0}^2}\mathbf{\mathbf{u}}^3\\
    0 & \mathbf{f_0\,f_{1}}& \mathbf{\mathbf{u}\,f_{1}}  & \mathbf{\mathbf{v}\,f_{1}}\\
    0  & 0& \mathbf{f_0\,f_{1}}^2 & \mathbf{\mathbf{u}\,f_{1}}^2\\
    0 & 0 & 0 & \mathbf{f_0\,f_{1}}^3
  \end{pmatrix}.
\end{equation}
Therefore, as a $G$-structure, the structure group of a Cayley  structure is given by $\mathrm G_\rho$ whose Lie algebra will be denoted by $\mathfrak{g}_\rho.$ Following Cartan's method of equivalence, we will find an $\{e\}$-structure  by constructing a $\mathfrak{g}_\rho$-valued 1-form expressed as
\begin{equation}
  \label{eq:StrAlgebra}
    \def\arraystretch{1.15}
\Omega=\begin{pmatrix}
  \phi_0 & \theta^2 & \theta^1 & 0\\
   0 &  \phi_0+\phi_1& \theta^2 & \theta^1\\
 0 &0 & \phi_0+2\phi_1  & \theta^2\\ 
  0 & 0 & 0 & \phi_0+3\phi_1
\end{pmatrix} 
\end{equation}
on a principal $\mathrm G_\rho$-bundle.  In order to do so, recall that the set  of \emph{adapted} coframes   $(\omega^0,\ldots,\omega^3)^T$ with respect to which  $[\rho]$ is represented as a nonzero multiple of   \eqref{eq:cayely}  is ambiguous up to an action of $\mathrm G_\rho$. As a result, the set of adapted coframes gives rise to a principal $\mathrm G_\rho$-bundle, which will be denoted by $\varsigma:\cP\rightarrow M$ with the right action  defined by
\[R_{g}( \vartheta_p)=g^{-1}\cdot \vartheta_p,\]
where $\vartheta_p\in\varsigma^{-1}(p)$ is a coframe at $p\in M,$    $g\in \mathrm G_{\rho},$  and the action on the right hand side is the ordinary matrix multiplication on  $(\omega^0,\ldots,\omega^3)^T$. 

Restricting to an open set $U\subset M$, on $\cP|_U\cong  U \times \mathrm G_\rho $ one can define a canonical  set of 1-forms  by setting
\[\underline \vartheta(p,g)=(\underline\omega^0,\ldots,\underline\omega^3)^T:= g^{-1}\cdot \vartheta_p,\]
where $\vartheta_p$ is a choice of adapted coframe  $(\omega^0,\ldots,\omega^3)^T$  at $p\in M.$ 

The exterior derivatives of the 1-forms in $\underline \vartheta$ are given by
\begin{equation}
  \label{eq:str-eqns-strart}
  \exd\underline \vartheta=\exd g^{-1} \w \vartheta+g^{-1}\cdot \exd \vartheta=-g^{-1}\cdot\exd g\cdot \w g^{-1} \vartheta+g^{-1}\cdot\exd \vartheta=-\Omega\w \underline \vartheta+ T,
\end{equation}
where $\Omega(p,g)= g^{-1}\cdot\exd g$ is the $\mathfrak g_\rho$-valued 1-form \eqref{eq:StrAlgebra}. The 2-forms $T(p,g)$ can be expressed as
\begin{equation}
  \label{eq:torsionT}
  T(p,g)=\textstyle\frac{1}{2}\underline T^i_{jk}\underline\omega^j\w\underline\omega^k.
\end{equation}

 The 1-form $\Omega$   is ambiguous up to a linear combination of $(\underline\omega^0,\ldots,\underline\omega^3).$ This ambiguity can be resolved by demanding the first order structure equations to have the following form
\begin{subequations}\label{eq:1st-order-str-eqns-cayley}
  \begin{align}
    \exd\underline\omega^0{}={}& -\phi_0  \w\underline\omega^0-{} \theta^2\w\underline\omega^1-\theta^1\w\underline\omega^2,\label{1st-cayley-I}\\
    \exd\underline\omega^1{}={}& -(\phi_0+\phi_1)\w\underline\omega^1 - \theta^2\w\underline \omega^2-\theta^1\w \underline\omega^3,\label{1st-cayley-II}\\
    \exd\underline\omega^2{}={}& -(\phi_0+2\phi_1)\w\underline\omega^2 -\theta^2\w \underline\omega^3+a_7\underline\omega^0\w\underline\omega^1 + a_2\underline\omega^0\w\underline\omega^3,\label{1st-cayley-III}\\
    \exd\underline\omega^3{}={}& -(\phi_0+3\phi_1)\w\underline\omega^3+a_8\underline\omega^0\w\underline\omega^1+a_6\underline\omega^0\w\underline\omega^2+a_4\underline\omega^0\w\underline\omega^3\label{1st-cayley-IV}\\
&+a_5\underline\omega^1\w\underline\omega^2+a_3\underline\omega^1\w\underline\omega^3+a_1\underline\omega^2\w\underline\omega^3\nonumber
  \end{align}
\end{subequations}
for some functions $a_1,\ldots,a_8$ defined on $\cP.$ In the language of Cartan's method of equivalence, the quantities $a_1,\ldots,a_8$ are called \emph{the essential torsion}\footnote{This procedure is known as the \emph{absorption of torsion}. Since in our case, absorbing the torsion is straightforward and  and writing the details is not illuminating, we refer the reader to \cite{IL, Gardner, Olver} for the background on Cartan's method of equivalence.}. For instance, starting with a coframe $(\omega^0,\ldots,\omega^3)$ satisfying
\[\exd\omega^i=\textstyle \frac{1}{2}T^i_{jk}\omega^j\w\omega^k,\]
 and lifting the coframe to $\cP,$  at the point  $(p,g)\in \cP,$ where $g$ is given by \eqref{eq:StrGroup},  one obtains 
\begin{equation}
  \label{eq:a8765}
  a_8={\textstyle\frac{\mathbf{f^2_1}}{\mathbf{f_0}}}\, T^3_{01},\quad a_7={\textstyle\frac{\mathbf{f_1}}{\mathbf{f^2_0}}}\,(\mathbf{f_0}T^2_{01}+\mathbf{u}T^3_{01})\quad a_6={\textstyle\frac{\mathbf{f_1}}{\mathbf{f^2_0}}}\,(\mathbf{f_0}T^3_{02}-\mathbf{u}T^3_{01})\quad
a_5={\textstyle\frac{\mathbf{1}}{\mathbf{f^2_0}}}\,(\mathbf{f_0}T^3_{12}-\mathbf{u}T^3_{02}+\mathbf{v}T^3_{01}).
\end{equation}
which will be used in Section \ref{sec:half-flatness-full}.  Since the structure equations are always written for the lifted coframe $\underline\vartheta$, the underline in  $\underline\omega^i$'s will be dropped from now on.

In each projectivized tangent space $\PP T_xM$, the ruled Cayley cubic $\cC_x$ can be parametrized as  \[v\gamma_1(u)+\gamma_2(u)\] for affine parameters $(u,v),$  where the projective curves $\gamma_1,\gamma_2$ are 
\begin{equation}
  \label{eq:Cayley-parametrized}
  \textstyle{\gamma_1(u)=[\frac{\partial}{\partial\omega^1}-u\frac{\partial}{\partial\omega^0}],\quad \gamma_2(u)=[\frac{\partial}{\partial\omega^3} - u \frac{\partial}{\partial\omega^2} +\frac{1}{3}u^3\frac{\partial}{\partial\omega^0}}].
\end{equation}
In other words, in each tangent space $T_xM,$ the 1-parameter family of 2-planes  given by
\begin{equation}
  \label{eq:RulingFomrs}
  \mathrm{Ker}\textstyle{\{\omega^0+u\,\omega^1-\frac{1}{3}u^3\omega^3,\,\, \omega^2+u\,\omega^3\}}
\end{equation}
are the ruling planes of the cone $\widehat\cC_x.$ One observes that the parameters $u$ and $v$ coincide with the group parameters \textbf{u} and \textbf{v} appearing in  \eqref{eq:StrGroup} where $u$  parametrizes the ruling planes of the Cayley cubic and $v$ is a parameter for the projective line and correspond to the connection forms $\theta^2$ and $\theta^1$  in  \eqref{eq:StrAlgebra} respectively. As a result, the vector fields $\frac{\partial}{\partial \theta^1}$ and $\frac{\partial}{\partial\theta^2}$ are tangent, respectively, to the  ruling lines  and the generating curves of the ruled Cayley  cubic in each projectivized tangent space.

Lastly, note that for each $x\in M$ the flag
\begin{equation}
  \label{eq:invflag}
  \cF_1:=\mathrm{Ker}\{\omega^1,\omega^2,\omega^3\}\subset   \cF_2:=\mathrm{Ker}\{\omega^2,\omega^3\}\subset   \cF_3:=\mathrm{Ker}\{\omega^3\}\subset T_xM
\end{equation}
is invariant under the action of the structure group \eqref{eq:StrGroup}. In fact, using the parametrization \eqref{eq:Cayley-parametrized} the 2-plane $\mathrm{Ker}\{\omega^2,\omega^3\}$ corresponds to $v=\infty.$ This is why, as was mentioned before, Cayley's ruled cubic is an \textit{affine} homogeneous surface since, as a projective surface, the homogeneous part is the open set corresponding to the real values of $u,v.$ The flag \eqref{eq:invflag} will be used further in Section \ref{sec:null-foliation-proj-str}. The reader may consult \cite{Cayley-url} for other interesting aspects of  Cayley's ruled cubic.

 \begin{remark} \label{rmk:Hwang-causal}
   In holomorphic category, causal structures  on \emph{uniruled} complex manifolds were studied in \cite{Hwang-Causal} assuming   that the fibers $\cC_x$ are  smooth at a general point, whose  degree is   at least 3, and arise as the VMRT for  the set of rational curves of minimal degree equal to $n$. Moreover, given a smooth   hypersurface $V\subset\CC\PP^n,$  Hwang  constructed  a $V$-isotrivially flat causal structure. Since the ruled Cayley cubic is not smooth, his construction does not result in a Cayley structure. In \cite{HM-TangentMap}, Hwang and Mok showed that a necessary condition for a causal structure to arise from a VMRT construction is the smoothness of the  normalization of  $\cC_x$ for a general point $x\in M.$ Since the normalization of the ruled Cayley cubic is smooth (\cite{Dolgachev}, Theorem 9.2.1), it would be interesting to know whether a Cayley structure can arise from the VMRT construction.
 \end{remark}

\subsection{The projectivized null cone bundle}
\label{sec:cayl-isotr-flat}
In this section we describe some of the important features of  Cayley structures, which hold, more generally, for all causal structures, that will be used in subsequent sections

Recall from the previous section that the 1-forms $\theta^1$ and $\theta^2$ give a coframe for each fiber $\cC_x\in\PP T_xM$ away from the singular set given by \eqref{eq:invflag}. Hence the 1-forms $(\omega^0,\omega^1,\omega^2,\omega^3,\theta^1,\theta^2)$ introduce a coframe on the 6-dimensional \emph{bundle of projectivized null cones} $\pi\colon\cC\to M$  which are ambiguous up to a multiplication by $\mathbf{f_0,f_1}$ in \eqref{eq:StrGroup}. Restricting to a section $s\colon\cC\to\cP,$ the  1-forms $(\omega^0,\ldots,\omega^3)$ are semi-basic and the 1-forms $(\theta^1,\theta^2)$ are vertical with respect to the fibration $\pi:\cC\rightarrow M.$ Note that the Cayley cubics, $\cC_x,$ are the 2-dimensional integral manifolds  of the Pfaffian system $\{\omega^0,\omega^1,\omega^2,\omega^3\}.$

\begin{remark}\label{rmk:distinction-between-1-forms}
  From now on we may switch viewing $(\omega^0,\omega^1,\omega^2,\omega^3,\theta^1,\theta^2)$ as 1-forms on $\cC$ after restricting to a section $s\colon \cC\to\cP,$ or as 1-forms on the principal $\mathrm{G}_\rho$-bundle $\cP$. The distinction should be clear from the context. 
 \end{remark}
For Cayley structures the conformal class, $[g]\subset\mathrm{Sym}^2(T^*\cC),$ of 
\begin{equation}
  \label{eq:conforma-class-g}
  g=\omega^0\omega^3-\omega^1\omega^2
\end{equation}
 is   invariant on $\cC$ and is degenerate along the vertical directions $\frac{\partial}{\partial\theta^1}$ and $\frac{\partial}{\partial\theta^2}.$ More explicitly, in this case the only group parameters are $\mathbf{f}_0,\mathbf{f}_1$, as expressed in \eqref{eq:StrGroup}, which act by scaling. In fact, $g$ can be viewed as the osculating quadric at each point of  the ruled  Cayley cubics. Note that $[g]$ never induces a conformal structure on $M$ since $\mathcal L_{\frac{\partial}{\partial\theta^2}}g=\omega^2\omega^2,$ where $\mathcal L$ denotes the Lie derivative.

The bundle $\cC$ is endowed with a \emph{quasi-contact} structure (a.k.a an \emph{even-contact} structure)  introduced by $\omega^0,$ i.e., a corank one distribution whose annihilator has Darboux rank five. As a result $\cC$ is  foliated by the \emph{characteristic curves} of this quasi-contact structure which are the integral curves tangent to the degenerate direction of $\omega^0,$ i.e., the line bundle generated by $\ell\in\Gamma(T\cC)$ satisfying 
\begin{equation}
  \label{eq:char-curve} 
\omega^0(\ell)=0,\quad \ell\im\exd\omega^0=0,
\end{equation}
 which implies by \eqref{eq:1st-order-str-eqns-cayley} that $\ell$ is proportional to $\frac{\partial}{\partial\omega^3}.$  In other words, the existence of a foliation by characteristic curves for causal structures is a topological consequence of the non-degeneracy of the projective second fundamental form of the null cones.

More explicitly,  suppose a  causal structure  $(M,\cC)$ is expressed locally as the light cone bundle of a distinguished pseudo-Finsler metric  $F:TM\supset \hat U \rightarrow \RR.$ Using $F,$ locally one can  find a function $L:TM\supset \hat U \rightarrow \RR$ whose degree of homogeneity in $y^i$'s is $r\neq 1$ such that at each point $x\in \pi(\hat U)\subseteq M$ the  set  
 \[L(x^0,x^1,x^2,x^3; y^0,y^1,y^2,y^3)=0,\]
 coincides with  $\widehat\cC_x\cap\hat U$ and moreover the 1-form $$\omega^0=\partial_{y^i}L\,\exd x^i,$$ 
referred to as the \textit{projective Hilbert form}, is well-defined everywhere in $\widehat\cC\cap\hat U.$ To clarify the distinction between $L$ and $F$, we point out that for Cayley structures $L(V)=\rho(V,V,V)$ and  $F(V)=\left(L(V)\right)^{\frac 13}$ where $V\in TM$ and $\rho$ is given in \eqref{eq:cayely}.   We use the term pseudo-Finsler metric since the vertical Hessian of $F^2,$ $[\partial_{y^i}\partial_{y^j}F^2],$  is only assumed to be non-degenerate at every point where $F\neq 0.$ For Cayley structures the vertical Hessian has signature (2,2). Furthermore, the non-degeneracy condition of the second fundamental form of $\cC_x$ implies that   the restriction of the vertical Hessian of $L,$ $[\partial_{y^i}\partial_{y^j}L],$ to  $L=0$ has rank $2.$ As a result, the characteristic  line field on $\cC$ is generated by the vector field 
\[\ell=y^i{\partial_{x^i}}+L^{ij}(\partial_{x^i}L - y^k\partial_{x^k}\partial_{y^i}L)\partial_{y^j},\]
where $[L^{ij}]=[\partial_{y^i}\partial_{y^j}L]^{-1}$. 
Note that if  $L$ is quadratic in $y^i$'s then the expression above   gives the \textit{null geodesic spray} on the projectivized null cone bundle of the pseudo-Riemannian metric $\partial_{y^i}\partial_{y^j}L.$

 \begin{remark}
In \cite{Omid-Sigma}    the equivalence problem of causal structures in any dimension was solved by introducing an $\{e\}$-structure on the bundle of projectivized null cones $\cC$.  It was found that the essential invariants at each point $p=(x,[y])\in \cC$ can be interpreted as  the Fubini cubic form of the fibers $\cC_x$ at the point $[y]$ and the so-called Weyl shadow flag curvature (\wsf curvature for short) which is a generalization of the sectional Weyl curvature in conformal geometry.  If the Fubini form is zero everywhere, then the null cones are quadratic and the causal structure descends to a conformal pseudo-Riemannian structure on $M$. In that case, the \wsf curvature, coincides with the sectional Weyl curvature restricted to a certain   sub-bundle of $T\cC$. 
If the Fubini form is nonzero in an open set, i.e., the null cones are not the quadric in an open set $U\subset \cC$, then one can carry out Cartan's reduction  by normalizing the non-vanishing components of the  Fubini cubic form locally to a constant. This will subsequently enable one to impose a $V$-isotriviality condition. For the sake of brevity, we do not discuss this viewpoint here. 
   Finally,  four dimensional causal structures  can be defined in terms of   rank $3$ distributions $\Sigma_1=\{\omega^0,\omega^1,\omega^2\}^\perp$ with growth vector $(3,5,6)$  containing an integrable distribution $\Sigma_2\subset\Sigma_1$ of corank 1 whose symbol algebra is of   certain form. It turns out that $\cC_x$ are the integral leaves of $\Sigma_2.$ From this point of view, 4-dimensional causal structures  are a special class of \emph{parabolic Monge geometries} of type  $(D_3,P_{123})$ as defined in \cite{ANN}.
 \end{remark}

\subsection{Half-flatness and the structure equations}
\label{sec:half-flatness-full}
  In this section we will define the notion of half-flatness, express the structure equations for a half-flat Cayley structure and state our first theorem.

  \begin{definition}\label{def:half-flatness-}
A Cayley structure is called \emph{half-flat} if there exists a three-parameter family of surfaces in $M$ with the property that through each point $x\in M$ and along each ruling plane of $\widehat \cC_x,$ there passes a unique member of that family that is tangent to that ruling plane.
\end{definition}
The term half-flat is used by analogy with half-flat indefinite conformal structures in dimension four which are characterized by the existence of a three-parameter family of null surfaces that are tangent to the $\alpha$-planes at each point \cite{Penrose}. Recall that the  indefinite quadric, $Q^2\subset\PP^3,$ is doubly ruled, i.e., $Q\cong \PP^1\times\PP^1$. The ruling planes are referred to as $\alpha$-planes and $\beta$-planes at each point. In the case of the ruled Cayley cubic, there is only a single ruling.

In order to express half-flatness as a  differential geometric condition, define $J$ to be the space of ruling planes of the Cayley cubic. It follows that  
\[J\cong \PP^1\times M\subset \mathrm{Gr}_2(TM),\] 
where $\mathrm{Gr}_2(TM)$ is the Grassmannian bundle of 2-planes in $TM$. By our discussion in Section \ref{sec:structure-equations}, $(x^0,x^1,x^2,x^3,u)$  serve as local coordinates for an open subset of $J$ and $(\omega^0,\cdots,\omega^3,\theta^2)$ gives a coframe. By the definition of half-flatness  the lift of the three-parameter family of surfaces foliate $J$ by 2-dimensional leaves. According to Section \ref{sec:structure-equations}, since these leaves arise as tangents to the ruling lines they have to be the integral surfaces of the Pfaffian system
\begin{equation}
  \label{eq:pfaffianI}
  I_{\mathsf{hf}}=\{\omega^0,\omega^2,\theta^2\}.
\end{equation}
In Section \ref{sec:twist-constr} we shall study the leaf space of $I_{\mathsf{hf}}$ and  prove that it is equipped with a path geometry.

By assuming the integrability of $I_{\mathsf{hf}}$ in the structure equations  \eqref{eq:1st-order-str-eqns-cayley}, one obtains the complete structure equations for half-flat Cayley structures to be of the form
\begin{subequations}\label{eq:str-eqns-cayley}
  \begin{align}
    \exd\omega^0{}={}& -\phi_0  \w\omega^0-{} \theta^2\w\omega^1-\theta^1\w\omega^2,\label{StrEq-cayley-I}\\
    \exd\omega^1{}={}& -(\phi_0+\phi_1)\w\omega^1 - \theta^2\w \omega^2-\theta^1\w \omega^3,\label{StrEq-cayley-II}\\
    \exd\omega^2{}={}& -(\phi_0+2\phi_1)\w\omega^2 -\theta^2\w \omega^3+c\, a_6\omega^0\w\omega^1 + a_2\omega^0\w\omega^3,\label{StrEq-cayley-III}\\
    \exd\omega^3{}={}& -(\phi_0+3\phi_1)\w\omega^3+a_6\omega^0\w\omega^2+a_4\omega^0\w\omega^3\label{StrEq-cayley-IV}\\
&+a_5\omega^1\w\omega^2+a_3\omega^1\w\omega^3+a_1\omega^2\w\omega^3\nonumber\\
    \exd\theta^1{}={}&2\phi_1 \w\theta^1+\textstyle{\frac{1}{2}}b_{1ij}\omega^i\w\omega^j + b_{1i4}\omega^i\w \theta^1+ b_{1i5}\omega^i\w \theta^2+b_{145}\theta^1\w\theta^2 \label{StrEq-cayley-V} \\
    \exd\theta^2{}={}&\phi_1\w\theta^2+\textstyle{\frac{1}{2}}b_{2ij}\omega^i\w\omega^j + b_{2i4}\omega^i\w \theta^1+ b_{2i5}\omega^i\w \theta^2+b_{245}\theta^1\w\theta^2\ \label{StrEq-cayley-VI}\\
    \exd\phi_0{}={}&\textstyle{\frac{1}{2}}b_{3ij}\omega^i\w\omega^j + b_{3i4}\omega^i\w \theta^1+ b_{3i5}\omega^i\w \theta^2+b_{345}\theta^1\w\theta^2\label{StrEq-cayley-VII}\\
    \exd\phi_1{}={}&\textstyle{\frac{1}{2}}b_{4ij}\omega^i\w\omega^j + b_{4i4}\omega^i\w \theta^1+ b_{4i5}\omega^i\w \theta^2+b_{445}\theta^1\w\theta^2\label{StrEq-cayley-VIII}
  \end{align}
\end{subequations}
where  $c=2$ or $\frac{3}{2}$ and $b_{\alpha\beta\gamma}$'s, where $0\leq\alpha,\beta,\gamma\leq 5,$ satisfy the relations \eqref{eq:relations-among-invs} in the Appendix.

Note that at the level of first order structure equations \eqref{eq:1st-order-str-eqns-cayley} the half-flatness condition results in the branching
\begin{equation}
  \label{eq:branching-c-StrEq}
  a_8=0,\quad a_7=ca_6,
\end{equation}
where $c=2$ or $\frac{3}{2}.$  In Remark \ref{rmk:LaxPair-branching-involutivity} it will be mentioned that this branching appears in the Lax pair as well.  More explicitly, in terms of the expressions \eqref{eq:a8765}, half-flatness implies $T^3_{01}=0.$ Imposing this condition on a coframe results in the additional condition
\begin{equation}
  \label{eq:TorsionTermsBranching}
  2(T^2_{01})^2- 7 T^2_{01} T^3_{02}+6 (T^3_{02})^2=0
\end{equation}
whose roots are $T^2_{01}={\textstyle\frac{3}{2}}T^3_{02}$ and $T^2_{01}=2T^3_{02}$. Using the expressions of $a_7$ and $a_6$ in \eqref{eq:a8765}, one recovers the branching $a_7=ca_6$. To establish the integrability of the Pfaffian system $I_{\mathsf{hf}}$, one needs to have
\[\exd\theta^2\w \omega^0\w\omega^2=0\] 
which gives a system of 7 first order PDEs in terms of $T^i_{ij}$'s.  Unfortunately, we could not establish  whether any of these branches are empty or contain examples of Cayley structures since our examples in Section \ref{sec:examples} satisfy $a_6=0$. Moreover, neither of the branches results in an involutive system in the sense of Cartan-K\"ahler analysis (see Remark \ref{rmk:LaxPair-branching-involutivity}).

In order to state our first theorem we recall the definition of a Cartan geometry \cite{CS-Parabolic,Sharpe}. 
\begin{definition}\label{def:cartan-conn-def}
  A \emph{Cartan geometry} $(\cG,S,\psi),$ of type $(G,H)$ is a principal $H$-bundle $\cG\to S,$ equipped with a $\mathfrak{g}$-valued 1-form $\psi,$ which is a \emph{Cartan connection}, i.e., 
  \begin{enumerate}
    \item $\psi_u:T_u\cG\rightarrow \mathfrak{g}$ is linear isomorphism for all $u\in \cG.$
    \item $\psi$ is $H$-equivariant, i.e., $R_h^*\psi=\mathrm{Ad}(h^{-1})\circ \psi,$ where $R_h$ denotes the right action by $h\in H.$ 
    \item $\psi(X_v)=v,$  for every fundamental vector field $X_v$ of $\tau:\cG\rightarrow S, v\in \mathfrak h.$
  \end{enumerate}
The curvature of $\psi$ is given by $\Psi=\exd\psi+ \psi\w\psi \in\Omega^2(\cG,\mathfrak{g})$ which is horizontal and defines the \emph{curvature function} $\kappa : \cG\to \bigwedge^2(\mathfrak{g}/\mathfrak{h})^*\otimes\mathfrak{g}.$  
\end{definition}
The following theorem gives the solution of the equivalence problem for  half-flat Cayley structures in terms of an $\{e\}$-structure.
\begin{theorem}\label{thm:half-flatness-full-1}
  Given a half-flat Cayley structure, there exists a principal $\mathrm G_\rho$-bundle $\mu:\cP\rightarrow M,$ and an   $\RR^4\rtimes \mathfrak{g}_\rho$-valued 1-form  
\[\varphi:=
\begin{pmatrix}
  0 & 0\\
  \omega & \Omega
\end{pmatrix}
\]
where $\omega=(\omega^0,\ldots,\omega^3)^T,$ and $\Omega$ is given  in \eqref{eq:StrAlgebra}, such that  $\varphi$  satisfies the structure equations \eqref{eq:str-eqns-cayley} for some functions $a_\mu,$ $ 1\leq\mu\leq 6$ and $b_{\alpha\beta\gamma}$, $0\leq\alpha,\beta,\gamma\leq 5,$ defined on $\cP$  satisfying \eqref{eq:relations-among-invs}. A half-flat Cayley structure is flat if and only if  $a_1$ and  $b_{123}$ vanish which implies that $a_\mu$'s and $b_{\alpha\beta\gamma}$'s vanish and there is  a local coordinate system in which $\omega^i=\exd x^i$ for $0\leq i\leq 3.$   Moreover,  $(\cP,M,\varphi)$ is a Cartan geometry  of type $(\RR^4\rtimes \mathrm{G}_\rho,\mathrm{G}_\rho)$  if and only if the Cayley structure is torsion-free  i.e. $a_1=\cdots=a_6=0$. Torsion-free half-flat Cayley structures locally  depend on 6 functions of 2 variables. 
\end{theorem}
\begin{proof}
 Showing that the vanishing of  $a_1$ and $b_{123}$ imply flatness  is a matter of computation. By setting $a_1,b_{123}$ to zero and checking the identities $\exd^2=0$ for the  structure equations, one obtains that   all the quantities in \eqref{eq:str-eqns-cayley} vanish. In this article we will not write down the differential relations among the quantities in \eqref{eq:str-eqns-cayley} since it takes a lot of space and is not illuminating.  In Remark \ref{rmk:fund-inv} we discuss how to interpret $a_1$ and $b_{123}$ as the fundamental invariants of half-flat Cayley structures. 

As a result, the vanishing of  $a_1$ and $b_{123}$   imply
\[\Phi=\exd \varphi+\varphi\w\varphi=0.\]
 Hence, the  structure equations \eqref{eq:str-eqns-cayley} become the  Maurer-Cartan equations for the Lie algebra $\RR^4\rtimes \mathfrak{g}_\rho$. Moreover, the Cayley structure obtained from the coframe $\omega^i=\exd x^i$ satisfies these Maurer-Cartan equations hence gives the local model for locally flat Cayley structures. Alternatively one can integrate the structure equations for the flat model to obtain this fact as outlined in  Remark  \ref{flat-model-dxi}. 

Using the definition of a Cartan geometry in \ref{def:cartan-conn-def}, it follows that the structure equations \eqref{eq:str-eqns-cayley} does not define a Cartan geometry $(\cP,M,\varphi)$, since the curvature $\Phi$ is not horizontal, as can be seen, for instance, using \eqref{StrEq-cayley-VIII} and \eqref{eq:relations-among-invs}, that $\exd\phi_1$ involves the 2-form $a_1\omega^3\w\theta^2$.

To obtain a Cartan connection from $\varphi$ on the principal bundle $\cP\to M$ one seeks replacements
\[
  \theta^1\longmapsto\theta^1+ q^1_j\omega^j,\qquad \theta^2\longmapsto \theta^2+q^2_j\omega^j,\qquad     \phi_0\longmapsto \phi_0+p_{0i}\omega^j,\qquad \phi_1\longmapsto \phi_1+p_{1i}\omega^i,
\] 
for some functions $q^1_i,q^2_i,p_{0i},p_{1i}$ on $\cP,$ so that  $\Phi$  involves only horizontal 2-forms, i.e. can be expressed in terms of  $\omega^i\w\omega^j$. Using the infinitesimal group action of $\mathrm{G}_\rho$ on $a_i$'s and the non-horizontal  terms in $\Phi$, it follows that the functions  $q^1_i,q^2_i,p_{0i},p_{1i}$ have to be a linear combination of $a_i$'s with constant coefficients. However, a straightforward computation shows that no such linear combination can make $\Phi$ horizontal, unless all $a_i$'s vanish. Such half-flat Cayley structures will be referred  to as \emph{torsion-free} since the first order structure equations for $\exd\omega^i$'s only involve constant coefficients. 
In the language of Cartan \cite{Gardner,Olver},  the coframe on the principal bundle $\cP\to M$ satisfying structure equations \eqref{eq:str-eqns-cayley} defines an \emph{$\{e\}$-structure} which becomes a Cartan connection if the Cayley structure is torsion-free.  Cartan-K\"ahler analysis shows that torsion-free half-flat Cayley structures  generically depend on 6 functions of 2 variables. We will not explain the Cartan-K\"ahler analysis and its implementation in our case since it will take too much space and refer the reader to \cite{Bryant-Notes} for the details.   
\end{proof}

\begin{remark}\label{rmk:fund-inv}
As was mentioned above, the vanishing of $a_1$ and $b_{123}$ imply the flatness of the Cayley structure. However, this does not necessarily mean that they form the set of fundamental invariants. To illustrate what we mean, two viewpoints will be considered. In the first viewpoint,  a Cayley structure is viewed as the Pfaffian system satisfying the structure equations \eqref{eq:str-eqns-cayley} on the 6-dimensional bundle of projectivized null cones, $\cC$, as explained in Section \ref{sec:cayl-isotr-flat}. In this case  the quantities $a_1$ and $b_{123}$ are well-defined relative invariants on $\cC$ and therefore constitute the set of \emph{fundamental} or \emph{essential invariants} for a Cayley structure. Also, it follows from the structure equations \eqref{eq:str-eqns-cayley} that     $\varphi$ defines a   Cartan connection on the principal $\RR^2$-bundle $\cP\to\cC.$ The vanishing of $a_1$ implies that $a_i=0$ for $i=2,\ldots,6$ i.e. the Cayley structure is torsion-free. The vanishing of $b_{123}$ has more severe consequences and leaves out only 3 possibilities which are discussed in Example \ref{sec:proj-struct-cohom}.
  
However, in this paper a Cayley structure is defined as a field of symmetric cubic tensor on a 4-dimensional manifold whose structure group is given by \eqref{eq:StrGroup}. Because of  differential relations
\[\exd a_1\equiv a_1(\phi_0+2\phi_2)-(a_4+4a_5)\theta^1+(4a_3-5a_2)\theta^2,\quad \exd b_{123}\equiv b_{123}(2\phi_0+7\phi_1)+(3b_{103}+b_{112}-a_{1;3})\theta^1+b_{223}\theta^2\]
modulo $\{\omega^0,\omega^1,\omega^2,\omega^3\},$  the quantities $a_1$ and $b_{123}$ are not tensorial invariants of a Cayley structure over a 4-dimensional manifold. From this point of view, there are two fundamental invariants, $I_1$ and $I_2$ whose transformation law is tensorial under the action of structure group. They take value in two modules $\mathfrak m$ and $\mathfrak n$ which are naturally equipped with a filtration arising from the action of the structure group. To obtain $I_1$, one first derives the following differential relations.
 \[\textstyle{ \frac{\partial}{\partial \theta^1}a_1=-a_4-4a_5},\quad \textstyle{ \frac{\partial}{\partial \theta^2}a_1=-5a_2+4a_3},\quad  \textstyle{ \frac{\partial^2} {\partial\theta^1\partial\theta^2}a_1=-\frac 52 a_6},\quad \textstyle{ \frac{\partial^2} {(\partial\theta^2)^2}a_1=9a_4+10a_5},\quad \textstyle{ \frac{\partial^3}{(\partial \theta^2)^3}a_1=-\frac 72 a_6 }.\]
Additionally, one has $\frac{\partial}{\partial \theta^i}a_6=0$ for $i=1,2.$ Hence, the fundamental invariant $I_1$ is 5-dimensional and the module $\mathfrak m,$ in which it  takes value,  is filtered as $\mathfrak{m}^{-1}\subset \mathfrak{m}^{-2}\subset\mathfrak m^{-3}=\mathfrak m$ where  $I_1|_{\mathfrak{m}^{-1}}=\{a_1\}$ and $I_1|_{(\mathfrak m^{-2}\slash\mathfrak m^{-1})}=\{a_5,a_4,4a_3-5a_2\}$ and $I_1|_{\mathfrak m^{-3}\slash\mathfrak m^{-2}}=\{a_6\}.$  Note that the quantities $a_2$ or $a_3$ do not appear in $I_1,$ instead the linear combination $4a_3-5a_2$ does. In fact, setting $a_1=0$ and checking $\exd^2=0$ for the structure equations \eqref{eq:str-eqns-cayley} implies $a_i=0$ for $i=4,5,6$ and $a_3=\frac 54 a_2.$ One needs to consider the higher Bianchi  identities arising from $\exd ^2 b_{ijk}=0$ and $\exd^2a_2=0$ in order to reach at $a_2=0.$ 

The second fundamental invariant, $I_2,$ is 23-dimensional and takes value in the module, $\mathfrak n,$ which has a filtration  of length 6, $\mathfrak n^{-1}\subset\cdots\subset\mathfrak n^{-6}$. We will not discuss this invariant in more detail as they take a lot of space and are not instructive. We only mention  that $I_2|_{\mathfrak n^{-1}}=\{b_{123}\}$ and  $I_2|_{\mathfrak n^{-6}\slash\mathfrak n^{-5}}=\{(a_6)^2\}$ which suggests that     several components of $I_2$ are quadratic expressions involving $a_i$'s. As a result, $I_1$ and $I_2$ are not functionally independent and, as will be discussed in Example \ref{sec:proj-struct-cohom}, the vanishing of $I_2$ which is equivalent to the vanishing of $b_{123},$ leaves only 3 local models of Cayley structures. 

In Section \ref{sec:3d-path-geometry}, we will show that $\varphi$ always results in a $\mathfrak{sl}_4(\RR)$-valued Cartan connection which corresponds to a 3-dimensional path geometry on the \emph{twistor space}, i.e., the leaf space of $I_{\mathsf{hf}}.$ Furthermore, $b_{123}$ generates the Fels curvature of this path geometry.  Lastly, we mention that, in the language of \cite{Omid-Sigma}, $b_{123}$ constitutes the so-called  \wsf curvature of the causal geometry arising from a Cayley structure. More generally,   the notion of half-flatness, as defined in Definition \ref{def:half-flatness-},  can be extended to any four dimensional causal structure of indefinite signature \cite{Omid-Thesis}. It necessarily implies that the fibers $\cC_x$ are ruled projective surfaces and the \wsf  curvature module is 1-dimensional.
\end{remark} 
\begin{remark}\label{flat-model-dxi}
 One can integrate the structure equations for the flat model and obtain a normal coordinate system for flat Cayley structures. This technique will be used in Example \ref{sec:proj-struct-cohom} as well. If all quantities in \eqref{eq:str-eqns-cayley} vanish,  then by Darboux's theorem equations \eqref{StrEq-cayley-VII} and \eqref{StrEq-cayley-VIII} imply that there is a coordinate system in which 
\[\phi_0=\exd \mathbf{s}_0,\qquad \phi_1=\exd\mathbf{s}_1.\]
Furthermore, by Darboux's theorem, equations \eqref{StrEq-cayley-V} and \eqref{StrEq-cayley-VI} imply the existence of coordinates with respect to which one has
\[\theta^1=e^{2\mathbf{s}_1}\exd\mathbf{t}_1, \qquad \theta^2=e^{\mathbf{s}_1}\exd \mathbf{t}_2.\] 
 Analogously, successive application of Darboux's theorem imply that there are coordinates $x^0,x^1,x^2,x^3$ with respect to which
\[
\begin{aligned}
  \omega^0&=e^{-\mathbf{s}_0}(\exd x^0+x^1\exd \mathbf{t}_2+x^2\exd \mathbf{t}_1),\qquad &\omega^1&=e^{-\mathbf{s}_0-\mathbf{s}_1}(\exd x^1+x^2\exd \mathbf{t}_2+x^3\exd \mathbf{t}_1),\\
 \omega^2&=e^{-\mathbf{s}_0-2\mathbf{s}_1}(\exd x^2+x^3\mathbf{t}_2),\qquad &\omega^3&=e^{-\mathbf{s}_0-3\mathbf{s}_1}\exd x^3.
\end{aligned}\]
Consequently, the choice of coframe on $M$ that corresponds to the identity element of the structure group is obtained by setting $\mathbf{s}_0=\mathbf{s}_1=\mathbf{t}_1=\mathbf{t}_2=0$ which gives $\omega^i=\exd x^i$ for $0\leq i\leq 3.$ 
\end{remark} 

\section{A review of path geometries}
\label{sec:path-geom-manif} 
In this section we  recall some of the well-known facts about path geometries that will be needed in   Sections \ref{sec:twist-constr} and  \ref{sec:addit-geom-struct}.

\subsection{Definitions}

Intuitively, a path geometry on an $(n+1)$-dimensional manifold $Q$ is a $2n$-parameter family of paths on $Q$ with the property that along each direction at every point of $Q$ there passes a unique path of that family. Using the natural lift of a path to the projectivized tangent bundle $\PP TQ,$ a path geometry on $Q$ can be expressed as a foliation of the  $\PP TQ$ by curves that are transversal to the fibers $\PP T_xQ.$ It follows that a path geometry in dimension $n+1$ can be understood, locally, as a system of $n$ ODEs of second order
\begin{equation}\label{systemODE}
   z''=F(t,z,z'),\quad t\in\R,\ \ z\in\R^n,
\end{equation}
given up to point transformations of variables $(t,z)$, i.e.,
\[t\mapsto \tilde t=\tilde t(t,z^1,\ldots,z^{n}),\quad z^i\mapsto \tilde z^i=\tilde z^i(t,z^1,\ldots,z^{n}),\quad 1\leq i\leq n.\]
Indeed, assume that on $Q$ a local coordinate system $z=(z^0,\ldots,z^n)$ is given with a choice of parametrization $t\mapsto \gamma(t)=(z^0(t),\ldots,z^n(t))$ for the paths. Consider an open set  $U\subset \PP TQ$ inside which the condition  $\frac{\exd z^0}{\exd t}\neq 0$ is satisfied. Since the paths are given up to a reparametrization, one can assume $z^0=t$ in $U,$ possibly after a change of  coordinates. This implies that we have a $2n$-parameter family of graphs  $t\mapsto \tilde z(t)=(z^1(t),\ldots, z^{n}(t))$ with a property that the  values of $\tilde z$ and $\frac{\exd \tilde z }{\exd t}$ at  $t=t_0$ determine the graphs uniquely.

In \cite{Cartan-proj}, Cartan solved the equivalence problem of path geometries in dimension two. The equivalence problem of  systems of second order ODEs under point transformations was solved in \cite{Chern}, and later in \cite{Fels,Grossman}. To state the obtained results and for the purposes of this article, we will follow    \cite{B-Finsler, Grossman, BGG}   to give a more abstract definition of path geometries.     

\begin{definition}
\label{def:2D-path-geom}
  A  path geometry  is a triple $(S,\ell_1,\ell_2)$ where $S$ is a $(2n+1)$-dimensional manifold  equipped with a  pair of transverse foliations $(\ell_1,\ell_2)$ whose leaves have dimension 1 and $n,$ respectively, such that  the unique rank $(n+1)$  distribution $\mathcal{K}$ tangent to both $\ell_1$ and $\ell_2$  defines a   \emph{multi-contact} structure on $S,$ i.e., one can write  $\mathcal{K}=\{\eta^1,\ldots,\eta^n\}^\perp$ such that  $\ell_1=\{\eta^1,\ldots,\eta^n,\zeta^1,\ldots,\zeta^n\}^\perp,$ $\ell_2=\{\eta^0,\ldots,\eta^n\}^\perp$  and 
\[\exd \eta^i\equiv -\eta^0\w\zeta^i\ \ \mathrm{mod\ \ }\{\eta^1,\ldots,\eta^n\},\]
for all $1\leq i\leq n$.  
\end{definition}

When $n=1,$  $S$ becomes a contact manifold of dimension three with a pair of transverse line bundles $(\ell_1,\ell_2).$  

It is easy to check that Definition ~\ref{def:2D-path-geom} is satisfied when a path geometry on a manifold $Q$ is introduced as  a foliation of $\PP TQ$ by curves that are transversal to the fibers $\PP T_xQ.$ This can be seen by setting $S=\PP TQ,$ with $\ell_1$  being the foliation of $S$ by curves and  $\ell_2$ being the fibers of the projection  $S\rightarrow Q.$ The converse is not necessarily true (see \cite{B-Finsler}). This is  why geometries defined by  Definition \ref{def:2D-path-geom} are sometimes referred to as \emph{generalized path geometries.}   

It can be shown that the local version of Definition \ref{def:2D-path-geom} can be realized as  a \emph{local path geometry}, i.e.,  a foliation of an open set $U\subset \PP TQ,$ with curves that are transversal to the fibers $ \PP T_xQ.$  In other words, restricting to a sufficiently small  neighborhood $U\subset S$ in Definition \ref{def:2D-path-geom},  $U$ can be realized as an open set of $\PP TQ$ for the $(n+1)$-dimensional manifold $Q$   locally defined as the leaf space of $\ell_2.$ Moreover, $\ell_1$ foliates $U\subset\PP TQ$ by curves that are transversal to the fibers of $\PP TQ\rightarrow Q.$

Given a path geometry on $S$ as in Definition \ref{def:2D-path-geom}, one obtains a Cartan geometry $(\cG,S,\psi)$ of type $(\mathrm{SL}_{n+1}(\RR),P_{12})$, where $P_{12}\subset \mathrm{SL}_{n+1}(\RR)$ is the parabolic subgroup preserving a point of the projectivized tangent bundle of $\PP^n$ (see \ref{def:cartan-conn-def} for the definition of a Cartan geometry). Recall that the action of $\mathrm{SL}_{n+1}(\RR)$ on $\PP^{n+1}$ naturally lifts to an action on $\PP T\PP^{n+1}$.

A modern account of deriving the Cartan connection for path geometries on surfaces can be found in \cite{BGH,IL} and in higher dimensions in \cite{Grossman-Thesis}.  For the purposes of this article, from now on we will  only discuss   path geometries in dimensions three and two. 

\subsection{3D path geometries}\label{sec:3d-path-geometries}
To any generalized path geometry on $S$ with $n=2$ one can associate a Cartan geometry $(\cG,S,\psi)$ of type $(\mathrm{SL}_4(\RR),P_{12})$. The Cartan connection is expressed as the following $\mathfrak{sl}_4(\RR)$-valued 1-form
  \begin{equation}
  \label{eq:path-geom-cartan-conn-3D}
  \psi=
  \def\arraystretch{1.3}
\begin{pmatrix}    
  -\frac{3}{4} \psi_0+\frac{1}{4}\psi_1+\frac{1}{4}\psi_2 &-\gamma_1& -\mu_1&-\mu_0\\
-\pi^2 &-\frac{3}{4}\psi_1+ \frac{1}{4}\psi_2+ \frac{1}{4}\psi_0&-\mu_3&-\mu_2\\
\pi^1&  \pi^4   &-\frac{3}{4}\psi_2+\frac{1}{4} \psi_1+\frac{1}{4}\psi_0& \gamma_2\\
\pi^0 & \pi^3 & \pi^5 &  \frac{1}{4}\psi_2+ \frac{1}{4}\psi_1+\frac{1}{4}\psi_0
\end{pmatrix}    
    \end{equation}
    Restricting to any section $s\colon S\rightarrow \cG,$ the leaves of the foliations $\ell_1$ and $\ell_2$ defined in Definition \ref{def:2D-path-geom} coincide with the integral manifolds of $\{\pi^0,\pi^1,\pi^3,\pi^4\}^\perp,$ and $\{\pi^0,\pi^1,\pi^2\}^\perp,$ respectively, and $\{\pi^0,\pi^1\}^\perp$ gives a multi-contact structure on $S$. 
The fundamental invariants of a path geometry in dimension three, as well as higher dimensions, are given by the so-called \emph{Fels invariants}: $\mathbf{T}=(T^i_j)_{1\leq i,j\leq 2}$ and $\mathbf{S}=(S^i_{jkl})_{1\leq i,j,k,l\leq 2}$, satisfying
\begin{equation}
  \label{eq:fels-invariants-path}
 T^i_j=T^j_i,\qquad T^i_i=0,\qquad S^i_{jkl}=S^i_{(jkl)},\qquad S^i_{ijk}=0.
\end{equation}
In the language of parabolic geometry the Fels invariant $\mathbf{T}$ is \emph{torsion} and  $\mathbf{S}$ is \emph{curvature}. The Fels torsion is also referred to as the Wilczy\'nski invariants since it represents the projective invariants of the projective variety that corresponds to the linearization of the system of ODEs at  each solution, as observed in \cite{Doubrov}. Finally, we mention that the Fels torsion can be viewed as a generalization of the W\"unschmann invariant for scalar third order ODEs (c.f. \cite{DT}). 

More explicitly,  given  a pair of ODEs in the form \eqref{systemODE} with $n=2$ its solution curves define a 3D path geometry on the  jet space $J^1(\RR,\RR^2).$ If $(t,z^1,z^2,p^1,p^2)$ denotes a set of local coordinates for $J^1(\RR,\RR^2)$ then 
\begin{equation}
  \label{eq:fels-invariants-explicit}
  T^i_j=F^i_j-\textstyle{\frac{1}{2}\delta^i_jF^k_k},\qquad\qquad S^i_{jkl}=F^i_{jkl}-\textstyle{\frac{3}{4}}F^r_{r(jk}\delta^i_{l)}
\end{equation}
where $1\leq i,j,k,l\leq 2$ and 
\begin{equation}
  \label{eq:fels-torsion-explicit-components}
  F^i_{j}=\textstyle{-\partial_{z^j} F^i+\frac{1}{2}X_F(\partial_{p^j} F^i)-\frac{1}{4}\partial_{p^k} F^i\partial_{p^j} F^k,}\quad F^i_{jkl}=\partial_{p^j}\partial_{p^k}\partial_{p^l}F^i  ,\quad X_F=\partial_t+p^i\partial_{z^i}+F^i\partial_{p^i}.
\end{equation}

If $S^i_{jkl}=0,$ then the path geometry defines a projective structure on the locally defined  leaf space of $\ell_2$ which is 3-dimensional. If $T^i_j=0,$ then the path geometry is said to be torsion-free and defines a half-flat indefinite conformal structure on the leaf space of $\ell_1,$ which is four dimensional \cite{Grossman,CDT}. 

\subsection{2D path geometries}\label{sec:2d-path-geometries}
Similarly, a Cartan connection for a path geometry on a surface can be expressed as an $\mathfrak{sl}_3(\RR)$-valued  1-form expressed as 
  \begin{equation}
  \label{eq:path-geom-cartan-conn-2D}
  \psi=
      \def\arraystretch{1.15}
  \begin{pmatrix}
    -\frac{2}{3}\psi_0-\frac{1}{3}\psi_1 & -\mu_2 & \mu_0\\
        \pi^1 & \frac{1}{3}\psi_0 -\frac{1}{3}\psi_1 & \mu_1\\
        \pi^0 & \pi^2 & \frac{1}{3}\psi_0 +\frac{2}{3}\psi_1
      \end{pmatrix}.
    \end{equation}
Restricting to any section $s\colon S\rightarrow \cG,$ the leaves of the foliations $\ell_1$ and $\ell_2$ coincide with the integral curves of $\{\pi^0,\pi^2\}^\perp,$ and $\{\pi^0,\pi^1\}^\perp,$ respectively, and $\pi^0$ gives a contact form on $S$. Consequently, the curvature 2-form $\Psi=\exd\psi+\psi\w\psi$ is expressed as
     \begin{equation}
            \label{eq:path-geom-curv}
            \Psi=
  \begin{pmatrix}
    0 & K_1 \pi^0 \w\pi^1 & K_2\pi^0 \w\pi^1+L_2\pi^0\w\pi^2\\
 0 & 0 & L_1\pi^0\w\pi^2\\
0 & 0 & 0
  \end{pmatrix}
\end{equation}
for some functions $K_1,K_2,L_1,L_2$ on $\cG.$  The differential relations among  $K_1,K_2,L_1,L_2$ can be written as
\begin{equation}
  \label{eq:Ki-Li-relations}
  \begin{aligned}
    \exd K_1&\equiv (3\psi_0+\psi_1)K_1 -K_2\pi_2,\mathrm{\ mod\ \pi^0,\pi^1}\\
    \exd L_1&\equiv (3\psi_1+\psi_0)L_1 -L_2\pi_1,\mathrm{\ mod\ \pi^0,\pi^2}\\
    \exd K_2&\equiv (3\psi_0+2\psi_1)K_2 -K_1\mu_1+J\pi^2,\mathrm{\ mod\ \pi^0,\pi^1}\\
    \exd L_2&\equiv (3\psi_1+2\psi_0)L_2 -L_1\mu_2+J\pi^1,\mathrm{\ mod\ \pi^0,\pi^2}
  \end{aligned}
\end{equation}
for some function $J$ on $\cG.$ It follows from the relations above that  $K_1,L_1$ are the \emph{fundamental invariants} of a 2D path geometry, i.e., their vanishing implies that $\Psi=0.$  Moreover, it can be shown that $L_1=0$ implies that the path geometry is locally equivalent to a projective structure on the surface $Q$  defined as the leaf space of the foliation  $\ell_2.$   Consequently, the integral curves of $\ell_1$  project to  a 2-parameter family of curves that are the unparameterized geodesics of a linear connection on $Q.$ 

We finish this section by defining the notion of a Weyl connection for a path geometry $(\cG,H,\psi)$ on a surface. Let $G_0\subset \mathrm{SL}_{3}(\RR)$ denote the set of diagonal matrices with $\mathfrak{g}_0$ denoting its Lie algebra, and consider the canonical projection $p_0\colon\cG\to\cG_0$ of the structure bundle $\cG$ to the underlying principal $G_0$-bundle, $\cG_0.$ 
\begin{definition}\label{def:weyl-conn-def}
  A \emph{Weyl structure} is a $G_0$-equivariant  section $s\colon\cG_0\rightarrow \cG$ of $p_0$ to which one can associate the \emph{Weyl connection} given as the $\mathfrak{g}_0$-valued part of $s^*\psi,$ i.e.,  $(s^*\psi_0,s^*\psi_1,s^*\psi_2)$ in \eqref{eq:path-geom-cartan-conn-2D}.
\end{definition}

\section{Path geometry on the twistor space}\label{sec:twist-constr}
In this section we will consider a twistorial construction for half-flat Cayley structures which is an extension of the standard twistor correspondence of Penrose for indefinite half-flat conformal structures. First we  describe the 3-dimensional path geometry induced on the twistor space. This will be followed by an invariant characterization of 3-dimensional path geometries arising from half-flat Cayley structures.   Finally, we will use $\varphi$ from \ref{thm:half-flatness-full-1} to give a Cartan connection for this path geometry and find its fundamental invariants. 
 
\subsection{Induced 3D path geometry}\label{sec:3d-path-geometry}
In this section we  show that the \emph{twistor space} of a half-flat Cayley structure, i.e., the locally defined 3-dimensional space of the integral manifolds of $I_{\mathsf{hf}},$ denoted by $\cT,$ is endowed with a path geometry. 

 Recall from Section \ref{sec:path-geom-manif} that a local path geometry in dimension 3 is expressed as the equivalence class of a pair of second order ODEs
 \begin{equation}
   \label{eq:pair-of-2nd-order-ODEs}
   (z^i)''=F^i(t,z^1,z^2,(z^1)',(z^2)'),\quad i=1,2,
    \end{equation}
 under point transformations. Given a  half-flat Cayley structure, define $J$ to be the locally defined 5-dimensional leaf space of the Pfaffian system $\{\omega^0,\omega^1,\omega^2,\omega^3,\theta^2\}.$ By our discussion in Section \ref{sec:structure-equations}, $J$ is the space of the ruling lines of the Cayley cubics.  Restricting to an open set of $J,$ a point $q\in J$ can be expressed in a local coordinate system $(x^0,\ldots,x^3;u)$ where $u$ is the fiber coordinate for $J\to M,$ and $(x^0,\ldots,x^3)$ are coordinates on $M.$  Let $(t,z^1,z^2)$ be some local coordinates for $\cT.$ One can write 
\begin{subequations}
  \label{eq:pair-of-ODEs-heuristic}
 \begin{align}
  t&= T(x^0,x^1,x^2,x^3;u),\label{eq:pair-of-ODEs-I}\\
    z^1&= Z^1( x^0,x^1,x^2,x^3;u),\label{eq:pair-of-ODEs-II}\\
  z^2&= Z^2(x^0,x^1,x^2,x^3;u).\label{eq:pair-of-ODEs-III}
\end{align}
\end{subequations}
  Locally, it is possible to solve  \eqref{eq:pair-of-ODEs-I} to get $u=u(t; x^i).$ Additionally, in a suitable choice of coordinates, it can be assumed that $\frac{\exd u}{\exd t}=1.$ Replacing $u=u(t;x^i)$ in  \eqref{eq:pair-of-ODEs-II} and \eqref{eq:pair-of-ODEs-III} and differentiating with respect to $t$, one obtains 
\[z^1=Z^1,\quad z^2=Z^2,\quad (z^1)'=\frac{\partial Z^1}{\partial t},\quad (z^2)'=\frac{\partial Z^2}{\partial  t}.\]
By the Implicit Function Theorem, the system above can be solved  to give  \[x^i=x^i(z^1,z^2,(z^1)',(z^2)'),\quad 0\leq i\leq 3.\]
Using these expressions to replace $x^i$'s in $(z^i)''=\frac{\partial^2 Z^i}{(\partial t)^2},1\leq i\leq 2,$ one arrives at a pair of ODEs of the form \eqref{eq:pair-of-2nd-order-ODEs}.

Alternatively, we employ structure equations \eqref{eq:str-eqns-cayley} to show the existence of a path geometry on   $\cT$  as a result of  Definition \ref{def:2D-path-geom}. On the 5-dimensional manifold $J$ define the line bundle $\ell_1=\{\omega^0,\omega^1,\omega^2,\omega^3\}^\perp$ and the rank 2 distribution $\ell_2=\{\omega^0,\omega^2,\theta^2\}^\perp=I_{\mathsf{hf}}^\perp.$ Using the structure equations $(\ell_1,\ell_2)$ define foliations of dimension 1 and 2 on $J$ and  the rank 3 distribution $\cK=\{\omega^0,\omega^2\}^\perp$ induces a multi-contact structure on $J$ since
\[\exd\omega^0\equiv \omega^1\w\theta^2,\quad \exd\omega^2\equiv\omega^3\w\theta^2,\]
modulo $\{\omega^0,\omega^2\}.$ Hence, $(J,\ell_1,\ell_2)$ defines a 3-dimensional local path geometry. Note that the twistor space $\cT$ is locally defined as the leaf of $\ell_2.$  As is mentioned in Section \ref{sec:path-geom-manif}, $J$ can be viewed as the projectivized tangent bundle of the  twistor space.  More precisely, recall  that  a 5-dimensional manifold $J$ is locally equivalent to $\PP TN$ for a 3-dimensional manifold $N,$ if it is  endowed with a rank 3 distribution $\Delta$ with $\mathrm{rank([\Delta,\Delta])}=5$ such that  there exists a corank one distribution $ \cD\subset \Delta$ which is completely integrable. In this case, $N$ is given as the leaf space of $\cD$ (see \cite{Respondek}). In our case, $\cD=I_{\mathsf{hf}}^\perp$ and $\Delta=\{\omega^0,\omega^2\}^\perp.$

\subsection{Invariant characterization}
We have proved that any half-flat Cayley structure defines a 3-dimensional path geometry on its twistor space. Our goal now is to prove Theorem \ref{thm1a} which characterizes the path geometries that can be obtained in this way. The answer will be given in terms of point equivalence classes of systems of pairs of second order ODEs.  Our approach to the problem will follow \cite{K2}.

We are considering  the pair of ODEs given by \eqref{eq:pair-of-2nd-order-ODEs} in  $J=\PP T\cT,$ which is locally equivalent  to $ J^1(\RR,\RR^2)$ - the space of 1-jets of functions $\R\to\R^2$. Consequently,  $(z^1)',(z^2)'$ are replaced by the jet coordinates $p^1,p^2$. It is a direct consequence of the B\"acklund theorem that all geometric information of the system  \eqref{eq:pair-of-2nd-order-ODEs} up to point transformations are encoded in the pair $(\X,\D)$ (c.f. \cite{DKM}), where
\[
\D=\spn\{\partial_{p^1},\partial_{p^2}\},\qquad\mathrm{and}\qquad \X=\spn\{X_F\},
\]
in which $X_F=\partial_t + p^i\partial_{z^i} + F^i\partial_{p^i}$ is the total derivative vector field  appearing earlier in \eqref{eq:fels-torsion-explicit-components}. Note that the pair $(\X,\D)$ coincides with the pair $(\ell_1,\ell_2)$ defining the path geometry in Section \ref{sec:3d-path-geometry}.

Recall from  Section \ref{sec:path-geom-manif} that for 3-dimensional path geometries the Fels torsion, $\mathbf{T}$, which we will refer to simply as  torsion, is given by
\begin{equation}\label{matrixW}
\mathbf{T}=\left(T^i_j\right)_{i,j=1,2}=\mathbf{F}-\textstyle{\frac{1}{2}}(\tr \mathbf{F})\Id,
\end{equation}
where  $\mathbf{F}=\left(F^i_j\right)_{i,j=1,2}$  is expressed in \eqref{eq:fels-torsion-explicit-components}. When considered as a point invariant, the torsion $\mathbf{T}$ is invariant up to a conjugation and a positive factor only. Unlike half-flat indefinite conformal structures that are in one to one correspondence with pairs of ODEs whose torsion is zero, here we  deal with non-vanishing torsion. We will start by defining convenient sections of $\X$ and $\D$ that will be referred to as projective and normal vector fields.

\subsubsection{Projective vector fields and normal frames.} Let $\mathbi{V}=(V_1,V_2)^T$ be a frame for $\D$ and fix a section $X$ of $\X$. Denote by $\ad_X$ the Lie bracket with respect to $X.$ Writing  $\ad_X\mathbi{{V}}=(\ad_XV_1,\ad_XV_2)^T$, the vector field  $X$ is called \emph{projective} if
\begin{equation}\label{eq_a}
\ad^2_X\mathbi{{V}}+\mathbf{T}^X\mathbi{{V}}=0 \mod \X
\end{equation}
for some trace-free matrix $\mathbf{T}^X\in\Gamma(\cD\otimes\cD^*)$, in which case $\mathbi{{V}}$ is called a \emph{normal frame} corresponding to $X$. 
The existence of a projective vector field and a corresponding normal frame is implied by the existence of solutions to a system of ODEs. By direct inspection it can be checked that  the matrix $\mathbf{T}^X$ coincides with the torsion $\mathbf{T}$ of the system up to a conjugation and a positive factor which depend on the choice of the projective vector field $X$ as explained below. 
\begin{proposition}\label{prop:proj-vect-fields}
If both $X$ and $fX$ are projective vector fields and $\emph{\mathbi{V}}$ and $\mathbf{G}\emph{\mathbi{V}}$ are normal frames corresponding to $X$ and $fX$ respectively, where $f$  and $\mathbf{G}$ are scalar and $\mathrm{GL_2(\RR)}$-valued functions on $J^1(\R,\R^2)$, then
\begin{subequations}\label{eq:eq12}
\begin{align}
2fX^2(f)-X(f)^2&=0,\label{eq1}\\
2fX(\mathbf{G})+X(f)\mathbf{G}&=0.\label{eq2}
\end{align}
\end{subequations}
Moreover, one obtains $\mathbf{T}^{fX}=f^2\mathbf{G} \mathbf{T}^X\mathbf{G}^{-1}$.
\end{proposition} 
\begin{proof}
The proof follows from direct computations.  See \cite{K2} for details.
\end{proof}

\subsubsection{Partial connections and Schwarzian derivative.}\label{sec:partial-conn-schw-der}
Fix a projective vector field $X$. We shall denote the second order operator acting on $f$ on the left hand side of  \eqref{eq1} by $\bS^X$, i.e. $\bS^X(f)=2fX^2(f)-X(f)^2$. Moreover, we shall refer to it  as the Schwarzian derivative. Our aim now is to extend its definition to an operator acting on sections of $\D^*\otimes\D$.

Note that in order to find all normal frames corresponding to $X$, one sets $f=1$ in Proposition \ref{prop:proj-vect-fields}. It follows from \eqref{eq2} that  the set of normal frames associated to $X$ is given up to a multiplication by $\mathbf{G}\in\mathrm{ GL}_2(\D)$ satisfying $X(\mathbf{G})=0$. As a result, one can define a \emph{partial connection} on the distribution $\cD,$ i.e., a notion of parallel transport of $\cD$ along integral curves of $X$. Namely,  define  $\nabla^X\colon \X\otimes\Gamma(\D)\to\Gamma(\D)$ by
\[
\nabla^X_Y\mathbi V=0,
\]
where $Y\in\Gamma(\X)$ and $\mathbi V$  is a normal frame  corresponding to $X$. It is straightforward to show that this partial connection is independent of the choice of the normal frame $\mathbi V$.

It follows from \eqref{eq2} that  
\[
\nabla^{fX}=\nabla^X+ \frac{df}{2f}\Id.
\]
Using the natural extension of   $\nabla^X$ to a connection on $\D^*\otimes\D$, we have the following.
\begin{proposition}
Given projective vector fields $X$ and $fX$, and  $\mathbf{A}\in\Gamma(\D^*\otimes\D)$ one obtains 
\[
\nabla^{fX}\mathbf{A}=\nabla^X\mathbf{A}.
\]
\end{proposition}
\begin{proof}
\begin{align*}
(\nabla^{fX}_Y\mathbf{A})(V)&=\nabla^{fX}_Y(\mathbf{A}(V))-\mathbf{A}(\nabla^{fX}_YV)\\
&=\nabla^X_Y(\mathbf{A}(V))+\alpha(Y)\mathbf{A}(V)-\mathbf{A}(\nabla^XV+\alpha(Y)V)\\
&=(\nabla^X_Y\mathbf{A})(V)
\end{align*}
where $\alpha=\frac{\exd f}{2f}$.
\end{proof}
The above proposition  implies that   $\D^*\otimes \D$ is equipped with a connection $\nabla$ induced by  $\nabla^X$ and independent of $X.$ 

Now, using $\nabla$, we are in a position to define a Schwarzian-like derivative
\[
\widehat\bS^X\colon \Gamma(\D^*\otimes\D)\to\Gamma(\mathrm{Sym}^2(\D^*\otimes\D)).
\]
acting as 
\[
\widehat\bS^X(\mathbf{A})=\textstyle{\frac{1}{2}\left(\nabla_X^2\mathbf{A}\otimes \mathbf{A}+\mathbf{A}\otimes\nabla_X^2\mathbf{A}\right)-\frac{5}{4}\nabla_X \mathbf{A}\otimes \nabla_X \mathbf{A}}
\]
where $\mathbf{A}\in\Gamma(\D^*\otimes\D)$. 
\begin{remark}
  Note that the distinction between the  Schwarzian derivatives $\bS^X$ and $\widehat\bS^X$ is that the former acts on functions and the latter acts on $\Gamma(\D^*\otimes\D).$ The following proposition justifies the definition of $\widehat\bS$.
\end{remark}
\begin{proposition}\label{prop-schwarzian}
If $X$ and $fX$ are projective vector fields for a pair of second order ODEs, for some function $f,$  then
\[
\widehat\bS^{fX}(\mathbf{T}^{fX})=f^6(\widehat\bS^X(\mathbf{T}^X)).
\]
Hence, $\widehat\bS^X(\mathbf{T}^X)$ defines a point invariant of the system, given up to a positive factor.
\end{proposition}
\begin{proof}
 Let $X$ be a projective vector field. Direct computation shows that
\[
\widehat\bS^{fX}(\mathbf{T}^{fX})=f^6\widehat\bS^X(\mathbf{T}^X)+f^4\bS^X(f)\mathbf{T}^X\otimes \mathbf{T}^X.
\]
It follows that, if $fX$ is another projective vector field, then \eqref{eq1} implies $\bS^X(f)=0$  and consequently $\widehat\bS^{fX}(\mathbf{T}^{fX})=f^6(\widehat\bS^X(\mathbf{T}^X))$. 
\end{proof}

Proposition \ref{prop-schwarzian} implies that the condition $\widehat\bS^X(\mathbf{T}^X)=0$, where $X$ is a projective vector field, has an invariant meaning. We shall write: $\widehat\bS(\mathbf{T})=0$ for short. Further, if $[\mathbf{A}]$ is a conformal class of a section $\mathbf{A}$ of $\D^*\otimes\D$, (a line subbundle of $\D^*\otimes\D$ spanned by $\mathbf{A}$), then we shall write  $\nabla\mathbf{A}\in[\mathbf{A}]$ meaning that $\nabla$ preserves $\mathbf{A}$ up to a scaling factor.  Now we can state and prove our main result in this section.
\begin{theorem}\label{thm1a}
There is a one to one correspondence between 3-dimensional path geometries arising from  half-flat Cayley structures and point equivalence classes of  systems of pairs  of  second order ODEs \eqref{eq:pair-of-2nd-order-ODEs}   satisfying  
\begin{equation}
  \label{eq:inv-cond-ODE-1}
  \mathrm{rank\,}\mathbf{T}=1,\qquad \nabla \mathbf{T}\in[\mathbf{T}],\qquad \widehat\bS (\mathbf{T})=0.
\end{equation}
More explicitly, given the pair of ODEs \eqref{eq:pair-of-2nd-order-ODEs}, the conditions \eqref{eq:inv-cond-ODE-1} means that there exists a function $\phi$ on $J^1(\R,\R^2)$ such that 
\begin{equation}\label{conditions}
\mathrm{rank\, }\mathbf{T}= 1,\qquad
X_F(\mathbf{T})+\frac{1}{2}[\mathbf{H},\mathbf{T}]=\phi \mathbf{T},\qquad
  X_F(\phi)-\frac{1}{4}\phi^2-2\tr \mathbf{F}=0,
\end{equation}
where $\mathbf{H}=\left( - \partial_{p^i}F^j\right)_{i,j=1,2},$ and $\mathbf{T}$ and $\mathbf{F}$ are given in \eqref{matrixW}.
\end{theorem}
\begin{proof}
Assume that \eqref{eq:inv-cond-ODE-1} holds. Fix $X$ - a projective vector field. Then $\mathrm{rank\,} \mathbf{T}^X=1,$ and it follows that $\mathrm{ker\,}\mathbf{T}^X$ is 1-dimensional. Because $\mathbf{T}^X$ is trace-free, we can choose a frame $\mathbi{U}=(U^1,U^2)^T$ for $\D$ with respect to which  
\begin{equation}\label{eqW}
\mathbf{T}^X=\left(\begin{array}{cc}
0 & 0\\
 w & 0
\end{array}\right)
\end{equation}
for some function $w.$ The frame is not necessarily normal with respect to $X$. However, the second condition in \eqref{eq:inv-cond-ODE-1} means that  $\nabla_X \mathbf{T}^X=\phi\mathbf{T}^X,$ for some function $\phi$. We shall exploit this and show  that $\mathbi{U}$ can be replaced by a normal frame such that $\mathbf{T}^X$ is still in the form \eqref{eqW}. This is essentially due to the fact that  $\nabla^X$  preserves the kernel of $\mathbf{T}^X$. More precisely, define  functions $c_j^i$  by the relation
\[
\nabla^X_XU^i=c^i_1U^1+c^i_2U^2.
\] 
Then, on one hand it follows that
\[
(\nabla_X\mathbf{T}^X)(U^2)=\phi w U^1,
\]
and, on the other hand we have
\[
\begin{aligned}
(\nabla_X\mathbf{T}^X)(U^2)&=\nabla^X_X(\mathbf{T}^X(U^2))-\mathbf{T}^X(\nabla^X_X(U^2))\\
&=w c^1_2 U^2+(X(w)+w(c^1_1-c^2_2))U^1.
\end{aligned}
\]
Comparing the two expressions, one obtains  $c^1_2=0$. Now, in order to find a normal frame, we can look for $V^1=gU^1$ satisfying $\nabla^X_XV^1=0$. Since $\nabla^X_XU^1=c^1_1U^1$ we get that any solution $X(g)=-gc^1_1$ yields a normal vector field $V^1=gU^1$ which is in the kernel of $\mathbf{T}^X$. This vector field, accompanied by any other normal vector field $V^2$, gives a normal frame on $\D$ for which $\mathbf{T}^X$ is in the form \eqref{eqW}, possibly for a different function $w$.

Now, in any normal frame  $\nabla_X(\mathbf{T}^X)$ can be expressed by  applying $X$ to each entry of $\mathbf{T}^X$. Therefore, the condition $\widehat\bS^X(\mathbf{T}^X)=0$ reduces to a scalar equation for $w,$ expressed as
\begin{equation}\label{eqomega}
w X^2(w)-\textstyle\frac{5}{4}X(w)^2=0.
\end{equation}
The equation above is equivalent to $\bS^X(f)=0$ where $w=\frac{1}{f^2},$ for a positive function $f$. By assumption, we know that \eqref{eqomega} is satisfied and we can take $fX$ with $f=w^{-\frac{1}{2}}$ as a new projective vector field. Using Proposition \ref{prop:proj-vect-fields}, one can always find a normal frame corresponding to $fX$ with respect to which 
\begin{equation}
  \label{eq:W-fX}
  \mathbf{T}^{fX}=\left(\begin{array}{cc}
0 & 0\\
 1 & 0
\end{array}\right).
\end{equation}
Furthermore, let $\lambda$ be a natural integral parameter of  $fX$ which implies $\ad_{fX}=\frac{\exd}{\exd\lambda}$. Consequently, the equation \eqref{eq_a} can be expressed as a second order ODE for $\mathbi{V}$ in terms of $\lambda$ with  $\mathbf{T}^{fX}$ given by  \eqref{eq:W-fX}. Solving this pair of linear ODEs, it follows that there exists a normal frame for $\D$ corresponding to $fX$ of the form
\[
V_1=Y_1+\lambda Z_1,\qquad V_2= Y_2+\lambda Z_2 +\frac{\lambda^2}{2}Y_1+\frac{\lambda^3}{6}Z_1 \mod\X,
\]
for some vector fields $Y_1,Y_2,Z_1,Z_2$ such that $\ad_{fX}Y_i=\ad_{fX}Z_i=0\mod\X$.
It follows that
\[
\D=\spn\{V_1,V_2\}=\spn\{Y_1+\lambda Z_1,Y_2+\lambda Z_2 +\frac{\lambda^2}{3}Y_1\}\mod\X.
\]
Let $\mu\colon J^1(\R,\R^2)\to J^1(\R,\R^2)/\X$ be the quotient map to the solution space. Then the collection of the 2-planes $\mu_*\D(\lambda)$, where $\lambda$ parametrizes an integral curve of $fX$, is a ruled Cayley cubic in the  tangent space of the solution space at the point that corresponds to the solution curve $\X$. Indeed, this follows from the parametric form of the Cayley cubic given by \eqref{eq:Cayley-parametrized} with   $\gamma_1=[\mu_*V_1]$,  $\gamma_2=[\mu_*V_2-\frac{1}{2}\mu_*V_2]$ and $u=-\lambda$.

The structure obtained in this way is clearly half-flat as $\D$ is an integrable distribution on $J^1(\R,\R^2)$ and the projection of its leaves results in a 3-parameter family of surfaces satisfying Definition \ref{def:half-flatness-}. 

Conversely, any half-flat Cayley structure gives rise to a path geometry as explained in Section \ref{sec:3d-path-geometry} and using  $\gamma_1 $ and $\gamma_2$ in  \eqref{eq:Cayley-parametrized} one obtains a normal frame that puts the torsion into a form of \eqref{eq:W-fX}.

Finally, the conditions \eqref{conditions} are obtained from \eqref{eq:inv-cond-ODE-1} by direct computations in a standard coordinate system on $J^1(\R,\R^2)$;.
\end{proof}
\begin{remark}\label{rmk:mobius-trans}
The key point in the proof is the  choice of $f$ made for  the vector field $fX$ that puts $\mathbf{T}^{fX}$ in the form \eqref{eq:W-fX}. A priori, one can always do so by a simple rescaling provided that $\mathbf{T}^X$ is in the form \eqref{eqW}. However, in the proof we require that the new $fX$ be a projective vector field, which means that $\bS(f)=0$. This is guaranteed be the third condition in \eqref{eq:inv-cond-ODE-1} which implies that the integral parameters of $X$ and $fX$ are related by a M\"obius transformation $\tilde \lambda=\frac{a\lambda+b}{c\lambda+d}$ for some constants $a,b,c,d$.
 \end{remark}

We believe that the  approach presented  here to characterize ODEs with non-vanishing torsion arising from Cayley structures can be used to treat  other classes of $V$-isotrivial half-flat causal structures. Lastly, we mention that in the language of parabolic geometry \cite{CS-Parabolic} a choice of projective vector field  in Proposition \ref{prop-schwarzian} gives rise to a Weyl connection for the path geometry on $J$ in terms of which the partial connection  in Section \ref{sec:partial-conn-schw-der} can be defined.   

\subsection{Cartan connection}\label{sec:cartan3d}
Using the structure equations \eqref{eq:str-eqns-cayley}, we introduce an $\mathfrak{sl}_4(\RR)$-valued   Cartan connection associated to the 3D path geometry on $J,$ restricted to the 8-dimensional bundle $\cP\to J.$ To do so,  define $\psi$ as in \eqref{eq:path-geom-cartan-conn-3D} where
\begin{equation}\label{eq:sl4-conn-horiz-forms}
\pi^0=\omega^0,\quad\pi^1=\omega^2,\quad\pi^2=\theta^2,\quad \pi^3=\omega^1,\quad\pi^4=\omega^3,\quad\pi^5=\theta^1.
\end{equation}
Due to their length, the expressions of the other  entries of $\psi$  is  provided   in  \eqref{eq:sl4-connection-forms} and \eqref{eq:sl4-conn-forms-c32}.

Using the Cartan connection above, one can find the Fels invariants \eqref{eq:fels-invariants-path} for the path geometry arising from a half-flat Cayley structure. It turns out that the Fels torsion is given by
\begin{equation}
  \label{eq:FelsTorsion-Cayley}
    \mathbf{T}=\left(T^i_j\right)_{i,j=1,2}=
  \begin{pmatrix}
    0 & 0\\
    1 & 0
  \end{pmatrix},
  \end{equation}
which coincides with \eqref{eq:W-fX} as expected. Moreover, the Fels curvature $\mathbf{S}$ is comprised of five components  $W_0,\ldots,W_4$  given by $W_0=b_{123}$ and
\begin{equation}
  \label{eq:FelsCurv-Cayley}
  W_i=\textstyle{\frac{\partial^i}{(\partial \theta^1)^i}b_{123},\qquad 1\leq i\leq 4. }
  \end{equation}
Furthermore, one obtains 
\[
\textstyle{\frac{\partial^5}{(\partial \theta^1)^5}b_{123} =0.}
\]

It follows that, if $b_{123}=0,$ then $\mathbf{S}=0,$ which implies that the 3-dimensional path geometry descends to a projective structure on $\cT$ which will be discussed in Example  \ref{sec:proj-struct-cohom} in detail.  Moreover, if $a_1=b_{123}=0,$ then the resulting projective structure can be shown \cite{Omid-Thesis} to be equivalent to the Egorov projective structure \cite{Egorov} which is the unique submaximal projective structure in dimension three.
More precisely, the Egorov projective structure is locally equivalent to the point equivalence class of  pair of ODEs
\begin{equation}
  \label{eq:Egorov-ODE-flatCayley}
  {z^1}''=z^2,\qquad
{z^2}''=0,
\end{equation}
which is classically associated to the ruled Cayley cubic surface \cite{Sasaki-Notes} with its algebra of point symmetries being isomorphic to $\RR^4\rtimes\mathfrak{g}_\rho.$  

Lastly, we mention that the  $\{e\}$-structure in Theorem  \ref{thm:half-flatness-full-1} defined for a Cayley structure on $M,$ in general, does not transform to a Cartan connection on $J$. This can be shown similarly to the proof of Theorem \ref{thm:half-flatness-full-1}. Also one can find necessary and sufficient conditions  for such Cartan connection to exist on $J$ but we will not discuss them here.   
\section{Appearances of 2D path geometry}
\label{sec:addit-geom-struct}
In this section we show how  path geometry in two dimensions arises in three different ways from  half-flat Cayley structures. This will enable us to define the notion of ultra-half-flatness and find the generality of some  interesting classes of half-flat Cayley structures using the Cartan-K\"ahler machinery.

\subsection{Three parameter family of projectively flat surfaces}
\label{sec:three-param-family}
In this section we  show that the integral manifolds of $I_{\mathsf{hf}}=\{\omega^0,\omega^2,\theta^2\}$ are equipped with a path geometry and that this path geometry is projectively flat. Firstly,
recall that  the integrable Pfaffian system $I_{\mathsf{hf}}$ in \eqref{eq:pfaffianI} gives a foliation of the projectivized null cone bundle $\cC$ by 3-dimensional manifolds. Let $S_{\mathsf{hf}}$ denote a 3-dimensional integral manifold of $I_{\mathsf{hf}}$. The tangent space of each integral manifold is spanned by $\frac{\partial}{\partial\omega^1},\frac{\partial}{\partial\omega^3},\frac{\partial}{\partial\theta^1}$. Hence, each integral manifold projects to a surface in $M.$ Note that these are the 3-parameter family of surfaces coming from Definition \ref{def:half-flatness-}.

As was mentioned in  Section \ref{sec:cayl-isotr-flat},  $\cC$ is foliated by characteristic curves generated by $\frac{\partial}{\partial\omega^3}. $ As a result,  the integral manifolds of $I_{\mathsf{hf}}$ are foliated by such curves as well, resulting in a line bundle $\ell_1$ over each integral manifold. Additionally, any leaf of $I_{\mathsf{hf}}$ is equipped with a projection to $M$ with 1-dimensional fiber, generated by $\frac{\partial}{\partial\theta^1}$, which defines another line bundle $\ell_2$ over each integral manifold. By Definition \ref{def:2D-path-geom}, $(S_{\mathsf{hf}},\ell_1,\ell_2)$ defines a 2D path geometry.

To write down the structure equations for the path geometry of a leaf $S_{\mathsf{hf}}$, we set $\{\omega^0=0,\omega^2=0,\theta^2=0\}.$ The 1-forms $(\omega^1,\omega^3,\theta^1)$ define a coframe on $S_{\mathsf{hf}}.$  It follows that 
\begin{equation}
  \label{eq:alpha-path-1st}
  \begin{aligned}
    \exd\omega^1&=-(\psi_0+\psi_1)\w\omega^1+\omega^3\w\theta^1,\\
    \exd\omega^3&\equiv -\psi_0 \w\omega^3\mathrm{\ \ mod\ \ }\omega^1,\\
     \exd\theta^1&\equiv -\psi_1 \w\theta^1\mathrm{\ \ mod\ \ }\omega^1,\\
  \end{aligned}
\end{equation}
where $\psi_0 \equiv  \phi_0+3\phi_1, \psi_1 \equiv -2\phi_1$ modulo $\{\omega^1,\omega^3,\theta^1\}$.
\begin{remark}\label{rmk:caution-1-forms-spaces}
  As was mentioned in Remark \ref{rmk:distinction-between-1-forms}, one should pay attention to  the spaces over which $\omega^0,\omega^2$ and $\theta^2$ are defined. In our description above, we treated $(\omega^0,\omega^2,\theta^2)$ as 1-forms over $\cC,$ although, to be more precise, one first takes a section $s:\cC\rightarrow \cP$ and then considers $s^* I_{\mathsf{hf}}.$ To avoid such technical issues we take another point of view which will be helpful in the next two sections as well.

  The 1-forms $(\omega^0,\omega^2,\theta^2)$ are defined over $\cP.$ Hence, the integral manifolds of $I_{\mathsf{hf}}$ are 5-dimensional submanifolds of $\cP$ endowed with a coframe $(\omega^1,\omega^3,\theta^1,\phi_0,\phi_1).$ Restricting to one of such integral manifolds, it follows that the Pfaffian system $I_1:=\{\omega^1,\omega^3,\theta^1\}$ is integrable   with 2-dimensional leaves. Define  $S_{\mathsf{hf}}$ to be the locally defined leaf space of $I_1$. The 1-forms $(\omega^1,\omega^3,\theta^1)$ give a coframe on $S_{\mathsf{hf}}.$ Moreover, by equations \eqref{eq:alpha-path-1st} one  defines $\ell_1$ and $\ell_2$ to be the foliations  $\{\omega^1,\theta^1\}^\perp$ and $\{\omega^1,\omega^3\}^\perp$, respectively. By Definition \ref{def:2D-path-geom}, $(S_{\mathsf{hf}},\ell_1,\ell_2)$ defines a path geometry. Finally, note that $S_{\mathsf{hf}}$ can be identified with an integral manifold of $I_{\mathsf{hf}}$ as a Pfaffian system on $\cC.$
 \end{remark}

Inspecting equations \eqref{eq:alpha-path-1st}, it follows that $(\omega^1,\omega^3,\theta^1,\psi_0,\psi_0)$ defines a Weyl connection, defined in \ref{def:weyl-conn-def}, for a 2D path geometry after absorption of inessential torsion terms.  More explicitly, one obtains an  $\mathfrak{sl}_3(\RR)$-valued 1-form $\psi_{\mathsf{hf}},$ as in  \eqref{eq:path-geom-cartan-conn-2D}, by setting 
\[
\begin{gathered}
\pi^0=\omega^1,\quad \pi^1=\omega^3,\quad \pi^2=\theta^1\\
  \psi_0 = -a_3\omega^1+\phi_0+3\phi_1,\quad  \psi_1 = -a_2\omega^1-2\phi_1,\quad\mu_1= 2a_6\omega^1\\
\textstyle{\mu_2 =-b_{212}\omega^1+b_{223}\omega^3-(2a_3+a_2)\theta^1},\qquad 
\mu_0=\textstyle{-2a_{6;3}\omega^1-(b_{102}+a_{2;3})\omega^3-2a_6\theta^1}.
\end{gathered}
\]
when $c=2$ in \eqref{eq:branching-c-StrEq}. For the case $c=\frac{3}{2}$ an associated Weyl connection differs from the one above by
\[\textstyle{\mu_1=\frac{3}{2}a_6\omega^1},\qquad 
\mu_0=\textstyle{-2a_{6;3}\omega^1-(b_{102}+a_{2;3})\omega^3-\frac{3}{2}a_6\theta^1}.\]
It is a matter of computation to show 
\[\exd\psi_{\mathsf{hf}}+\psi_{\mathsf{hf}}\w\psi_{\mathsf{hf}}=0,\] 
which implies that the induced path geometry on the leaves of $I_{\mathsf{hf}}$ is projectively flat. Hence, one obtains the following.
\begin{proposition}\label{prop:flat-path-geom}
Given a half-flat Cayley structure, each member of the 3-parameter family of surfaces in Definition \ref{def:half-flatness-} is equipped with a projectively flat connection. 
\end{proposition}
\begin{remark}
  More generally, it can be shown that the 3-parameter family of surfaces for any half-flat causal structures are equipped with a flat projective structure \cite{Omid-Thesis}.  The conformal version of this fact, i.e., the 3-parameter family of null surfaces for any half-flat conformal structure of split signature are equipped with a flat projective structure, has been shown in \cite{LM-zollfrei,Belgun}. 
\end{remark}

\subsection{A null foliation and 2D projective structures}
\label{sec:null-foliation-proj-str}

Recall from \eqref{eq:invflag} that the distribution $\cF_2=\mathrm{Ker}\{\omega^2,\omega^3\}$ is invariantly defined for any Cayley structure. 
Inspecting the structure equations \eqref{eq:str-eqns-cayley}, it follows that the condition 
\[a_6=0\]
implies  the integrability of  the  Pfaffian system
\[I_{\mathsf{cong}}=\{\omega^2,\omega^3\}.\]
Moreover, using   \eqref{eq:cayely}, it follows that the 2-distribution $\cF_2$ is null with respect to the cubic form $\rho.$ By analogy with conformal structures, we say that the integrability of $I_{\mathsf{cong}}$  gives a foliation (or a congruence) of $M$  by \emph{null surfaces}.

\begin{remark}\label{rmk:null-foliation-2d-caution}
  Here again one has to be careful with the spaces over which the 1-forms $(\omega^2,\omega^3)$  are defined. If they are considered as 1-forms over $\cP$ then $I_{\mathsf{cong}}$ foliates $\cP$ by 6-dimensional leaves which project to a foliation of $M,$ by 2-dimensional null surfaces mentioned above. Alternatively, one can take a section  $s:M\rightarrow \cP$ and consider the foliation induced by $s^*I_{\mathsf{cong}}.$ Note that because  of the form of the structure group given by \eqref{eq:StrGroup}, the integrability of  $I_{\mathsf{cong}}$ is independent of the section $s.$
\end{remark}

Similar to the previous section, it can be shown that the  integral manifolds of $I_{\mathsf{cong}}$ are equipped with a path geometry. Following Remark \ref{rmk:caution-1-forms-spaces}, setting $\omega^2=0,\omega^3=0,$ one restricts to a 6-dimensional submanifold  of $ \cP$ over which  the Pfaffian system
\[I_{\mathsf{null}}=\{\omega^0,\omega^1,\theta^2\}\]
is integrable. Denoting the locally defined leaf space of $I_{\mathsf{null}}$ by $S_{\mathsf{null}},$ it follows  that $S_{\mathsf{null}}$ is 3-dimensional and $(\omega^0,\omega^1,\theta^2)$ gives a coframe on it satisfying
\begin{equation}
  \label{eq:null-Folia-path-geom-1st}
  \begin{aligned}
        \exd\omega^0&=-(\psi_0+\psi_1)\w\omega^0+\omega^1\w\theta^2,\\
    \exd\omega^1&\equiv -\psi_0 \w\omega^1\mathrm{\ \ mod\ \ }\omega^0,\\
     \exd\theta^2&\equiv -\psi_1 \w\theta^2\mathrm{\ \ mod\ \ }\omega^0,\\
  \end{aligned}
\end{equation}
where $\psi_0 \equiv \phi_0+\phi_1, \psi_1 \equiv -\phi_1$ modulo $\{\omega^0,\omega^1,\theta^2\}$.

Defining  $\ell_1$ and $\ell_2$ to be $\{\omega^0,\omega^1\}^\perp$ and $\{\omega^0,\theta^2\}^\perp$, respectively, Definition \ref{def:2D-path-geom} implies that $(S_{\mathsf{null}},\ell_1,\ell_2)$ is a path geometry. 

 \begin{proposition}\label{prop:null-foliation-2d}
 The path geometry induced on  3-folds $S_{\mathsf{null}}$ locally arises from a projective structure on the integral surfaces of $I_{\mathsf{cong}}$.
\end{proposition}
\begin{proof}\label{pf:null-foliation-2d}
  Firstly, we find the curvature of the path geometry $(S_{\mathsf{null}},\ell_1,\ell_2).$ In order to do so, after absorption of inessential torsion terms, one obtains an  $\mathfrak{sl}_3(\RR)$-valued 1-form $\psi_{\mathsf{null}}$   defined as in \eqref{eq:path-geom-cartan-conn-2D}, where
\[
\begin{gathered}
\pi^0=\omega^0,\quad \pi^1=\omega^1,\quad \pi^2=\theta^2,\\
    \psi_0 =(a_2-a_3)\omega^1+\phi_0+\phi_1,\quad\psi_1 = (a_4+a_5)\omega^0+(a_3-a_2)\omega^1-\phi_1,\quad\mu_1=0\\
\textstyle{\mu_2 =(-2 a_2 a_4-\frac 7 3  a_2 a_5+\frac 43a_3a_4+\frac{4}{3} a_3 a_5+\frac{2}{3} a_{4;1}-a_{5;1}) \omega^0+b_{201} \omega^1+(a_4+a_5)\theta^2},\\
\mu_0=\textstyle{(a_2a_4+\frac 53 a_2a_5-\frac 23 a_3a_4-\frac 23a_3a_5+b_{401}+\frac 56 a_{4;1}+a_{5;1})\omega^0}.
\end{gathered}
\]
The curvature 
\[\Psi_{\mathsf{null}}=\exd \psi_{\mathsf{null}}+\psi_{\mathsf{null}}\w\psi_{\mathsf{null}}\] 
takes the form \eqref{eq:path-geom-curv}, where $L_1=0$ and 
\[
\begin{aligned}
K_1&=-\textstyle{\frac{41} 2 a_2a_3a_4-\frac{133}6 a_2a_3a_5 -\frac 56 b_{201;0}+ \frac 16 b_{401;1}-\frac 54b_{402;0} +6a_2a_{4;1}-\frac 14 a_4b_{402}+ \frac{17}6 a_2b_{401}}\\
&\ \ \textstyle{+\frac{5}3 b_{201}a_5
+ \frac{22}3 a_2 a_{5;1}- \frac 73 a_3b_{401}+\frac{53} 3 a_2^2a_5+\frac 74 a_4b_{201}-4a_4a_{5;2}-\frac 14a_5b_{402}}\\
&\ \ \textstyle{-4a_5a_{4;2}-4a_4a_{4;2}-4a_5a_{5;2} -\frac{53}{12}a_3a_{4;1}+15a_2^2a_4+7a_3^2a_4+7a_3^2a_5-\frac{29}6a_3a_{5;1}}\\
\end{aligned}
\]
Since $L_1=0$, it follows that the path geometry locally arises from a projective structure on  the leaf space of the foliation  $\ell_2.$  Such surfaces are integral surfaces of $I_{\mathsf{cong}}$ discussed in Remark \ref{rmk:null-foliation-2d-caution}\footnote{Since  the integral manifolds of $I_{\mathsf{cong}}$ are referred to as null surfaces of $M,$ we used the subscript $\mathsf{null}$ for the 3-fold $S.$ }.
\end{proof}

\begin{remark}
  Using Cartan-K\"ahler analysis, it follows that projective structures arising from the construction above on the leaves of  $I_{\mathsf{cong}}$ locally depend on 1 function of 2 variables.

  It turns out that if the half-flat Cayley structure is torsion-free, i.e., $a_1=0,$ then the projective structure on the leaves of the null foliation  $I_{\mathsf{cong}}$ is flat.
\end{remark}

\subsection{Ultra-half-flatness and 2D path geometries}
\label{sec:ultra-half-flatness-1}

There is yet another appearance of 2-dimensional path geometry. As discussed in the previous section, the condition $a_6=0,$ implies that the Pfaffian system $I_{\mathsf{cong}}$ is integrable. Consider the 2-dimensional  leaf space, $Q,$ of  the induced foliation on $M$.
The 1-forms $(\omega^2,\omega^3)$ give  a coframe on $Q.$ We would like to investigate the conditions that imply the existence of a path geometry on $Q.$ Motivated by \cite{Calderbank}, we  give the following definition
\begin{definition}
  A half-flat Cayley structure is called \emph{ultra-half-flat} if it is equipped with a foliation by null surfaces whose locally defined space of leaves has a path geometry.
\end{definition}
\begin{remark}\label{rmk:ultra-half-flatness}
  Recall from Section  \ref{sec:3d-path-geometry} that  Cayley structures that are half-flat correspond to a class of  3D path geometries on the leaf space of the Pfaffian system $I_{\mathsf{hf}}.$  The term ultra-flatness is used to highlight the association of an additional 2D path geometry.
\end{remark}
In order to find conditions that ensure  Definition \ref{def:2D-path-geom} is satisfied for some 3-fold projecting to $Q$,  we make the replacement 
\begin{equation}
  \label{eq:gamma2-theta2}
  \gamma^2=\theta^2-a_2 \omega^0.
\end{equation}
As a result, the 1-forms $(\omega^2,\omega^3,\gamma^2)$ satisfy 
\[
\begin{aligned}
  \exd\omega^2&=-(\psi_0+\psi_1)\w\omega^2+\omega^3\w\gamma^2,\\
  \exd\omega^3&\equiv -\psi_0\w\omega^3\quad \mathrm{mod\ \ }\omega^2,\\
    \exd\gamma^2&\equiv -\psi_1\w\gamma^2+\textstyle{(\frac{\partial }{\partial\omega^3}a_2+b_{203})\omega^0\w\omega^3} \quad \mathrm{mod\ \ }\omega^2,\\
\end{aligned}
\]
where $\psi_0\equiv \phi_0+3 \phi_1,\psi_1\equiv -\phi_1$ modulo $\{\omega^0,\omega^1,\omega^2,\omega^3,\theta^1,\gamma^2\}.$ 

Assuming 
\begin{equation}
  \label{eq:cond-path-cN}
\textstyle{\frac{\partial}{\partial\omega^3}a_2}+b_{203}=0,
\end{equation}
it follows that the Pfaffian system
\[I_{\mathsf{uhf}}=\{\omega^2,\omega^3,\gamma^2\}\]
is integrable. Moreover, its 3-dimensional space of leaves, $S_{\mathsf{uhf}},$ is endowed with a path geometry, where $\ell_1,\ell_2$ correspond to the foliations  $\{\omega^2,\gamma^2\}^\perp$ and $\{\omega^2,\omega^3\}^\perp,$ respectively.

Imposing the condition $a_6=\textstyle{\frac{\partial}{\partial\omega^3}a_2}+b_{203}=0$ and checking the relations arising from  $\exd^2=0,$ one obtains a branching based on the  two possibilities
\begin{subequations}
\label{eq:cond-path-ais}  
  \begin{align}
  a_5&=0, \label{eq:cond-path-ais-I}\\ 
 a_5&= \textstyle{-\frac{1}{2}a_4}. \label{eq:cond-path-ais-II}
\end{align}
\end{subequations}

Let us assume that the conditions $a_6=a_5=\textstyle{\frac{\partial}{\partial\omega^3}a_2+b_{203}}=0$ are satisfied. As a result,  the path geometry $(S_{\mathsf{uhf}},\ell_1,\ell_2)$ carries a $\mathfrak{sl}_3(\RR)$-valued 1-form $\psi_{\mathsf{uhf}},$ as in \eqref{eq:path-geom-cartan-conn-2D}, where
\[
\begin{gathered}
  \pi^0=\omega^2,\quad \pi^1=\omega^3,\quad \pi^2=\gamma^2,\\
  \psi_0= -a_4\omega^0-a_3\omega^1-a_1\omega^2+\phi_0+3\phi_1,\quad \psi_1= a_4\omega^0+a_3\omega^1-\phi_1\\
  \begin{aligned}
    \mu_2&= \textstyle{a_1\gamma^2+(2a_1a_2-b_{101}-2b_{403})\omega^0-b_{212}\omega^1 +(-\frac{1}{3}a_{1;3} +\frac{4}{3}b_{112}+b_{103})\omega^2}\\
   &\phantom{=} +b_{223}\omega^3-(2a_3+a_2)\theta^1,
  \end{aligned}\\
  \mu_1= \textstyle{\frac{2}{3}(a_3-a_2)\omega^2,\quad \mu_0= \frac{2}{3}(a_2-a_3)\gamma^2 -(\frac{1}{3}b_{212}+b_{102})\omega^2+\frac{1}{3}(b_{112}-a_{1;3})\omega^3}
\end{gathered}
\]
The curvature $\Psi_{\mathsf{uhf}}=\exd \psi_{\mathsf{uhf}}+\psi_{\mathsf{uhf}}\w \psi_{\mathsf{uhf}}$ is expressed as \eqref{eq:path-geom-curv}, where
\begin{equation}
  \label{eq:K1-L1}
  \begin{aligned}
    K_1&=\textstyle{4a_1b_{223}-3a_3b_{123}+a_2b_{123}+b_{223;2}+\frac{4}{3}b_{103;3}-\frac{5}{3}b_{123;1}-\frac 13b_{223;35}}\\
    L_1&=\textstyle{-\frac{4}{3}a_4}.
  \end{aligned}
\end{equation}

Carrying out necessary Cartan-K\"ahler analysis, one  arrives at the following  theorem.

\begin{theorem}\label{thm:ultra-half-flatness}
  A half-flat Cayley structure satisfying $a_6=a_5=\textstyle{\frac{\partial}{\partial\omega^3}a_2+b_{203}}=0$ in \eqref{eq:str-eqns-cayley} is ultra-half-flat. Locally, such Cayley structures depend on 3 functions of 3 variables. Moreover, such ultra-half-flat Cayley structures for which the path geometries $(S_{\mathsf{null}},\ell_1,\ell_2)$ and $(S_{\mathsf{uhf}},\ell_1,\ell_2)$ are both projectively flat depend on 8 functions of 2 variables.  
\end{theorem}
\begin{remark}
In the case $a_6=\textstyle{\frac{\partial}{\partial\omega^3}a_2+b_{203}}=0$ and $a_5=-\frac{1}{2}a_4$ the quotient space is equipped with a projective structure, i.e., $L_1=0$  and the invariant $K_1$ is identical to $K_1$ in \eqref{eq:K1-L1}. Moreover, such Cayley structures  depend on 4 functions of 3 variables. 
\end{remark}

\begin{remark}
  For a half-flat Cayley structure that is torsion-free, i.e., $a_1=0,$  the obstruction for ultra-half-flatness is $b_{203}=0,$ in which case the quotient space is endowed with a projective structure. Torsion-free ultra-half-flat Cayley structures  locally depend on 4 functions of 2 variables. Finally, one obtains that  torsion-free ultra-half-flat Cayley structures for which  the path geometry $(S_{\mathsf{uhf}},\ell_1,\ell_2)$ is projectively flat, locally depend on 2 functions of 2 variables.
\end{remark}
\begin{remark}
Using the  invariant flag \eqref{eq:invflag} and the structure equations \eqref{eq:str-eqns-cayley}, one obtains that  the distribution $\cF_3$ is integrable if $a_6=a_5=0.$ Consequently, one can determine   differential conditions under which the 3-dimensional leaves $\cF_3$ or the 3-dimensional leaf space of $\cF_1$ are equipped with certain geometric structure. It turns out that under certain conditions  these 3-dimensional spaces can be equipped with 3D path geometries. However, the conditions are rather laborious to obtain and more restrictive than ultra-half-flatness. We decided to skip discussing these additional considerations. 
\end{remark}

\section{Examples}\label{sec:examples}
In this section we present examples illustrating some of the results obtained so far.  
 
\subsection{Submaximal models and \emph{pp}-waves }    
\label{sec:cayley-analogue-pp}
The indefinite analogue of \emph{pp}-wave metrics in dimension four are given by $\omega^0\omega^3-\omega^1\omega^2$ where 
\begin{equation}
  \label{eq:pp-wave}
  \omega^0=\exd x^0,\quad \omega^1=\exd x^1- G(x^0,x^1,x^2,x^3)\exd x^2,\quad \omega^2=\exd x^2,\quad \omega^3=\exd x^3.\
\end{equation}
Among indefinite  conformal structures in dimension four, the submaximal model which has 9-dimensional algebra of infinitesimal symmetries corresponds to the conformal class of the half-flat \emph{pp}-wave for which $ G(x^0,x^1,x^2,x^3)=(x^3)^2$ (see \cite{KT}.) The corresponding pair of torsion-free ODEs for this conformal structure is $(z^1)''=({z^2}')^3$ and $(z^2)''=0$ (see \cite{CDT}).

Similarly, one can introduce Cayley structures of the form \eqref{eq:cayely} where $\omega^i$'s are as in \eqref{eq:pp-wave}. Following the coframe adaptation in Section \ref{sec:V-iso-triv}, the obstruction to half-flatness for such Cayley structures is found to be 
\[\mathbf{f}_0^2G_{x^1x^1}+\mathbf{u}^2G_{x^0x^0}-2\mathbf{uf}_0G_{x^0x^1},\]
where $\mathbf{f}_0$ and $\mathbf{u}$ are the parameters in the structure group \eqref{eq:StrGroup}. It is straightforward to verify that inequivalent classes of half-flat Cayley structures of  this form correspond to  
\[G=G_0(x^2,x^3)x^0+G_1(x^2,x^3)x^1+G_3(x^3)\]
for which $a_1=\frac{3\mathbf{f}_0G_1-7\mathbf{u}G_0}{\mathbf{f}_0^2\mathbf{f}_1^2}$ and $b_{123}\neq 0$ but too long to write down. 

Assuming  $G=G(x^3),$ one can check that the only non-vanishing quantity in \eqref{eq:str-eqns-cayley} is 
\[b_{123}=\textstyle{\frac{1}{\mathbf{f}_0^2\mathbf{f}_1^7}G''(x_3)}.\] 
As a result, such Cayley structures are ultra-half-flat and the induced 2D path geometries on the leaves and the leaf space of $\cF_2$ in \eqref{eq:invflag}  are flat. 

Using the twistor correspondence, one can start with an ultra-half-flat Cayley structure and determine the corresponding pair of 2nd order ODEs as discussed in  Example \ref{sec:an-ansatz-pairs}. However, for Cayley analogues of \emph{pp}-waves one can check the ODEs are given by
\begin{equation}
  \label{eq:pp-wave-ODE}
(z^1)''=z^2+f({z^2}'),\qquad {z^2}''=0.
\end{equation}
More explicitly,
the system \eqref{eq:pp-wave-ODE} can be solved explicitly to give
\[\textstyle{z^1(t;C) = \frac 16 C_3 t^3+\frac 12 (f(C_3)+C_2)t^2+C_1 t+C_0,\qquad z^2(t;C) = C_3t+C_2}\]
where we have used the notation $z(t;C)=(z^1(t;C),z^2(t;C))$ in order to take the initial conditions $C=(C_0,\ldots,C_3)$ into consideration. As explained in Section \ref{sec:twist-constr}, the 4D solution space, $M,$ can be identified by a section  $t=t_0$ of the jet space $J^2(\RR,\RR^2)$ and  $C_0,\ldots,C_3$ serve as its local coordinates.  Using the double fibration, a solution $z(t;C)$ corresponds to a surface in $M.$ Therefore, we can consider the 1-parameter family  of 2-planes that are tangent to a solution and pass through  a point $C\in M.$ Consequently, one obtains a curve $\gamma^*(t):= \exd_{C}z(t;C)\subset T^*_CM$ defined as
\[\textstyle{\exd_Cz^1(t;C)=\frac 16 \exd C_3 t^3+\frac 12 (f'(C_3)\exd C_3+\exd C_2)t^2+\exd C_3 t+\exd C_2,\quad \exd_Cz^2(t;C)=\exd C_3 t+\exd C_2}.\]
A simple calculation shows that this ruled surface is a Cayley cubic and $\omega^i$'s are given by \eqref{eq:pp-wave} where 
\[x^i=C_i,\ 0\leq i\leq 3,\quad \mathrm{and}\quad G(x^3)=\textstyle{\frac{1}{2}f'(x^3)}.\]
Assuming  $f(x^3)=(x^3)^n,$ it follows that when $n\in \{0,1,2\},$ the Cayley structure is flat and therefore has 8-dimensional algebra of infinitesimal symmetries. In the case of $n=3,$ the algebra of infinitesimal symmetries for the Cayley structure is 7-dimensional and when $n\in \RR\backslash\{0,1,2,3\}$ it has dimension 6. 

As a result, the Cayley structure  corresponding to $f(x^3)=(x^3)^3$ is  a submaximal Cayley structure. In fact, there are only two inequivalent Cayley structures that are submaximal. This can be shown rather straightforwardly in two different ways. One way is to view a Cayley structure as a reduced causal geometry defined on a 6D manifold i.e. its projectivized null cone bundle as explained in Section \ref{sec:cayl-isotr-flat}. In this case the structure group is comprised of two scaling parameters $\mathbf{f}_0$ and $\mathbf{f}_1.$  Alternatively, one can view Cayley structures defined on a 4-dimensional manifold whose structure group is 4-dimensional. For simplicity we will use the  first viewpoint. According to Theorem \ref{thm:half-flatness-full-1} and Remark \ref{rmk:fund-inv}, the fundamental invariants of a Cayley structure are $a_1$ and $b_{123}.$ The infinitesimal  action of the structure group on $a_1$  and $b_{123}$ is given by
\begin{equation}
  \label{eq:submax-1}
  \exd a_1\equiv (\phi_0+2\phi_1)a_1,\quad \exd b_{123}\equiv(2\phi_0+7\phi_1)b_{123}
\end{equation}
modulo $ \{\omega^0,\omega^1,\omega^2,\omega^3,\theta^1,\theta^2\},$ where, as mentioned in Section \ref{sec:structure-equations}, $\phi_i$ is the Maurer-Cartan form associated to the group parameter $\mathbf{f}_i$ for $i=0,1.$ Following Cartan's reduction procedure \cite{Gardner,Olver}, the infinitesimal group actions \eqref{eq:submax-1} imply that if $a_1$ and $b_{123}$ are both non-zero  then they can be both normalized to a constant simultaneously, and therefore the algebra of infinitesimal symmetries cannot be larger than 6-dimensional. As a result,  for the symmetry algebra to be 7-dimensional, one of these scalars has to vanish. The case $b_{123}=0$ is considered in Example \ref{sec:proj-struct-cohom} and results in two distinct non-flat models whose algebra of infinitesimal symmetries are of dimension 7 and 6. Here we  assume  $a_1=0,$ which by Remark \ref{rmk:fund-inv} implies torsion-free Cayley structures. In this case, one can normalize $b_{123}$ to a constant and reduce the structure bundle which results in the relation
$\phi_1\equiv -\frac 2 7\phi_0+P_i\omega^i+Q_a\theta^a$ for some functions $P_0,\ldots,P_3,Q_1,Q_2.$   
It turns out that  the scaling parameter $\mathbf{f}_0$ acts on all $P_i$'s and $Q_a$'s. Therefore, for a torsion-free Cayley structure with 7-dimensional symmetry algebra one has $\phi_1=-\frac 27\phi_0.$ This  means that  such Cayley structure is homogeneous and  uniquely defined by the conditions
\[a_1=0,\qquad\qquad \exd b_{123}=(2\phi_0+7\phi_1)b_{123},\]
which hold for the  Cayley structure given by ~\eqref{eq:pp-wave}. Because in this case $b_{123}\neq 0,$  the Fels curvature of the corresponding path geometry, given by \eqref{eq:FelsCurv-Cayley},  is nonzero. Furthermore, for all Cayley structures the Fels torsion of the corresponding path geometry, given by \eqref{eq:FelsTorsion-Cayley}, is always nonzero. Therefore, the induced path geometry on the twistor space of this Cayley structure  is neither torsion-free nor a projective structure.

\subsection{Projective structures corresponding to  Cayley structures}
\label{sec:proj-struct-cohom}
In Theorem \ref{thm:half-flatness-full-1} and Remark \ref{rmk:fund-inv}, it was discussed how  the vanishing of $a_1$ and $b_{123}$ imply flatness.  
In this section  the branch $b_{123}=0$ is investigated. By equation \eqref{eq:FelsCurv-Cayley}, if $b_{123}=0$ then the Fels curvature of the 3D path geometry on the twistor space is zero and therefore one obtains a projective structure on $\cT.$ Setting $b_{123}=0,$ and checking $\exd^2=0$ for the structure equations \eqref{eq:str-eqns-cayley} and the quantities $b_{ijk}$ and $a_i,$ and carrying out rather tedious but straightforward computations,  it follows that there are only \emph{three} inequivalent Cayley structures for which $b_{123}=0$. The first possibility is  the flat Cayley structure for which all quantities $b_{ijk}$ and $a_i$ in \eqref{eq:str-eqns-cayley} vanish. As was mentioned at the end of Section \ref{sec:cartan3d}, the flat Cayley structure corresponds to the Egorov projective structure \eqref{eq:Egorov-ODE-flatCayley} on the twistor space.

The second possibility is  when all quantities $b_{ijk}$ and $a_i$ vanish except $a_1$ satisfying the differential relation
\[\exd a_1=-2a_1^2\omega^2+a_1(\phi_0+2\phi_1).\]
Since there is a scaling action on $a_1,$ one can normalize $a_1=1$ which yields the reduction
\[\phi_1=\omega^2-\textstyle\frac 12\phi_0.\]
We can simplify the structure equations by defining
\begin{equation}
  \label{eq:coframe-zeta-hom}
  \zeta^0=\omega^0-\theta^1,\quad \zeta^1=3\omega^1-\theta^2,\quad \zeta^2=4\omega^2-\phi_0,\quad \zeta^3=\omega^3,\quad\zeta^4=\theta^1,\quad \zeta^5=\theta^2,\quad \zeta^6=\phi_0
\end{equation}
as a result of which
\begin{subequations}    \label{eq:hom-proj-str-zeta}
  \begin{align}
    \exd\zeta^0&=-\zeta^6\w\zeta^0 \label{eq:hom-proj-str-1-zeta}\\
    \exd\zeta^1&= \textstyle{-\frac 34\zeta^6\w\zeta^1+\zeta^1\w\zeta^2-3\zeta^4\w\zeta^3} \label{eq:hom-proj-str-2-zeta}\\
    \exd\zeta^2&=-\zeta^5\w\zeta^3\label{eq:hom-proj-str-3-zeta}\\
    \exd\zeta^3&=-2\zeta^2\w\zeta^3\label{eq:hom-proj-str-4-zeta},\\
    \exd\zeta^4&= -\zeta^5\w\zeta^2\textstyle{-\frac 34\zeta^6\w\zeta^4-\frac 13\zeta^5\w\zeta^1} \label{eq:hom-proj-str-5-zeta},\\
    \exd\zeta^5&=-2\zeta^5\w\zeta^2\label{eq:hom-proj-str-6-zeta}\\
    \exd\zeta^6&=0,\label{eq:hom-proj-str-7-zeta}
  \end{align}
\end{subequations}
Now one can use Darboux's theorem to express $\zeta^i$'s in terms of some coordinate system. In particular, \eqref{eq:hom-proj-str-7-zeta} implies $\zeta^6=\exd x^6$ and \eqref{eq:hom-proj-str-zeta} gives $\zeta^0=e^{-x^6}\exd x^0.$ It follows from \eqref{eq:hom-proj-str-3-zeta}, \eqref{eq:hom-proj-str-4-zeta}   and \eqref{eq:hom-proj-str-6-zeta} that 
\[\zeta^2=-\exd x^2+\textstyle{\frac{\partial}{\partial x^3}f(x^3,x^5)\exd x^3,\qquad \zeta^3=e^{2x^2}\exd x^3,\qquad\zeta^2=e^{-2x^2+2f(x^3,x^5)}\exd x^5},\]
where $f(x^3,x^5)$ satisfies Liouville's equation
\[\textstyle{\frac{\partial^2}{\partial x^3\partial x^5}f(x^3,x^5)+e^{2f(x^3,x^5)}=0}.\]
The solutions of Liouville's equation can be expressed as \cite{Henrici}
\begin{equation}
  \label{eq:Liouville-solution}
  f(x^3,x^5)=\mathrm{ln}\left(\frac{\sqrt{p'(x^3)q'(x^5)}}{p(x^3)-q(x^5)}\right)
\end{equation}
for two arbitrary functions $p(x^3)$ and $q(x^5).$ A choice of $p(x^3)$ and $q(x^5)$ will be introduced later in such a way that it will simplify the expressions for $\zeta^1$ and $\zeta^4.$ Writing
\[\zeta^1=e^{-\frac 34 x^6}\left(A_1\exd x^1+A_2\exd x^2+A_3\exd x^3\right)\]
for  functions $A_i=A_i(x^1,x^2,x^3,x^4),i=1,2,3,$ it follows from  \eqref{eq:hom-proj-str-2-zeta} that 
\[A_1=x^2+B_1(x^1,x^3),\quad A_2=B_2(x^2,x^3)\quad \zeta^4=F_1\exd x^1+F_2\exd x^2+F_3\exd x^3+F_4\exd x^4\]
where the functions $F_1,F_2,F_4$ are uniquely expressed in terms of previously specified functions and $F_4=e^{-\frac 34 x^6}E_4(x^1,x^2,x^3,x^4,x^5).$   
The only remaining task is to  impose conditions obtained from   \eqref{eq:hom-proj-str-5-zeta} which in particular  imply  $p(x^3)=x^3.$ We will not illustrate these conditions which involve elementary PDEs. Solving the resulting PDEs and using the relations \eqref{eq:coframe-zeta-hom}  the Cayley structure is given by the coframe  
\begin{equation}
  \label{eq:Cayley-homo-coframe}
\begin{aligned}
  \omega^0 &= \textstyle{\exd x^0-\frac 13 e^{-x^3-x^2}\exd x^1+\frac 13 e^{-2x^2}(\exd x^2-\exd x^3)},\qquad   &\omega^2 &= -\exd x^2-\exd x^3\\
 \omega^1 &= \textstyle{\frac 13 e^{x^2}\exd x^1+\frac 13 e^{x^3}\exd x^3},\qquad &\omega^3 &= e^{2x^2+x^3}\exd x^3
\end{aligned}
\end{equation}
This example is the second submaximal Cayley structure. Note that its corresponding projective structure on the twistor space is also homogeneous. In fact, one can start by a projective structure and reduce the structure equations and find the conditions which imply that it is equivalent to this submaximal model. These conditions involve many PDEs and solving them in general is not easy. But in this case one can find that the corresponding point equivalence class of pair of ODEs can be presented as
\begin{equation}
  \label{eq:submaximal-projstr-CayleyExa1}
{z^1}''=tz^2{z^2}',\qquad {z^2}''=0.
  \end{equation}
Solving this pair of ODEs gives
\[z^1(t;C)=\textstyle{\frac 1{12}C_3^2 t^4+\frac 16C_3C_2t^3+C_1t+C_0},\qquad z^2(t;C)=C_3t+C_2.\]
As a result, following the previous example, one obtains that the Cayley structure is given by the cubic \eqref{eq:cayely} where
\begin{equation}
  \label{eq:coframe-Cayley-hom-simpler}
  \omega^0= \exd x^0,\qquad \omega^1=-\textstyle{(\frac{3}{2x_2})^{\frac 14}\exd x^1},\qquad \omega^2=-\exd x^2,\qquad \omega^3=\exd x^3,
\end{equation}
with $x^0=C_0,x^1=C_1,x^3=C_3$ and $x^2=-3(2)^{-\frac 73} (C_2)^{\frac 43}.$ Comparing with \eqref{eq:Cayley-homo-coframe}, the coframe above has a  much simpler form. Naturally, one can use the coframe \eqref{eq:Cayley-homo-coframe} together with the appropriate structure group element and change of coordinates to obtain \eqref{eq:coframe-Cayley-hom-simpler} which we will leave to the interested reader. 

The third possibility when $b_{123}=0$ turns out to be when  $a_1$ and $a_2$ are nonzero satisfying the differential relations
\[\exd a_1=-2a_2^2\omega^0+3a_1a_2\omega^1-2a_1^2\omega^2+a_1(\phi_0+2\phi_1)-a_2\theta^2,\qquad \exd a_2=a_2^2\omega^1-a_1a_2\omega^2+a_2(\phi_0+\phi_1)\]
with the only  nonzero quantities in \eqref{eq:str-eqns-cayley} being   
\[a_3=a_2,\qquad b_{101}=a_1a_2,\qquad b_{201}=-2a_2^2,\qquad b_{202}=2a_1a_2,\qquad b_{403}=a_1a_2.\] 
In this case, one imposes the normalization $a_2=1$ which implies the reduction
\begin{equation}
  \label{eq:red1-hom-proj}
  \phi_1=-\phi_0-\omega^1+a_1\omega^2.
    \end{equation}
Furthermore, one can translate $a_1$ to zero which, using the expression for $\exd a_1,$ results in the additional reduction
\begin{equation}
  \label{eq:red2-hom-proj}
  \theta^2=-2\omega^0.
  \end{equation}
Replacing the values $a_2=1,a_1=0$ together with \eqref{eq:red1-hom-proj} and \eqref{eq:red2-hom-proj} into the structure equations one obtains a set of six structure equations which is homogeneous. Defining
\[\zeta^0=\omega^0,\quad \zeta^1=\omega^1-\phi_0,\quad \zeta^2=\omega^0,\quad \zeta^3=\omega^3,\quad \zeta^4=\theta^1,\quad \zeta^5=2\omega^1+\phi_0,\] 
it follows that 
 \begin{subequations}    \label{eq:cohom-proj-str-eqns}
  \begin{align}
    \exd\zeta^0&=-\zeta^5\w\zeta^0-\zeta^4\w\zeta^2\label{eq:cohom-proj-str-1}\\
    \exd\zeta^1&=6\zeta^0\w\zeta^2\label{eq:cohom-proj-str-2}\\
    \exd\zeta^2&=\zeta^5\w\zeta^2+3\zeta^0\w\zeta^3\label{eq:cohom-proj-str-3}\\
    \exd\zeta^3&=2\zeta^5\w\zeta^3\label{eq:cohom-proj-str-4}\\
    \exd\zeta^4&=-2\zeta^5\w\zeta^4\label{eq:cohom-proj-str-5}\\
    \exd\zeta^5&=3\zeta^3\w\zeta^4\label{eq:cohom-proj-str-6}
  \end{align}
\end{subequations}
Similar to the previous case, one uses Darboux's theorem to express the coframe in terms of a  coordinate system. Note that in this case Liouville's equation also appears as a result of  \eqref{eq:cohom-proj-str-4}, \eqref{eq:cohom-proj-str-5} and \eqref{eq:cohom-proj-str-6}. A choice of coframe for this Cayley structure can be expressed as
 \[
      \begin{aligned}
        \omega^0 &= \exd x^0 +\exd x^2+2\exd x^3,\\
          \omega^1 &=\textstyle{ -6\frac{2(x^0)^2x^2-(x^0)^2x^3-2x^2(x^3)^2}{x^2(x^0)^2x^3}\exd x^0 +\frac{1}{3}\exd x^1 -6\frac{2x^0(x^2)^2-2x^0(x^3)^2-(x^2)^2x^3}{(x^2)^2x^3x^0}\exd x^2},\\
          \omega^2 &= \textstyle{3\frac{(x^0)^2+(x^3)^2}{(x^0x^3)^2}\exd x^0 +3\frac{(x^2)^2+(x^3)^2}{(x^2x^3)^2}\exd x^2},\\
        \omega^3 &= \textstyle{\frac{2}{(x^3)^3}\exd x^3}
\end{aligned}
\]
We did not succeed in solving the PDEs required to obtain a pair of ODEs that corresponds to this Cayley structure. We point out that although the Cayley structure arising from \eqref{eq:cohom-proj-str-eqns} is homogeneous with 6-dimensional algebra of infinitesimal symmetries, the corresponding   projective structure induced on its twistor space is not homogeneous but instead has cohomogeneity one. This is a result of the reduction \eqref{eq:red2-hom-proj}. Recall from \eqref{eq:sl4-conn-horiz-forms} that the vector field $\frac{\partial}{\partial \theta^2}$ is tangent to the curves of the path geometry. Hence  reducing $\theta^2$ amounts to taking a hypersurface in $J$ that is transversal to the paths.
As a result, the 6-dimensional symmetry group foliates  $J$ by  orbits of codimension one  which are transversal to the 1-dimensional foliation induced by the  path geometry.

\subsection{A diagonal ansatz}\label{ex:diagonal}
Consider a Cayley structure for which 
\begin{equation}\label{eq:ex-coframe}
\omega^0=f^0\exd x^0,\quad \omega^1=f^1\exd x^1,\quad \omega^2=f^2\exd x^2,\quad \omega^3=f^3\exd x^3
\end{equation}
for the functions $f^i=f^i(x^0,x^1,x^2,x^3),i=0,1,2,3.$  Straightforward computation shows that such Cayley structure is half-flat if
\[f^i=f^i(x^2,x^3),\quad i=0,1,2,3.\]
Imposing the structure equations \eqref{eq:str-eqns-cayley}, one obtains that $a_2=a_3=\cdots=a_6=0$ and
\[\textstyle{a_1= \frac{1}{\mathbf{f}_0\mathbf{f}^2_1f^0f^1f^2f^3}(f^0f^1f^3_{x^2}-3f^0f^3f^1_{x^2}+2f^1f^3f^0_{x^2})}\]
where $\mathbf{f}_0,\mathbf{f}_1$ are parameters in the structure group \eqref{eq:StrGroup}.
The expression for $b_{123}$ is too long to express here and  depends on the 2nd jet of $f^i$'s. One can check that such Cayley structures are ultra-half-flat and the induced projective structure on the leaves is flat. However, the induced path geometry on the 2-dimensional leaf space is always a projective structure and not necessarily flat. We do not express the curvature, $K_1$ of this projective structure here. Considering the case $f^0=f^1=f^2=1$, one can see that if $f^3(x^2,x^3)=h^1(x^2)h^2(x^3)$ then $K_1=0.$   For such choice of $f^i$'s the Cayley structure is not flat and
\[\textstyle{a_1=\frac{h^1_{x^2}}{h^1\mathbf{f}_0\mathbf{f}_1^2},\qquad b_{123}=-\frac{\mathbf{u}((h^1_{x^2})^2+h^1h^1_{x^2x^2})}{\mathbf{f}_0^5\mathbf{f}_1^2(h^1)^2}}\]

Furthermore, if we only set $f^3=1$ then according to equation \eqref{eq:RulingFomrs} the ruling planes of the Cayley cubic are spanned by
\[
\gamma_1(u)=\textstyle{\left[\frac{1}{f^1}\partial_1+u \frac{1}{f^0}\partial_0\right]},\qquad
\gamma_2(u)=\textstyle{\left[\partial_3+u \frac{1}{f^2}\partial_2+\frac{1}{3}u^2\frac{1}{f^1}\partial_1\right]}
\]
It follows that the null foliation is given by $x^2=\const$, $x^3=\const$ and is spanned by $\gamma_1(u)$ with $u\in\R$. Moreover, we are able to find the lift of the ruling planes to $TJ$ in terms of the integrable distribution $\cD$ which is dual to the Pfaffian system $I_{\mathsf{hf}}$. They are spanned by the following vector fields
\[
V_1=\textstyle{\frac{1}{f^1}\partial_1-\lambda \frac{1}{f^0}\partial_0,}\qquad V_2=\textstyle{\partial_3-\lambda \frac{1}{f^2}\partial_2+\frac{1}{3}\lambda^2\frac{1}{f^1}\partial_1+\lambda\left(\frac{v(f^0)}{f^0}-\frac{v(f^1)}{f^1}\right)\partial_\lambda,}
\]
where $\lambda$ is a projective parameter on the fibers of $J\to M$ and we denote $v=\partial_3-\lambda \frac{1}{f^2}\partial_2$. Furthermore, the path geometry on the leaf space is defined by the integral curves of $V_2$ (recall that all functions $f^i$ depend on $x^2$ and $x^3$ only) and corresponds to the point equivalence class of the ODE
\begin{equation}
  \label{eq:diagonal-ansatz-ODE}
\textstyle{\frac{\exd^2x^2}{(\exd x^3)^2} =\frac{\exd x^2}{\exd x^3} \partial_3\ln\left(\frac{f^0}{f^1f^2}\right)+\left(\frac{\exd x^2}{\exd x^3}\right)^2 \partial_2\ln\left(\frac{f^0}{f^1f^2}\right)}.
\end{equation}
It is a polynomial of second order in $\frac{\exd x^2}{\exd x^3}$ which confirms the fact that the path geometry is a projective structure.
Note that for $V_2$ the vertical component $\left(\frac{v(f^0)}{f^0}-\frac{v(f^1)}{f^1}\right)\partial_\lambda$ is zero if and only if $f^0$ and $f^1$ are proportional in which case the two-dimensional twistor space of the projective structure fibers over $\PP^1$.  This is equivalent to the fact that the projective structure is defined by a Veronese web or, equivalently,  by a Chern connection of a 3-web,  which implies that this projective structure can be represented by a  connection  whose  Ricci tensor is skew-symmetric
(see \cite{DW1,K-webs}).
 
\subsection{An ansatz for pairs of ODEs} \label{sec:an-ansatz-pairs}
Finding a system of second order ODEs from a given Cayley structure yields a solution to a linear system of PDEs as explained in Section \ref{sec:3d-path-geometry}. Namely one looks for functions that are constant along $\ell_1=\cD=I_{\mathsf{hf}}^\perp$ (so-called twistor functions) which constitute coordinates on the twistor space $\cT$, i.e. the space where the 3-dimensional path geometry is defined. The twistor functions accompanied with coordinates on the integral manifolds of $I_{\mathsf{hf}}$ give new system of coordinates on $J$  using which it can be viewed as $J^1(\mathbb{R},\mathbb{R}^2)$. In the new coordinate system, the vertical (with respect to the projection to $M$) vector field $\partial_\lambda$ on $J$  takes the form of a total derivative vector field, from which a system of ODEs can be derived. Unfortunately one cannot give a general formula as a result of this procedure. However one can easily notice that under the diagonal ansatz of the previous section the resulting system involves one dependent function and its derivative only, i.e. it is of the form
\begin{subequations}\label{ex:ode}
  \begin{align}
    (z^1)''=F^1(t,z^2,{z^2}'),\label{eq:ODE-ansatz-exa-1}\\
 (z^2)''=F^2(t,z^2,{z^2}') \label{eq:ODE-asnsatz-exa-2}
  \end{align}
\end{subequations}
To illustrate this fact consider a general systems of the form \eqref{ex:ode}. Our aim is to apply Theorem \ref{thm1a} in this case. There are two non-trivial components of the torsion $\mathbf{T}$
\[
T^2_2=\textstyle{-\partial_{z^2}F^2+\frac{1}{2}\tilde X(\partial_{p^2}F^2)-\frac{1}{4}(\partial_{p^2}F^2)^2,}
\]
and
with
\[
T^1_2=\textstyle{-\partial_{z^2}F^1+\frac{1}{2}\tilde X(\partial_{p^2}F^1)-\frac{1}{4}\partial_{p^2}F^1\partial_{p^2}F^2}
\]
where 
\begin{equation}
  \label{eq:xtilde-def}
 \tilde X=\partial_t+p^2\partial_{z^2}+F^2\partial_{p^2} 
\end{equation}
is a part of the total derivative vector field involving coordinates that appear on the right hand side of \eqref{ex:ode} only. 
Theorem \ref{thm1a} gives two conditions for a pair of ODEs that corresponds to a Cayley structure which are 
\[
T^2_2=0,\qquad  \textstyle{\tilde X(\phi)=\frac{1}{4}\phi^2,}
\]
where $\phi$ is given by
\[
\textstyle{\phi=\frac{\tilde X(T^1_2)}{T^1_2}+\frac{1}{2}\partial_{p^2}F^2}
\]
The first condition, $T^2_2=0$, ensures that $\mathbf{T}$ is of rank 1 in this case. Note that $T^2_2$ is a curvature invariant of a spray associated to the ODE \eqref{eq:ODE-asnsatz-exa-2}.  Vanishing of $T^2_2$ implies that \eqref{eq:ODE-asnsatz-exa-2} defines 2-dimensional projective structure on its solution space with a skew-symmetric Ricci tensor  \cite{K-webs}. Using the notion of  duality for path geometries in dimension two (the duality for second order ODEs is defined in terms of $\ell_1$ and $\ell_2$ by swapping the role of the two line fields), it follows that \eqref{eq:ODE-asnsatz-exa-2}  is the dual ODE to  \eqref{eq:diagonal-ansatz-ODE}. Let us stress that $T^2_2$ is not a point invariant of \eqref{eq:ODE-asnsatz-exa-2}.
\begin{proposition}\label{ex:thm} 
For any function $F^2=F^2(t,z^2,p^2)$ satisfying $T^2_2=0$ there exists a function $F^1=F^1(t,z^2,p^2)$ such that the system \eqref{ex:ode} defines a Cayley structure on the solution space. 
\end{proposition}
\begin{proof}
Assume that $F^2$ satisfies $T^2_2=0.$ Choose a function $\phi$ such that $\tilde X(\phi)=\frac{1}{4}\phi^2$ where $\tilde X$ is defined in \eqref{eq:xtilde-def}. Then we can find a function $\tau$  such that $\phi=\frac{\tilde X(\tau)}{\tau}+\frac{1}{2}\partial_{p^2}F^2$ holds. Note that $\phi$ and $\tau$ are uniquely determined by initial conditions given by functions of two variables. Lastly, the second order PDE $\tau=T^1_2$ can be used to determined a choice of $F^1=F^1(t,z^2,p^2)$.
\end{proof}

We would like to utilize Proposition \ref{ex:thm} and write down explicit examples of pairs of ODEs that define Cayley structures on their solution space. We have already seen the example  \eqref{eq:pp-wave-ODE} in which $F^1(t,z^2,p^2)=z^2+f(p^2)$ and $F^2=0$.   Let us assume  $F^2=0$ and look for a choice of $F^1$ that is of the form 
$$
F^1(t,z^2,p^2)=(z^2)^k(p^2)^\ell.
$$
It follows that 
$$
T^1_2=\textstyle{k(\frac{1}{2}\ell-1)(z^2)^{k-1}({z^2}')^\ell.}
$$
and it is non-zero provided that $k\neq 0$ and $\ell\neq 2$. One obtains that
$$
\textstyle{\phi=\frac{X(T^1_2)}{T^1_2}=(k -1)\frac{z^2}{{z^2}'}}
$$
which satisfies $X(\phi)-\frac{1}{4}\phi^2=0$ if and only if $k=-3$ or $k=1$. Therefore, we have the following.
\begin{proposition}\label{propex3}
For any $\ell\neq 2$ the point equivalence class of pairs of ODEs
\begin{equation}
  \label{eq:exa-1-ode}
 {z^1}''=z^2({z^2}')^\ell,\qquad {z^2}''=0, 
\end{equation}
and
 \begin{equation}
   \label{eq:exa-2-ode}
 {z^1}''=(z^2)^{-3}({z^2}')^\ell,\qquad {z^2}''=0,
 \end{equation}
define Cayley structures on their solution spaces. If $\ell=2$ then the pair of ODEs is torsion-free and corresponds to a conformal structure.
\end{proposition}
One can easily solve the ODEs in Proposition \ref{propex3} explicitly and find the corresponding Cayley structures.  The general solution of \eqref{eq:exa-1-ode} is given by 
\[z^1(t;C) = \textstyle{C_3^\ell(\frac 16C_3t^3+\frac 12 C_2t^2)+C_1t+C_0},\qquad  z^2(t;C) = C_3t+C_2.\]
  Following Example \ref{sec:cayley-analogue-pp}, one obtains that the solution space is equipped with the cubic
\[\textstyle{\frac{\ell}{2} (x^3)^{\ell-1} x^2 (\exd x^2)^2 \exd x^3+\frac 16 (2-\ell) (x^3)^\ell (\exd x^2)^3+\exd x^0 (\exd x^3)^2-\exd x^1 \exd x^2 \exd x^3}\]
where $x^i=C_i$ for $i=0,\ldots,3.$ This cubic can be expressed as \eqref{eq:cayely} where 
\begin{equation}
  \label{eq:cof-exa-1}
\begin{aligned}
  \omega^0&=\exd x^0,\quad &\omega^1&=\textstyle{\left(\frac{2}{(2-\ell)(x^3)^\ell}\right)^{\frac 13}\left(\exd x^1-\frac 12 x^2(x^3)^{\ell-1}\exd x^2\right)}\\
  \omega^2&=\textstyle{\left(\frac{(2-\ell)(x^3)^\ell}{2}\right)^{\frac 13} \exd x^2},\quad &\omega^3&=\exd x^3. 
\end{aligned}
\end{equation}
The general solution of \eqref{eq:exa-2-ode} can be expressed as
\[z^1(t;C)=\textstyle{\frac 12 \frac{C_2^{\ell-2}}{C_2 t+C_3}+C_0 t+C_1},\qquad z^2(t;C)=C_2 t+C_3\] 
As a result, the solution space is equipped with a field of cubics given by 
\begin{equation*}
  \textstyle{x^3 \exd x^1 (\exd x^2)^2-x^3 \exd x^0 \exd x^3 \exd x^2-x^2 \exd x^1 \exd x^3 \exd x^2+x^2 \exd x^0 (\exd x^3)^2-\frac 12 (x^2)^{\ell-3} (\exd x^2)^3 (2-\ell)} 
\end{equation*}
where $x^i=C_i$ for $i=0,\ldots,3.$ This cubic can be put in the form  \eqref{eq:cayely}  by defining 
\begin{equation}
  \label{eq:cof-exa-2}
\begin{aligned}
  \omega^0&=\textstyle{\frac{1}{x^2}\exd x^0},\quad &\omega^1&=2\left({3\left(\ell-2\right)\left(x^2\right)^{\ell}}\right)^{-\frac 13}\left(x^2\exd x^1-x^3\exd x^0\right)\\
  \omega^2&=\textstyle{\frac 12 \left(12(\ell-2)(x^2)^{\ell-3}\right)^\frac{1}{3}\exd x^2},\quad &\omega^3&=x^2\exd x^3-x^3\exd x^2. 
\end{aligned}  
\end{equation}
Following the procedure explained in Section \ref{sec:half-flatness-full} and using the coframe  \eqref{eq:cof-exa-1}, one finds that  $a_1$ and $b_{123}$ for Example \eqref{eq:exa-1-ode} are  
\begin{equation}
  \label{eq:a1b123-exa41}
a_1=0,\qquad b_{123}=\textstyle{\frac{2\ell(\ell^2-3\ell+2)}{\mathbf{f}_0^2\mathbf{f_1}^7(8-4\ell)^{\frac 23}}x^2(x^3)^{\frac{\ell}{3} -3}-\frac{\mathbf{u}}{\mathbf{f}_0^3\mathbf{f}_1^7}\frac{\ell(\ell-1)}{(x^3)^2}}
\end{equation}
and for Example \eqref{eq:exa-2-ode}, using the coframe \eqref{eq:cof-exa-2}, are
\begin{equation}
  \label{eq:a1b123-exa42}
a_1=\textstyle{\frac{1}{\mathbf{f}_0\mathbf{f}_1^2}\left(\frac 32 (x^2)^\ell(\ell-2)\right)^{-\frac 13}(\ell-1)},\qquad b_{123}=\textstyle{-\frac 13\left(\frac{2}{3(\ell-2)}\right)^{\frac 23}\ell(\ell-1)\frac{\mathbf{u}^3}{\mathbf{f}_0^5\mathbf{f}_1^7}(x^2)^{-\frac{2\ell}{3}}}.
\end{equation}
The expression for $b_{123}$ in \eqref{eq:a1b123-exa42} suggests that when $\ell=0$ the pair of ODEs \eqref{eq:exa-2-ode} define a projective structure. This pair of ODEs turn out to correspond to the submaximal Cayley structure obtained in Section \ref{sec:proj-struct-cohom}. In other words, the  pair of ODEs \eqref{eq:exa-2-ode} for $\ell=0$ and \eqref{eq:submaximal-projstr-CayleyExa1} are    point equivalent. 
Lastly, we point out that the algebra of infinitesimal symmetries for Examples  \eqref{eq:exa-1-ode} and \eqref{eq:exa-2-ode} are 6-dimensional for  $\ell\in\mathbb R\slash\{0,1,2\}$. Moreover, they are ultra-half-flat and the path geometries induced on the leaves and on the leaf space are both flat.

\section{Dispersionless Lax pairs}\label{sec:disp-lax-pairs}
The purpose of this section is to put a half-flat Cayley structure in a preferred coordinate system and describe the half-flatness condition in terms of a dispersionless integrable system. We  will show that  the  geometry induced by characteristic variety of the linearization of this integrable system is a Cayley structure. This can be viewed as a generalization of the results obtained in \cite{FK} for conformal structures.  Finally, we will present an alternative way of viewing the 2D path geometries appearing in Section \ref{sec:addit-geom-struct}.

\subsection{A normal form}
Using the correspondence between pairs of second order ODEs and  half-flat Cayley structures  as described in Theorem \ref{thm1a},  we find a convenient coframing on $M$ which will be used  to find a Lax pair.
\begin{theorem}\label{propNormalform} 
For a half-flat Cayley structure given by $\rho$ in \eqref{eq:cayely}, there are local coordinates $(x^0,x^1,x^2,x^3)$ with respect to which
\[
\begin{aligned}
&\omega^0=\emph{\exd} x^0\\
&\omega^1=\emph{\exd} x^1-(\partial_1F^2\emph\exd x^0+\partial_3F^2\emph\exd x^2)\\
&\omega^2=\emph\exd x^2-E\emph\exd x^0\\
&\omega^3=\emph\exd x^3-E\emph\exd x^1 -(\partial_1F^1\emph\exd x^0+\partial_3F^1\emph\exd x^2) +E(\partial_1F^2\emph \exd x^0+\partial_3F^2\emph\exd x^2)
\end{aligned}
\]
for some functions $E,F^1,F^2$.
\end{theorem}
\begin{proof}
The idea of proof is similar to what is presented in \cite{DFK} in the conformal setting. Given a half-flat Cayley structure on $M$, Theorem \ref{thm1a} allows one to identify $M$ as the solution space of  a pair of second order ODEs \eqref{eq:pair-of-2nd-order-ODEs} via the quotient map (see diagram \eqref{eq:AllFib-FINAL}) $\tilde\mu:=\pi\circ\mu^{-1}\colon J^1(\R,\R^2)\cong J\to M.$ Let $(t,z^1,z^2,p^1,p^2)$ denote the standard coordinates on $J^1(\R,\R^2)$. As in the proof of Theorem \ref{thm1a}, let $\mathbf{V}=(V_1,V_2)$ denote a normal frame of $\D$ corresponding to a projective vector field $X\in\Gamma(\X)$ and $J_0$ be the submanifold of $J^1(\R,\R^2)$ corresponding to $t=0$. Define
\begin{equation}\label{eqY}
Y_i=\partial_{p^i}.
\end{equation}
The freedom of choice for the initial condition for \eqref{eq2} allows us to  take $V_i=Y_i$ on $J_0$. Furthermore, we have a freedom of choice for $X$ so that  on $J_0$ it coincides, up to 1st jet, with the total derivative vector field $X_F=\partial_t+p^1\partial_{z^1}+p^2\partial_{z^2}+F^1\partial_{p^1}+F^2\partial_{p^2}$. For this choice, it follows that  $Z_i=\ad_XV_i$, $i=1,2$, are given by 
\begin{equation}\label{eqZ}
  Z_1=\partial_{z^1}+
 \textstyle \frac{1}{2}(\partial_{p^1}F^1\partial_{p^1}+\partial_{p^1}F^2\partial_{p^2}),\qquad  Z_2=\partial_{z^2}+\textstyle \frac{1}{2}(\partial_{p^2}F^1\partial_{p^1}+\partial_{p^2}F^2\partial_{p^2}),
\end{equation}
on $J_0$. We shall drop the factor $\frac{1}{2}$ in the formulae as it can be  incorporated in functions $F^1$ and $F^2$.  The formulae above for $Z_i$'s and $Y_i$'s (and consequently for $\mathbi{V}$ and $\ad_X\mathbi{V}$) are general and valid for any choice of the torsion. Now we shall utilize special properties of $\mathbf{T}^X$. Firstly, without any lost of generality, we can assume that there is a function $E$ such that the vector field
\[
K=V_2+EV_1
\]
spans the one-dimensional kernel of $\mathbf{T}^X$. Secondly, since  $\nabla_X\mathbf{T}^X$ is proportional to $\mathbf{T}^X,$ it follows that $\nabla_X$ preserves $K$ and consequently  $X(E)=0$ and
$\ad_XK=\ad_XV_2+E\ad_XV_1$.

Now, let $\lambda=\lambda(t,z,p)$ be the natural integral parameter on each integral line of $X$ satisfying $\lambda|_{J_0}=0$. Then $(\lambda,z|_{J_0},p|_{J_0})$ is a new system of coordinates on $J^1(\R,\R^2)$, where $(z|_{J_0},p|_{J_0})$ are the constants of integration of the system. Since, $\ad_X^2K=\partial^2_\lambda K=0\mod X$, after taking into account the initial condition on $J_0$, we get
\begin{equation}\label{eqK}
K(\lambda)=Y_2+EY_1+\lambda(Z_2+EZ_1)\mod X.
\end{equation}
A complementary normal vector field $L$  to $K$ has to satisfy $\ad_X^2L=\partial^2_\lambda L=K\mod X$ which leads to
\begin{equation}\label{eqL}
\textstyle{L(\lambda)=Y_1+\lambda Z_1+ \frac{\lambda^2}{2}(Y_2+E Y_1)+\frac{\lambda^3}{6}(Z_2+EZ_1)\mod X.}
\end{equation}
Vector fields $Y_1,Y_2,Z_1,Z_2$ are given explicitly by \eqref{eqY} and \eqref{eqZ}. Let $\omega^0,\ldots,\omega^3$ be  dual 1-forms to $Z_2+EZ_1$, $Y_2+EY_1$, $Z_1$ and $Y_1$, i.e.
\[
\textstyle{\frac{\partial}{\partial\omega^0}=Z_2+EZ_1,\quad
\frac{\partial}{\partial\omega^1}=Y_2+EY_1,\quad
\frac{\partial}{\partial\omega^2}=Z_1,\quad
\frac{\partial}{\partial\omega^3}=Y_1.}
\]
Note that $J_0$ is naturally identified with the solution space $M$. Hence, the restriction of $\omega^i$'s to $J_0$ defines a coframing on $M$ which, by construction, is adapted to the original Cayley structure. We set $(x^0,x^1,x^2,x^3)=(z^2,p^2,z^1,p^1)$.
\end{proof}

\begin{remark} Theorem \ref{propNormalform} can be rephrased in the language of \cite{KM} in the sense that any half-flat Cayley structure is $H$-flat, where $H$ is the following 5-dimensional subgroup of $\mathrm{GL}_4(\R)$
\[
\left(
\begin{array}{cccc}
1 & 0 & 0 & 0\\
e & 1 & 0 & 0\\
h_1 & h_2 & 1 & 0\\
h_3 & h_4 & e & 1\\
\end{array}
\right),
\]
In other words, for a half-flat Cayley structure one can always find local coordinate system $x=(x^0,\ldots,x^3)^T$ and an $H$-valued function $h$ such that $h\exd x$ is a local coframe adapted to the structure.
\end{remark}

\subsection{A Lax pair}
Now we shall express half-flatness of a Cayley structure with the coframing given in Theorem \ref{propNormalform}. The ruling planes for the Cayley structure are spanned by $\lambda$-dependent vector fields $K_\lambda$ and $R_\lambda$ given explicitly by
\begin{equation}\label{KRformula1}
\begin{aligned}
&K_\lambda=K_1+\lambda K_2,\\
&R_\lambda=R_1+\lambda R_2+\textstyle{\frac{1}{3}\lambda^2K_1}
\end{aligned}
\end{equation}
where the vector fields $K_2$, $K_1$, $R_2$ and $R_1$ are dual to 1-forms $\omega^i$, respectively, i.e.
\begin{equation}\label{KRformula2}
\begin{aligned}
K_1=\partial_1+E\partial_3,&\qquad K_2=\partial_0+E\partial_2+P\partial_3+Q\partial_1,\\
R_1=\partial_3,&\qquad R_2=\partial_2+S\partial_3+T\partial_1,
\end{aligned}
\end{equation}
with
\begin{equation}\label{defPQST}
P=K_1(F^1),\quad Q=K_1(F^2),\quad S=R_1(F^1),\quad T=R_1(F^2).
\end{equation}
Note that if a Cayley structure is half-flat then $K_\lambda=\tilde\mu_*K$ and $R=\tilde\mu_*L-\frac{\lambda^2}{6} \tilde\mu_*K$ where $\tilde\mu\colon J\to M$ and $K$ and $L$ are vector fields on $J$ given by \eqref{eqK} and \eqref{eqL}, respectively.

In order to characterize half-flatness we look for a rank-2 distribution on $J$ that projects to the ruling planes. For this we consider the following lifts of $K_\lambda$ and $R_\lambda$ to vector fields on $J$
\begin{equation}\label{Laxeq}
\begin{aligned}
&L_0=K_\lambda+m(\lambda)\partial_\lambda,\\
&L_1=R_\lambda+n(\lambda)\partial_\lambda,
\end{aligned}
\end{equation}
and ask if there exist a choice of $m(\lambda)$ and $n(\lambda)$ such that $L_0$ and $L_1$ span an integrable distribution. We shall show that $m(\lambda)$ and $n(\lambda)$ are polynomial functions in $\lambda$ and refer to $(L_0,L_1)$ as the \emph{Lax pair}. The integrability condition
\begin{equation}\label{Laxsys}
[L_0,L_1]\in\spn\{L_0,L_1\}
\end{equation}
gives a system of differential equations for functions $F^1$, $F^2$ and $E$, which are satisfied if and only if the structure is half-flat. We shall not write the system explicitly due to its length and the fact that the expression is not illuminating. However we have the following
\begin{theorem}\label{thm2}
The functions $m(\lambda)$ and $n(\lambda)$ in \eqref{Laxeq} have the following form
\[
m(\lambda)=\textstyle{\frac{1}{3}}\lambda^3\psi(\lambda)+\lambda^2\varphi(\lambda),\quad n(\lambda)=\textstyle{-\frac{\lambda^2}{3}}(m(\lambda)+3\eta(\lambda)+\xi(\lambda)),
\]
where
\[
\begin{aligned}
&\psi(\lambda)=K_\lambda(E),\\
&\varphi(\lambda)=K_\lambda(S)-R_\lambda(P)+SR_\lambda(E)-EK_\lambda(T)+ER_\lambda(Q)-ETR_\lambda(E),\\
&\xi(\lambda)=R_\lambda(E),\\
&\eta(\lambda)=K_\lambda(T)-R_\lambda(Q)+TR_\lambda(E).
\end{aligned}
\]
Moreover, condition \eqref{Laxsys} is equivalent to a system of 7 PDEs for unknown functions $F^1$, $F^2$ and $E$. The equations are of order 3 in $F^1$ and $F^2$ and of order 2 in $E$.
\end{theorem}

\begin{proof}
In order to find  $(L_0,L_1),$ we shall consider a rank-3 distribution $\mathcal{E}$ on $J$ spanned by $K_\lambda$, $R_\lambda$ and $\partial_\lambda$. Clearly one has $[\mathcal{E},\mathcal{E}]=TJ$ which   shows that there is unique rank-2 sub-distribution $\D\subset\mathcal{E}$ satisfying $[\D,\D]\subset\mathcal{E}$. Such $\D$ is uniquely defined by the kernel of the linear mapping $\mathcal{E}\wedge\mathcal{E}\to (TJ)/\mathcal{E}$ induced by the Lie bracket of sections. This sub-distribution is spanned  by $L_0=K_\lambda+m\partial_\lambda$ and $L_1=R_\lambda+n\partial_\lambda$, where $m$ and $n$ are obtained as the unique solution to a pair of algebraic equations. As a result one obtains
\[
m(\lambda)=\textstyle{\frac{1}{3}}\lambda^3\psi(\lambda)+\lambda^2\varphi(\lambda),\quad n(\lambda)=\textstyle{-\frac{\lambda^2}{3}}(m(\lambda)+3\eta(\lambda)+\xi(\lambda)).
\]
The polynomials $m(\lambda)$ and $n(\lambda)$  are of degree 4 and 6 in $\lambda$ respectively. Now, condition \eqref{Laxsys} is equivalent to the fact that $[\D,\D]=\D$. Since the form of $m(\lambda)$ and $n(\lambda)$ guarantee that $[\D,\D]\subset \mathcal{E}$, the Lie bracket $[L_0,L_1]$ can be written as a unique linear combination of $L_0$, $L_1$ and $\partial_\lambda$. Therefore, \eqref{Laxsys} is equivalent to the fact that there cannot be a $\partial_\lambda$ component in this linear combination. The coefficient of $\partial_\lambda$ in this linear combination is a polynomial function in $\lambda$ of degree 9 with no terms of  order 0, 1 and 2. Thus, there are 7 terms left and they have to vanish if  \eqref{Laxsys} is to be satisfied. The terms depend on derivatives of $F^1$ and $F^2$ up to order 3 and on derivatives of $E$ up to order 2.
\end{proof}
\begin{remark}\label{rmk:LaxPair-branching-involutivity}  The system  \eqref{Laxsys}, referred to as the Lax system, is  overdetermined for the 3 unknown functions $F^1, F^2$ and $E$. Using Cartan-K\"ahler analysis one can study the involutivity of this system. In order to investigate involutivity, one observes that   the system of  PDEs give rise to polynomial equations in terms of the coordinates of the jet space $J^3(\RR^4,\RR^3)$.  To find the  generic solution of these equations one encounters only one branching as a result of a quadratic equation which is a manifestation of the  branching in \eqref{eq:TorsionTermsBranching}.   The expression of these equations are extremely long and unilluminating and therefore not presented here. The system turns out to be  not involutive and has torsion. We could not interpret the torsion geometrically and its vanishing did not imply involutivity after one prolongation. Further prolongation was not possible since the computations could no longer be  carried out by a computer. Moreover, starting with the structure equations \eqref{eq:str-eqns-cayley}, we also could not establish involutivity  for general half-flat Cayley structures  in spite of two prolongations  after which computations become intractable.  This is in contrast to the half-flat conformal case where the system obtained similarly is involutive \cite{DFK}  depending on 6 functions of 3 variables and the involutivity is obtained without any prolongation. Lastly, we mention that it is possible to establish involutivity in the case $E=0.$ After setting torsion to zero and prolonging the equations one obtains that such Cayley structures depend on 9 functions of 2 variables.  
\end{remark}

\subsection{The characteristic variety}
In this section we will analyze the symbol of the Lax system \eqref{Laxsys}.

Recall that the symbol of the linear differential operator of order $k$ from $C^\infty(M,\R^n)$ to  $C^\infty(M,\R^m)$ is a homogeneous polynomial $\sigma$ of order $k$ on $T^*M$ with values in $\mathrm{Hom}(\R^n,\R^m)$. The characteristic variety of $\sigma$ is a subset $char(\sigma)\subset T^*M$ consisting of points $p\in T^*M$ for which the rank of $\sigma(p)\in \mathrm{Hom}(\R^n,\R^m)$ is not maximal. In the non-linear system we are studying   we shall assume that a solution to the system is given and then  consider a linearization of \eqref{Laxsys} around that solution. Then we will prove that the characteristic variety of the linearized equation is a field of ruled Cayley cubics in $T^*M$ that are dual varieties to the Cayley cubics defining the Cayley structure in $TM$ from Theorem \ref{propNormalform}.

As mentioned in the previous subsection, the system \eqref{Laxsys} is of order 3 in $F^1$ and $F^2$ and of order 2 in $E$. In order to have a system of non-mixed order one can replace $F^1$ and $F^2$ by the functions $P$, $Q$, $S$ and $T$ introduced in \eqref{defPQST}. Indeed, the vector fields $K_\lambda$ and $R_\lambda$ as well as the functions $m(\lambda)$ and $n(\lambda)$ defined in Theorem \ref{thm2} are expressed  in terms of $P,Q,R$ and $S$. That leads to a formula for $L_0$ and $L_1$ in terms of $P,Q,R$ and $S$ rather than $F^1$ and $F^2$. Then \eqref{Laxsys} becomes a system of second order for functions $P,Q,R,S$ and $E$. We are going to find the symbol of the linearisation of this system of second order PDEs. 

The equations in system \eqref{Laxsys} are long and inconvenient to express. Therefore we are not going to write down the linearization of the system directly but instead exploit its nature as a dispersionless Lax system. In fact, our method will be more general with potential applications to other Lax pairs.

Let us assume that a Lax pair $(L_0,L_1)$ on a 4-dimensional manifold is given and define
$$
L_0=K_\lambda+m(\lambda)\partial_\lambda,\qquad L_1=R_\lambda+n(\lambda)\partial_\lambda.
$$
However, now we assume that $K_\lambda$ and $R_\lambda$ depend linearly on unknown functions $f_1,\ldots,f_n$ for some $n\in\NN$. Hence, we can decompose
$$ 
K_\lambda=\hat K_\lambda+\sum_i f_iK^i_\lambda,\qquad R_\lambda=\hat R_\lambda+\sum_i f_iR^i_\lambda,
$$
where  $\hat K_\lambda$, $\hat R_\lambda$ as well as $K^i_\lambda$ and $R^i_\lambda$, $i=1,\ldots n$ are $\lambda$-dependent vector fields that do not involve the unknown functions $f_1,\ldots,f_n$. Furthermore, we assume that for any $\lambda$ the tuple $(K_\lambda, R_\lambda, \frac{\exd}{\exd\lambda}K_\lambda, \frac{\exd}{\exd\lambda}R_\lambda)$ constitutes a frame on $M$. Let $\alpha_\lambda$ and $\beta_\lambda$ be the 1-forms annihilating $K_\lambda$ and $R_\lambda$ and dual to $\frac{\exd}{\exd\lambda}K_\lambda$ and $\frac{\exd}{\exd\lambda}R_\lambda$, i.e. 
\[
\textstyle{\alpha_\lambda\left(\frac{\exd}{\exd\lambda}K_\lambda\right)=1,\quad \alpha_\lambda\left(\frac{\exd}{\exd\lambda}R_\lambda\right)=0, \quad \beta_\lambda\left(\frac{\exd}{\exd\lambda}K_\lambda\right)=0, \quad \beta_\lambda\left(\frac{\exd}{\exd\lambda}R_\lambda\right)=1.}
\]
Finally, for any $p\in T^*_xM$ and $V\in T_xM$ let $\langle p,V\rangle=p(V)$ denotes the standard pairing of a co-vector and a vector. With this notation we can state the following.
\begin{lemma}\label{lemmasym}
Let
\[
\sigma^i_\lambda(p)=\langle p,K_\lambda\rangle^2 \alpha_\lambda(R^i_\lambda) 
-\langle p, K_\lambda\rangle\langle p, R_\lambda\rangle(\alpha_\lambda(K^i_\lambda)-\beta_\lambda(R^i_\lambda)) - \langle p,R_\lambda\rangle^2 \beta_\lambda(K^i_\lambda)
\]
and consider a column $\lambda$-dependent vector $\sigma_\lambda(p)=(\sigma^1_\lambda(p),\ldots,\sigma^n_\lambda(p))^T$. Then the symbol $\sigma(p)$ of the linearization of the dispersionless Lax system $[L_0,L_1]\in\spn\{L_0,L_1\}$ along a solution $(f_1,\ldots,f_n)$ is given by the matrix whose columns are coefficients of consequent powers of $\lambda$ in $\sigma_\lambda(p)$, i.e. if $\sigma^i_\lambda(p)=\sum_j \lambda^j\sigma^i_j(p)$ for some polynomials $\sigma^i_j(p)$ then $\sigma(p)$ is the matrix $(\sigma^i_j(p))$.
\end{lemma}
\begin{proof}
Proceeding as in the proof of Theorem \ref{thm2} one find that  \eqref{Laxsys} implies 
\[
n(\lambda)=\alpha_\lambda([K_\lambda,R_\lambda]),\qquad \mathrm{and}
\qquad 
m(\lambda)=-\beta_\lambda([K_\lambda,R_\lambda]).
\]
Indeed $n(\lambda)$ and $-m(\lambda)$ are the coefficients of $\frac{\exd}{\exd\lambda}K_\lambda$ and $\frac{\exd}{\exd\lambda}R_\lambda$, respectively, in the expansion of $[L_0,L_1]\mod\partial_\lambda$ in the basis $(K_\lambda, R_\lambda, \frac{\exd}{\exd\lambda}K_\lambda, \frac{\exd}{\exd\lambda}R_\lambda)$. Hence, since $\alpha_\lambda([L_0,L_1])=0$ and $\beta_\lambda([L_0,L_1])=0$, one gets the formulae for $n(\lambda)$ and $m(\lambda)$ as claimed.

Now we are looking for the second order derivatives of functions $f_i$'s in the expansion of $[L_0,L_1]$ which can only appear in the coefficient for $\partial_\lambda$. Expanding the Lie bracket one sees that they are obtained from the expression
$$
K_\lambda(n(\lambda))-R_\lambda(m(\lambda)).
$$
Substituting $n$ and $m$ and replacing the derivatives of $f_i$ in direction of $K_\lambda$ and $R_\lambda$ by $\langle p, K_\lambda\rangle$ and $\langle p, R_\lambda\rangle$, respectively, one obtains the formula for $\sigma^i_\lambda$.
\end{proof}

Now we  make use of Lemma \ref{lemmasym} for  Cayley structures. For simplicity we limit ourselves to the symbol of a linearization along the trivial solution. This is sufficient for our purpose because the formulae in Lemma \ref{lemmasym} imply that for any solution the form of the symbol is the same up to a linear transformation sending $K_\lambda$ and $R_\lambda$ to the ones corresponding to the trivial solution.
\begin{theorem}\label{thmCharVar}
The symbol of the linearization of \eqref{Laxsys} along the trivial solution is given by the following $5\times 6$ matrix
\[
\sigma(p)= \def\arraystretch{1.15}
  \begin{pmatrix}
    p_3^2 & 2p_2p_3 & p_1p_3+p_2^2 & p_1p_2+\frac{1}{3}p_0p_3 & \frac{2}{9}p_1^2+\frac{1}{3}p_0p_2 & \frac{1}{9}p_0p_1 \\
    p_1p_3 & p_0p_3+p_1p_2 & p_0p_2+\frac{2}{3}p_1^2 & p_0p_1 & \frac{1}{3}p_0^2 & 0 \\
    p_1p_3 & p_0p_3+p_1p_2 & p_0p_2+\frac{1}{3}p_1^2 & \frac{1}{3}p_0p_1 & 0 & 0 \\
    p_1^2 & 2p_0p_1 & p_0^2 & 0 & 0 & 0\\
    0 & \frac{2}{3}p_1p_3 & \frac{2}{3}p_0p_3+\frac{2}{3}p_1p_2 & \frac{2}{3}p_0p_2+\frac{1}{3}p_1^2 & \frac{4}{9}p_0p_1 & \frac{1}{9}p_0^2
  \end{pmatrix}.
\]
$\sigma(p)$ has rank less or equal 4 if and only if
$$
\frac{1}{3}p_1^3+p_0^2p_3-p_0p_1p_2=0
$$
which is the dual variety of the ruled Cayley cubic $\frac{1}{3}(y^2)^3+y^0(y^3)^2-y^1y^2y^3=0$ which defines a Cayley structure with respect to  the coordinate system from Theorem \ref{propNormalform}.
\end{theorem}
\begin{proof}
The rows in the matrix correspond to unknown functions in the following order: $f_1=P$, $f_2=S$, $f_3=Q$, $f_4=T$ and $f_5=E$.
In order to apply Lemma \ref{lemmasym} we use \eqref{KRformula1} and \eqref{KRformula2} to find the following decomposition
\[ 
\hat K = \partial_1+\lambda\partial_0,\quad K^1=\lambda\partial_3,\quad K^2=0,\quad K^3=\lambda\partial_1,\quad K^4=0,\quad K^5=\partial_3+\lambda\partial_2,
\]
and
\[
\hat R = \partial_3+\lambda\partial_2+\frac{1}{3}\lambda^2\partial_1,\quad R^1=0,\quad R^2=\lambda\partial_3,\quad R^3=0,\quad R^4=\lambda\partial_1,\quad R^5=\frac{1}{3}\lambda^2\partial_3.
\]
Moreover, for the trivial solution one has $K_\lambda=\hat K_\lambda, R_\lambda=\hat R_\lambda$ and 
\[
\textstyle{\alpha_\lambda=\exd x^0-\lambda \exd x^1\frac{2}{3}\lambda^2\exd x^2-\frac{1}{3}\lambda^3\exd x^3,\qquad \beta_\lambda=\exd x^2-\lambda \exd x^3.}
\]
Now we are in a position to apply Lemma \ref{lemmasym}. For instance, we obtain
\[
\textstyle{\sigma^1=(p_3+\lambda p_2+\frac{1}{3}\lambda^2p_1)(\lambda^2p_3+\lambda^3p_2+\frac{2}{3}\lambda^4p_1+\frac{1}{3}\lambda^5p_0)}
\]
which, when decomposed into powers of $\lambda$, gives the first row in the matrix $\sigma(p)$. The other rows are found in the same way. Finally, by straightforward, but rather long, operations on the rows of $\sigma(p)$ one can verify that the rank of $\sigma(p)$ is not maximal if and only if $\frac{1}{3}p_1^3+p_0^2p_3-p_0p_1p_2=0$ holds.
\end{proof}

\begin{remark}
Recall that Theorem \ref{thm2} says there are 7 equations in the Lax  system \eqref{Laxsys}. At first this seems to be inconsistent with the size of the matrix $\sigma(p)$. However, it can be verified that one equation in the system is of lower order with respect to derivatives and thus it is not considered in the symbol and does not appear in the formulae from Lemma \ref{lemmasym}.

If one would like to consider a third order system for functions $F^1$ and $F^2$ then the corresponding symbol is of the form $\tilde\sigma(p)=(p_3\sigma^1(p)-p_1\sigma^2(p), p_1\sigma^3(p)-p_3\sigma^4(p))$. In this case one obtains a $2\times 5$ matrix (since the first column vanishes) which involves polynomials of degree 3 in $p_i$'s. This reflects the fact that among 7 equations in Theorem \ref{thm2} only 5 are of order 3 in $F^1$ and $F^2$. Furthermore, if $\tilde\sigma$ is extended by $\sigma^5$ (which corresponds to the function $E$), then the resulting $3\times 5$ matrix has rank less then or equal to 2 if and only if $\frac{1}{3}p_1^3+p_0^2p_3-p_0p_1p_2=0$ holds.
\end{remark}

\subsection{Special cases}
We shall exploit  Theorem \ref{thm2} and analyze the 2D path geometries appearing in Section \ref{sec:addit-geom-struct}. We will show how they can be interpreted in terms of the Lax pair defined above in \eqref{Laxeq}. The results have been explained in details in Section \ref{sec:3d-path-geometries} in a different way. Therefore we only present  very concise arguments below.

\subsubsection{The path geometry on the three parameter family of surfaces.} Recall that $J$ is equipped with a double fibration $\tilde\mu:=\pi\circ\mu^{-1}\colon J\to M$ and $\nu\colon J\to \cT$ (see diagram \eqref{eq:AllFib-FINAL}). The null-surfaces on $M$ (i.e. the surfaces from Definition \ref{def:half-flatness-}) are projections via $\tilde\mu$ of fibers of $\nu$. However, since locally $J\cong J^1(\R,\R^2)$ and $\cT\cong J^0(\R,\R^2)$, it follows that the fibers of $\nu$ have a natural affine structure inherited from the space of jets. Consequently the null-surfaces in $M$ have natural affine structure. It is straightforward to check that the null curves on null-surfaces are exactly the straight lines with respect to the affine structure. This means that the path geometry on the null surfaces is flat. This gives another proof of Proposition \ref{prop:flat-path-geom}. Note that this reasoning generalizes to other half-flat causal structures in dimension 4 as well as to higher dimensional structures  arising from ODEs via similar twistorial construction.

\subsubsection{Null foliation.}
The null foliation of Section \ref{sec:null-foliation-proj-str} is defined by imposing the integrability condition on the rank-2 distribution on $M$ defined as
$$
\mathcal{K}=\spn\{K_\lambda\ |\ \lambda\in\R\}=\spn\{K_1,K_2\}.
$$
Recall that $\mathbf{T}^X$ is of rank one for Cayley structures. It follows that the distribution $\mathcal{K}$ is invariantly defined by projecting the kernel of $\mathbf{T}^X$ along  1-dimensional fibers of $\tilde\mu\colon J\cong J^1(\R,\R^2)\to J/\X\cong M$. The integrability condition for $\mathcal{K}$ is written as $[K_1,K_2]\in\mathcal{K}$. It follows that 
\begin{equation}\label{eq-null-foliation}
K_1(E)=0,\qquad K_2(E)=K_1(P)-EK_1(Q).
\end{equation}
Moreover, the path geometry on $\mathcal{K}$ is defined by projecting the integral curves of $L_0$ from $J$ to $M$.
Recall that
\[
L_0=K_1+\lambda K_2+m(\lambda)\partial_\lambda
\]
where $m(\lambda)$ is of degree 4 in $\lambda$. However, one can check by direct computations that the coefficient of $\lambda^4$ in $m(\lambda)$ vanishes when \eqref{eq-null-foliation} holds. It follows that $m(\lambda)$ is of degree 3 in this case. This translates to the fact that the second order ODE for integral curves of $L_0$ projected to $M$ (i.e. the path geometry of Section \ref{sec:null-foliation-proj-str}) is a polynomial of degree 3 in first derivatives which means that the path geometry is a projective structure.

\subsubsection{Ultra-half-flat case.}
In order to explain the notion of ultra-half-flatness in terms of the Lax pair, recall that the null foliation is defined by integral leaves of the distribution $\mathcal{K}$ which is assumed to be integrable. The leaf space is denoted by $Q.$ 

The paths on $Q$ are the projection of the ruling planes of the Cayley cubics in $T M$. The projection is well-defined if the ruling planes extended by the null distribution $\mathcal{K}$ span an integrable distribution. More precisely, there exists a 2-parameter family of 3-dimensional submanifolds of $M$ with the property that for each point $x\in M$ and any ruling plane $\mathcal{P}_x\subseteq T_xM$ there is a submanifold $S$ in the family which passes through $x$ and satisfies $T_xS=\mathcal{P}_x+\mathcal{K}_x$. The family give rise to a  foliation of $J$ with 3-dimensional leaves. The tangent space to the leaves is spanned by $L_0$, $L_1$ and a lift of $K_1$ which we denote by $L_2$. We have
$$
L_2=K_1+(\lambda^2K_1(S)-\lambda^2EK_1(T)+\lambda R_\lambda(E))\partial_\lambda\qquad \mathrm{mod}\qquad L_0,L_1.
$$
In order to find $L_2$ one considers a rank 4 distribution $\mathcal{Q}$ on $J$ spanned by ${\tilde\mu}^{-1}_*\mathcal{K}$ and $L_1$. Since  $\mathcal{D}=\spn\{L_0,L_1\},$ the mapping $[L_1,.]\colon\mathcal{Q}/\mathcal{D}\to TJ/\mathcal{Q}$ defined by the Lie bracket with $L_1$ has one dimensional kernel which can be found explicitly by solving a linear system of algebraic equations. The kernel is spanned exactly by the vector field $L_2$. 

The integrability condition for the rank 3 distribution spanned by $L_0,L_1$ and $L_2$ gives the involutive system of Theorem \ref{thm:ultra-half-flatness}. 
 
\newpage
 \appendix
\setcounter{equation}{0}
\setcounter{subsection}{0}
 \setcounter{theorem}{0}

\section*{Appendix}
\renewcommand{\theequation}{A.\arabic{equation}}
\renewcommand{\thesection}{A}

In this section we will present the relations between the quantities $b_{ijk}$ in structure equations \eqref{eq:str-eqns-cayley} when $a_7 = 2a_6$ in the branching \eqref{eq:branching-c-StrEq} as described below
\begin{equation}
  \label{eq:relations-among-invs}
  \begin{gathered}
b_{104} = -a_4,\ \  b_{105} = -a_2, \ \ b_{113} = b_{223},\ \  b_{114} = -2a_3,\ \  b_{115} = a_1, \ \ b_{124} = -a_1,\ \  b_{125} = 0,\\  
b_{134} = 0,\ \ b_{135} = 0,\  \  b_{145} = 0,\ \   b_{203} = b_{102},\ \  b_{204} = -a_6, b_{205} = 0,\ \  b_{215} = -a_3+a_2, \\ 
 b_{224} = -2a_3,\ \  b_{225} = a_1,\ \  b_{235}=0,\ \  b_{245} = 0,\ \ b_{302}= -b_{402}+b_{201},\ \  b_{303} = -b_{403}+b_{101},\\
  b_{304} = 0,\ \  b_{305} = -5a_6,\ \  b_{312} = b_{202}-b_{101},\ \  b_{313}=b_{102},\ \  b_{314} = a_6,\ \  b_{315} = 0,\\ 
 b_{323} = b_{103},\ \  b_{324} = -a_4,\ \  b_{325} = -a_2,\ \  b_{334} = a_2,\ \  b_{335} = 0,\ \ b_{345}=0,\ \   b_{404} = 0,\\
  b_{405} = 3a_6,\ \  b_{413} = \textstyle{\frac{1}{2}}(b_{212}-b_{102}),\ \  b_{414} = \textstyle{-\frac{3}{2}}a_6,\ \ b_{415}=-\textstyle{\frac{1}{2}}a_5,\\
  b_{423} = -b_{103}-b_{112},\ \  b_{424} = a_5+a_4,\ \  b_{425} = -a_3+2a_2,\ \  b_{434} = -a_3-a_2, b_{435} = a_1,\ \   b_{445} = 0.
\end{gathered}
\end{equation}
For the branch $a_7 = \textstyle{\frac{3}{2}}a_6$,  in \eqref{eq:branching-c-StrEq}, the quantities that are  different from the ones in \eqref{eq:relations-among-invs} are
\[b_{305}=-\textstyle{\frac{7}{2}a_6}, \ \ b_{314}=\textstyle{\frac{1}{2}a_6},\ \ b_{405}=-\textstyle{2a_6},\ \ b_{414}=-a_6\]
Additionally, the expression of the $\mathfrak{sl}_4(\RR)$-valued Cartan connection  \eqref{eq:path-geom-cartan-conn-3D}, mentioned in Section \ref{sec:cartan3d}, satisfying \eqref{eq:sl4-conn-horiz-forms}, is as follows when $a_7 = 2a_6$ in the branching \eqref{eq:branching-c-StrEq}.
\begin{equation}
\label{eq:sl4-connection-forms}
\begin{aligned}
   \psi_0 &=  -(a_4 + a_5) \omega^0 + \phi_0,\\[3 pt]
  \psi_1 &= -(a_4 + a_5) \omega^0+(a_2 - a_3) \omega^1 + \phi_0 + \phi_1,\\[3 pt]
 \psi_2 &= - a_5 \omega^0+a_1 \omega^2  - 2 \phi_1,\\[3 pt]
\gamma_1&=\textstyle{\frac 12 a_6 \omega^0-(a_4+a_5)\omega^1+(-\frac{3}{4}a_2+\frac{1}{2}a_3) \omega^2},\\[3 pt]
 \gamma_2 &= \textstyle{2 a_6 \omega^1-a_4 \omega^2+a_2 \omega^3,}\\[3 pt]
\mu_3&=\textstyle{(- \frac 34 b_{101}-\frac{3}{4}b_{403}+\frac 14  b_{412}+\frac 14 a_{3;2}-\frac 14 a_{4;3}+\frac 12 a_1 a_2) \omega^0-b_{212} \omega^1}\\
     &+ \textstyle{(\frac{5}{4}b_{112}-\frac 14 a_{1;3}+\frac 34 b_{103}) \omega^2+b_{223} \omega^3-2 a_3 \theta^1+a_1 \theta^2},\\[3 pt]
\mu_2&=\textstyle{(\frac 32 a_3 a_4+\frac 32 a_5 a_3-\frac 14 a_1 a_6-\frac 94 a_4 a_2-\frac{45}{16} a_2 a_5-\frac{5}{16}b_{301}-\frac 78 b_{401}-\frac{7}{8} a_{4;1}-a_{5;1}) \omega^0}\\[3 pt]
&+\textstyle{b_{201} \omega^1+\frac 14(2 a_1 a_2+a_{3;2}-3 b_{101}+4b_{202}-3 b_{403}+b_{412}- a_{4;3}) \omega^2+b_{102} \omega^3-a_6 \theta^1},\\[3 pt]
\mu_1&=\textstyle{(-\frac 12 a_2a_3 +\frac 34 a_2^2-\frac 12 a_{6;3})\omega^0+\frac 14(2 a_1a_2- b_{101}+ 3a_{3;2} +  a_{4;3}- b_{403}+ 3 b_{412})\omega^1}\\[3 pt]
&+\textstyle{\frac 14(3b_{102}+ b_{212}- a_{2;3})\omega^2+\frac 14(b_{103}-b_{112}+a_{1;3})\omega^3+(-a_4-a_5)\theta^1+(-\frac 34 a_2+\frac 12 a_3)\theta^2}\\[3 pt]
\mu_0&=\textstyle{(4a_2a_6-\frac 52 a_3a_6+4a_4a_5+2a_4^2+2a_5^2-\frac 12 a_{6;1})\omega^0}\\[3 pt]
&+\textstyle{(-\frac 34a_1a_6-\frac 34 a_4a_2-\frac{15}{16}a_2a_5+\frac 12a_3a_4+\frac 12 a_5a_3+\frac{9}{16}b_{301}+\frac{3}{8}b_{401}+\frac 38a_{4;1})\omega^1}\\[3 pt]
&+\textstyle{\frac 14(-2a_2a_3+ 3a_2^2+b_{201}- 3b_{402}- a_{6;3}+ a_{4;2}+ a_{5;2})\omega^2}\\[3 pt]
&+\textstyle{\frac 14(-2 a_1a_2+ 3 b_{101}- b_{403}- b_{412}- a_{3;2}+ a_{4;3})\omega^3-\frac 52 a_6\theta^2}\\[3 pt]
\end{aligned}
\end{equation}
For the branch  $a_7 = \textstyle{\frac{3}{2}}a_6$  let us denote the  entries of  the $\mathfrak{sl}_4(\RR)$-valued Cartan connection \eqref{eq:path-geom-cartan-conn-3D} not included in \eqref{eq:sl4-conn-horiz-forms} by $\tilde\psi_0,\tilde\psi_1,\tilde\psi_2,\tilde\gamma_1,\tilde\gamma_2,\tilde\mu_3,\tilde\mu_2,\tilde\mu_1,\tilde\mu_0.$ Using the expressions given in \eqref{eq:sl4-connection-forms} one obtains
\begin{equation}
  \label{eq:sl4-conn-forms-c32}
  \begin{gathered}
    \tilde\psi_0=\psi_0,\qquad\tilde\psi_1=\psi_1,\qquad \tilde\psi_2=\psi_2,\qquad
\tilde\gamma_1=\gamma_1-\textstyle{\frac 18}a_6\omega^0,\qquad \tilde\gamma_2=\gamma_2-\textstyle{\frac 12}a_6\omega^1,\\
\tilde\mu_3=\mu_3,\qquad \tilde\mu_2=\mu_2+\textstyle{\frac{1}{16}}(a_1a_6-3a_2a_5-3b_{301}-6b_{401})\omega^0,\qquad \tilde\mu_1=\mu_1-\textstyle{\frac 18}a_{6;3}\omega^0,\\
\tilde\mu_0=\mu_0 +\textstyle{\frac 18(-20a_6 a_2+15a_3 a_6-5 a_{6;1}) \omega^0+\frac{1}{16}(3a_1 a_6- a_2 a_5- b_{301}-2b_{401}) \omega^1-\frac 18a_{6;3} \omega^2 +\frac{11}{8}a_6 \theta^2}
  \end{gathered}
\end{equation}
\newpage
\bibliographystyle{alpha}   
\bibliography{MergedFileBibliography}

\begin{thebibliography}{DDKR00}

\bibitem[AG96]{AG-Book}
M.~A. Akivis and V.~V. Goldberg.
\newblock {\em Conformal differential geometry and its generalizations}.
\newblock Pure and Applied Mathematics (New York). John Wiley \& Sons Inc., New
  York, 1996.
\newblock A Wiley-Interscience Publication.

\bibitem[ANN15]{ANN}
I.~Anderson, Z.~Nie, and P.~Nurowski.
\newblock Non-rigid parabolic geometries of {M}onge type.
\newblock {\em Adv. Math.}, 277:24--55, 2015.

\bibitem[Bas91]{Baston}
R.~J. Baston.
\newblock Almost {H}ermitian symmetric manifolds. {I}. {L}ocal twistor theory.
\newblock {\em Duke Math. J.}, 63(1):81--112, 1991.

\bibitem[BE91]{BE-Paraconformal}
T.~N. Bailey and M.~G. Eastwood.
\newblock Complex paraconformal manifolds---their differential geometry and
  twistor theory.
\newblock {\em Forum Math.}, 3(1):61--103, 1991.

\bibitem[Bel04]{Belgun}
F.~A. Belgun.
\newblock On the {W}eyl tensor of a self-dual complex 4-manifold.
\newblock {\em Trans. Amer. Math. Soc.}, 356(3):853--880, 2004.

\bibitem[BGG03]{BGG}
R.~L. Bryant, P.~Griffiths, and D.~Grossman.
\newblock {\em Exterior differential systems and {E}uler-{L}agrange partial
  differential equations}.
\newblock Chicago Lectures in Mathematics. University of Chicago Press,
  Chicago, IL, 2003.

\bibitem[BGH95]{BGH}
R.~L. Bryant, P.~Griffiths, and L.~Hsu.
\newblock Toward a geometry of differential equations.
\newblock In {\em Geometry, topology, \& physics}, Conf. Proc. Lecture Notes
  Geom. Topology, IV, pages 1--76. Int. Press, Cambridge, MA, 1995.

\bibitem[BNS90]{BNS-affine}
N.~Bokan, K.~Nomizu, and U.~Simon.
\newblock Affine hypersurfaces with parallel cubic forms.
\newblock {\em Tohoku Math. J. (2)}, 42(1):101--108, 1990.

\bibitem[Bry91]{B-exotic}
R.~L. Bryant.
\newblock Two exotic holonomies in dimension four, path geometries, and twistor
  theory.
\newblock In {\em Complex geometry and {L}ie theory ({S}undance, {UT}, 1989)},
  volume~53 of {\em Proc. Sympos. Pure Math.}, pages 33--88. Amer. Math. Soc.,
  Providence, RI, 1991.

\bibitem[Bry97]{B-Finsler}
R.~L. Bryant.
\newblock Projectively flat {F}insler {$2$}-spheres of constant curvature.
\newblock {\em Selecta Math. (N.S.)}, 3(2):161--203, 1997.

\bibitem[Bry06]{Bryant-36}
R.~L. Bryant.
\newblock Conformal geometry and 3-plane fields on 6-manifolds.
\newblock {\em in Developments of Cartan Geometry and Related Mathematical
  Problems, RIMS Symposium Proceedings}, 1502:1--15, 2006.

\bibitem[Bry14]{Bryant-Notes}
R.~L Bryant.
\newblock Notes on exterior differential systems.
\newblock {\em arXiv preprint}, 2014.

\bibitem[Cal14]{Calderbank}
D.~M.~J. Calderbank.
\newblock Selfdual 4-manifolds, projective surfaces, and the {D}unajski-{W}est
  construction.
\newblock {\em SIGMA Symmetry Integrability Geom. Methods Appl.}, 10:Paper 035,
  18, 2014.

\bibitem[Car23]{Cartan-Conformal1}
E.~Cartan.
\newblock Les espaces {\`a} connection conforme.
\newblock {\em Annales de la Soci{\'e}t{\'e} Polonaise de Math{\'e}matique T. 2
  (1923)}, 1923.

\bibitem[Car24]{Cartan-proj}
E.~Cartan.
\newblock Sur les vari\'{e}t\'{e}s \`a connexion projective.
\newblock {\em Bull. Soc. Math. France}, 52:205--241, 1924.

\bibitem[Car37]{Cartan-Conformal2}
E.~Cartan.
\newblock L'extension du calcul tensoriel aux g\'{e}om\'{e}tries non-affines.
\newblock {\em Ann. of Math. (2)}, 38(1):1--13, 1937.

\bibitem[CDT13]{CDT}
S.~Casey, M.~Dunajski, and P.~Tod.
\newblock Twistor geometry of a pair of second order {ODE}s.
\newblock {\em Comm. Math. Phys.}, 321(3):681--701, 2013.

\bibitem[Che43]{Chern}
S.-S. Chern.
\newblock A generalization of the projective geometry of linear spaces.
\newblock {\em Proc. Nat. Acad. Sci. U. S. A.}, 29:38--43, 1943.

\bibitem[CS07]{CS-Segre}
M.~Crampin and D.~J. Saunders.
\newblock Path geometries and almost {G}rassmann structures.
\newblock 48:225--261, 2007.

\bibitem[{\v{C}}S09]{CS-Parabolic}
A.~{\v{C}}ap and J.~Slov{\'a}k.
\newblock {\em Parabolic geometries. {I}}, volume 154 of {\em Mathematical
  Surveys and Monographs}.
\newblock American Mathematical Society, Providence, RI, 2009.
\newblock Background and general theory.

\bibitem[DDKR00]{DDKR-HomSurf}
F.~Dillen, B.~Doubrov, B.~Komrakov, and M.~Rabinovich.
\newblock Homogeneous surfaces in the three-dimensional projective geometry.
\newblock {\em J. Math. Soc. Japan}, 52(1):199--230, 2000.

\bibitem[DFK15]{DFK}
M.~Dunajski, E.~V. Ferapontov, and B.~Kruglikov.
\newblock On the {E}instein-{W}eyl and conformal self-duality equations.
\newblock {\em J. Math. Phys.}, 56(8):083501, 10, 2015.

\bibitem[DK14]{DunKry}
M.~Dunajski and W.~Kry\'{n}ski.
\newblock {E}instein–{W}eyl geometry, dispersionless {H}irota equation and
  {V}eronese webs.
\newblock {\em Math. Proc. Cambridge Phil. Soc.}, 157(1):139--150, 2014.

\bibitem[DKM99]{DKM}
B.~Doubrov, B.~Komrakov, and T.~Morimoto.
\newblock Equivalence of holonomic differential equations.
\newblock {\em Lobachevskii J. Math.}, 3:39--71, 1999.
\newblock Towards 100 years after Sophus Lie (Kazan, 1998).

\bibitem[Dol12]{Dolgachev}
I.~V. Dolgachev.
\newblock {\em Classical algebraic geometry}.
\newblock Cambridge University Press, Cambridge, 2012.
\newblock A modern view.

\bibitem[Dou08]{Doubrov}
B.~Doubrov.
\newblock Generalized {W}ilczynski invariants for non-linear ordinary
  differential equations.
\newblock In {\em Symmetries and overdetermined systems of partial differential
  equations}, volume 144 of {\em IMA Vol. Math. Appl.}, pages 25--40. Springer,
  New York, 2008.

\bibitem[DT06]{DT}
M.~Dunajski and P.~Tod.
\newblock Paraconformal geometry of {$n$}th-order {ODE}s, and exotic holonomy
  in dimension four.
\newblock {\em J. Geom. Phys.}, 56(9):1790--1809, 2006.

\bibitem[DV98]{DV-Cayley}
F.~Dillen and L.~Vrancken.
\newblock Hypersurfaces with parallel difference tensor.
\newblock {\em Japan. J. Math. (N.S.)}, 24(1):43--60, 1998.

\bibitem[DW07]{DW1}
M.~Dunajski and S.~West.
\newblock Anti-self-dual conformal structures with null {K}illing vectors from
  projective structures.
\newblock {\em Comm. Math. Phys.}, 272(1):85--118, 2007.

\bibitem[DW08]{DW2}
M.~Dunajski and S.~West.
\newblock Anti-self-dual conformal structures in neutral signature.
\newblock pages 113--148, 2008.

\bibitem[EE06]{EE-Cayley}
M.~Eastwood and V.~Ezhov.
\newblock Cayley hypersurfaces.
\newblock {\em Tr. Mat. Inst. Steklova}, 253(Kompleks. Anal. i
  Prilozh.):241--244, 2006.

\bibitem[EFV20]{Cayley-url}
A.~Esculier, R.~Ferr\'eol, and L.~G. Vidiani.
\newblock Ruled cubic.
\newblock
  \url{https://www.mathcurve.com/surfaces.gb/cubic/cubique_reglee.shtml}, 2017
  (accessed March 11, 2020).

\bibitem[Ego51]{Egorov}
I.~P. Egorov.
\newblock Collineations of projectively connected spaces.
\newblock {\em Doklady Akad. Nauk SSSR (N.S.)}, 80:709--712, 1951.

\bibitem[Eis97]{Eisenhart}
L.~P. Eisenhart.
\newblock {\em Riemannian geometry}.
\newblock Princeton Landmarks in Mathematics. Princeton University Press,
  Princeton, NJ, 1997.
\newblock Eighth printing, Princeton Paperbacks.

\bibitem[Fel95]{Fels}
M.~E. Fels.
\newblock The equivalence problem for systems of second-order ordinary
  differential equations.
\newblock {\em Proc. London Math. Soc. (3)}, 71(1):221--240, 1995.

\bibitem[FK14]{FK}
E.~V. Ferapontov and B.~S. Kruglikov.
\newblock Dispersionless integrable systems in 3{D} and {E}instein-{W}eyl
  geometry.
\newblock {\em J. Differential Geom.}, 97(2):215--254, 2014.

\bibitem[FK16]{FK-GL2}
E.~Ferapontov and B.~Kruglikov.
\newblock Dispersionless integrable hierarchies and ${GL(2,\mathbb R)}$
  geometry.
\newblock {\em Mathematical Proceedings of the Cambridge Philosophical
  Society}, pages 1--26, 2016.

\bibitem[Fur95]{FurtherAdvancesII}
Further advances in twistor theory. {V}ol. {II}.
\newblock 232:xii+274, 1995.
\newblock Integrable systems, conformal geometry and gravitation.

\bibitem[Gar89]{Gardner}
R.~B. Gardner.
\newblock {\em The method of equivalence and its applications}, volume~58 of
  {\em CBMS-NSF Regional Conference Series in Applied Mathematics}.
\newblock Society for Industrial and Applied Mathematics (SIAM), Philadelphia,
  PA, 1989.

\bibitem[GN10]{GN-ODE}
M.~Godlinski and P.~Nurowski.
\newblock $\mathrm{GL}(2,\mathbf{R})$ geometry of {ODE}'s.
\newblock {\em J. Geom. Phys.}, 60(6-8):991--1027, 2010.

\bibitem[Gro00a]{Grossman-Thesis}
D.~A. Grossman.
\newblock {\em Path geometries and second-order ordinary differential
  equations}.
\newblock ProQuest LLC, Ann Arbor, MI, 2000.
\newblock Thesis (Ph.D.)--Princeton University.

\bibitem[Gro00b]{Grossman}
D.~A. Grossman.
\newblock Torsion-free path geometries and integrable second order {ODE}
  systems.
\newblock {\em Selecta Math. (N.S.)}, 6(4):399--442, 2000.

\bibitem[Hen86]{Henrici}
P.~Henrici.
\newblock {\em Applied and computational complex analysis. {V}ol. 3}.
\newblock Pure and Applied Mathematics (New York). John Wiley \& Sons, Inc.,
  New York, 1986.
\newblock Discrete Fourier analysis---Cauchy integrals---construction of
  conformal maps---univalent functions, A Wiley-Interscience Publication.

\bibitem[HM04]{HM-TangentMap}
J.-M. Hwang and N.~Mok.
\newblock Birationality of the tangent map for minimal rational curves.
\newblock {\em Asian J. Math.}, 8(1):51--63, 2004.

\bibitem[HS10]{HS1}
J.~Holland and G.~Sparling.
\newblock {Causal geometries and third-order ordinary differential equations}.
\newblock {\em Arxiv preprint arXiv:1001.0202}, 2010.

\bibitem[HS11]{HS2}
J.~Holland and G.~Sparling.
\newblock Causal geometries, null geodesics, and gravity.
\newblock {\em arXiv preprint arXiv:1106.5254}, 2011.

\bibitem[Hwa10]{Hwang-Cone}
J.-M. Hwang.
\newblock Equivalence problem for minimal rational curves with isotrivial
  varieties of minimal rational tangents.
\newblock {\em Ann. Sci. \'Ec. Norm. Sup\'er. (4)}, 43(4):607--620, 2010.

\bibitem[Hwa12]{Hwang-Survey}
J.-M. Hwang.
\newblock Geometry of varieties of minimal rational tangents.
\newblock In {\em Current developments in algebraic geometry}, volume~59 of
  {\em Math. Sci. Res. Inst. Publ.}, pages 197--226. Cambridge Univ. Press,
  Cambridge, 2012.

\bibitem[Hwa13]{Hwang-Causal}
J.-M. Hwang.
\newblock Varieties of minimal rational tangents of codimension 1.
\newblock {\em Ann. Sci. \'Ec. Norm. Sup\'er. (4)}, 46(4):629--649 (2013),
  2013.

\bibitem[IL03]{IL}
T.~A. Ivey and J.~M. Landsberg.
\newblock {\em Cartan for beginners: differential geometry via moving frames
  and exterior differential systems}, volume~61 of {\em Graduate Studies in
  Mathematics}.
\newblock American Mathematical Society, Providence, RI, 2003.

\bibitem[JS20]{JS}
M.~A. Javaloyes and M.~S{\'a}nchez.
\newblock On the definition and examples of cones and {F}insler spacetimes.
\newblock {\em Revista de la Real Academia de Ciencias Exactas, Fisicas y
  Naturales. Serie A. Matem{\'a}ticas}, 114(1):30, 2020.

\bibitem[KM19]{KM}
W.~Kry\'nski and T.~Mettler.
\newblock $\mathrm{GL}(2)$-structures in dimension four, ${H}$-flatness and
  integrability.
\newblock {\em Comm. Anal. Geom.}, 27(8):1851--1868, 2019.

\bibitem[Kry07]{K2}
W.~Kry{\'n}ski.
\newblock On contact equivalence of systems of ordinary differential equations.
\newblock {\em arXiv:0712.1455}, 2007.

\bibitem[Kry14]{K-webs}
W.~Kry\'{n}ski.
\newblock Webs and projective structures on a plane.
\newblock {\em Differential Geom. Appl.}, 37:133--140, 2014.

\bibitem[Kry16]{K3}
W.~Kry\'{n}ski.
\newblock Paraconformal structures, ordinary differential equations and totally
  geodesic manifolds.
\newblock {\em J. Geom. Phys.}, 103:1--19, 2016.

\bibitem[KT17]{KT}
B.~Kruglikov and D.~The.
\newblock The gap phenomenon in parabolic geometries.
\newblock {\em J. Reine Angew. Math.}, 723:153--215, 2017.

\bibitem[LM07]{LM-zollfrei}
C.~R. LeBrun and L.~J. Mason.
\newblock Nonlinear gravitons, null geodesics, and holomorphic disks.
\newblock {\em Duke Math. J.}, 136(2):205--273, 2007.

\bibitem[Mak16]{Omid-Thesis}
O.~Makhmali.
\newblock {\em Differential Geometric Aspects of Causal Structures}.
\newblock PhD thesis, McGill University, 2016.

\bibitem[Mak18]{Omid-Sigma}
O.~Makhmali.
\newblock Differential geometric aspects of causal structures.
\newblock {\em SIGMA}, 14(080), 2018.

\bibitem[Met13]{Mettler-Segre}
T.~Mettler.
\newblock Reduction of {$\beta$}-integrable 2-{S}egre structures.
\newblock {\em Comm. Anal. Geom.}, 21(2):331--353, 2013.

\bibitem[Mok08]{Mok}
N.~Mok.
\newblock Geometric structures on uniruled projective manifolds defined by
  their varieties of minimal rational tangents.
\newblock Number 322, pages 151--205. 2008.
\newblock G\'{e}om\'{e}trie diff\'{e}rentielle, physique math\'{e}matique,
  math\'{e}matiques et soci\'{e}t\'{e}. II.

\bibitem[MW96]{MW-Book}
L.~J. Mason and N.~M.~J. Woodhouse.
\newblock {\em Integrability, self-duality, and twistor theory}, volume~15 of
  {\em London Mathematical Society Monographs. New Series}.
\newblock The Clarendon Press, Oxford University Press, New York, 1996.
\newblock Oxford Science Publications.

\bibitem[NP89]{NP-Cayley}
K.~Nomizu and U.~Pinkall.
\newblock Cayley surfaces in affine differential geometry.
\newblock {\em Tohoku Math. J. (2)}, 41(4):589--596, 1989.

\bibitem[NS94]{NS-Book}
K.~Nomizu and T.~Sasaki.
\newblock {\em Affine differential geometry}, volume 111 of {\em Cambridge
  Tracts in Mathematics}.
\newblock Cambridge University Press, Cambridge, 1994.
\newblock Geometry of affine immersions.

\bibitem[Nur05]{Nurowski-G2}
P.~Nurowski.
\newblock Differential equations and conformal structures.
\newblock {\em J. Geom. Phys.}, 55(1):19--49, 2005.

\bibitem[Olv95]{Olver}
P.~J. Olver.
\newblock {\em Equivalence, invariants, and symmetry}.
\newblock Cambridge University Press, Cambridge, 1995.

\bibitem[Pen76]{Penrose}
R.~Penrose.
\newblock Nonlinear gravitons and curved twistor theory.
\newblock {\em General Relativity and Gravitation}, 7(1):31--52, 1976.
\newblock The riddle of gravitation--on the occasion of the 60th birthday of
  Peter G. Bergmann (Proc. Conf., Syracuse Univ., Syracuse, N. Y., 1975).

\bibitem[PLR01]{Respondek}
W.~Pasillas-L\'{e}pine and W.~Respondek.
\newblock Contact systems and corank one involutive subdistributions.
\newblock {\em Acta Appl. Math.}, 69(2):105--128, 2001.

\bibitem[Sas99]{Sasaki-Book}
T.~Sasaki.
\newblock {\em Projective differential geometry and linear homogeneous
  differential equations}.
\newblock Department of Mathematics, Kobe University, 1999.

\bibitem[Sas06]{Sasaki-Notes}
T.~Sasaki.
\newblock Line congruence and transformation of projective surfaces.
\newblock {\em Kyushu J. Math.}, 60(1):101--243, 2006.

\bibitem[Seg76]{Segal-Book}
I.~E. Segal.
\newblock {\em Mathematical cosmology and extragalactic astronomy}.
\newblock Academic Press [Harcourt Brace Jovanovich, Publishers], New
  York-London, 1976.
\newblock Pure and Applied Mathematics, Vol. 68.

\bibitem[Sha97]{Sharpe}
R.~W. Sharpe.
\newblock {\em Differential geometry}, volume 166 of {\em Graduate Texts in
  Mathematics}.
\newblock Springer-Verlag, New York, 1997.
\newblock Cartan's generalization of Klein's Erlangen program, With a foreword
  by S. S. Chern.

\bibitem[Smi10]{Smith}
A.~D. Smith.
\newblock Integrable {${\rm GL}(2)$} geometry and hydrodynamic partial
  differential equations.
\newblock {\em Comm. Anal. Geom.}, 18(4):743--790, 2010.

\end{thebibliography}
\end{document}